\documentclass[a4paper,10pt]{elsarticle}
\usepackage{amsmath,amssymb,enumerate,amsfonts}
\usepackage[all]{xy}
\usepackage{color}
\usepackage{latexsym,amsbsy,amscd,amsthm}
\usepackage{gastex}
\usepackage{hhline}
\usepackage{array}

\usepackage{a4wide}

\newtheorem{theorem}{Theorem}[section]
\newtheorem{lemma}[theorem]{Lemma}
\newtheorem{corollary}[theorem]{Corollary}
\newtheorem{proposition}[theorem]{Proposition}

\newtheorem{thesis}[theorem]{Thesis}
\theoremstyle{definition}
\newtheorem{definition}[theorem]{Definition}
\newtheorem{observation}[theorem]{Observation}
\newtheorem{example}[theorem]{Example}

\newtheorem{remark}[theorem]{Remark}
\newtheorem*{question}{Question}
\numberwithin{equation}{section}

\newcommand {\N}{\mathbb{N}} 
\newcommand {\Z}{\mathbb{Z}} 
\newcommand {\R}{\mathbb{R}} 

\newcommand{\T}{\mathcal{T}}

\newcommand{\blank}{\mbox{b}}

\newcommand{\A}{\mathcal{A}}

\newcommand{\CC}{\mathcal{C}}

\newcommand{\FF}{\mathcal{F}}
\newcommand{\GG}{\mathcal{G}}

\newcommand{\MM}{\mathcal{M}}

\newcommand{\PP}{\mathcal{P}}

\newcommand\impliess{\buildrel * \over \implies}
\newcommand\then{\!\!\implies\!\!}
\newcommand\thens{\!\!\impliess\!\!}

\DeclareMathOperator{\dist}{dist}

\DeclareMathOperator{\WP}{WP}
\DeclareMathOperator{\PSL}{PSL}
\DeclareMathOperator{\Dehn}{Dehn}
\DeclareMathOperator{\Area}{Area}
\DeclareMathOperator{\Inf}{Inf}

\journal{European Journal of Combinatorics}

\begin{document}
\title{Groups, Graphs, Languages, Automata, Games and Second-order Monadic Logic}

\author{Tullio Ceccherini-Silberstein}
\address{Dipartimento di Ingegneria, Universit\`a del Sannio, C.so
Garibaldi 107, 82100 Benevento, Italy}
\ead{tceccher@mat.uniroma1.it}
\author{Michel Coornaert}
\address{Institut de Recherche Math\'ematique Avanc\'ee,
UMR 7501, Universit\'e  de Strasbourg et CNRS,
7 rue Ren\'e-Descartes,67000 Strasbourg, France}
\ead{coornaert@math.unistra.fr}
\author{Francesca Fiorenzi}
\address{Laboratoire de Recherche en Informatique,
Universit\'e Paris-Sud 11,
91405 Orsay Cedex, France}
\ead{fiorenzi@lri.fr}
\author{Paul E. Schupp}
\address{Department of Mathematics,
University of Illinois at Urbana-Champaign,
1409 W. Green Street (MC-382)
Urbana, Illinois 61801-2975, USA}
\ead{schupp@math.uiuc.edu}

\begin{abstract}
In this paper we survey some surprising connections between group theory, the theory of automata and formal languages,
the theory of ends, infinite games of perfect information, and monadic second-order logic.
\end{abstract}

\maketitle

\begin{center}
\textit{Dedicated to Toni Mach\`\i \ on the occasion of his 70th
birthday}
\end{center}
\tableofcontents

\section{Introduction}
In this survey we discuss some interleaved strands of ideas connecting the items in the title.
We do not, of course, develop  all the connections between groups and automata. In particular, we do not consider either \emph{automatic groups}
(see, for instance, the monograph \cite{epstein} by Epstein, Cannon, Hold, Levy, Paterson, and Thurston) or \emph{automata groups}, also called \emph{self-similar groups} (including the well known Grigorchuk group of intermediate growth \cite{grigorchuk, delaharpe1}: see, for instance, \cite{GNS, BGN, BGS} and the monograph \cite{nekrashevych} by Nekrashevych).

A finitely generated group can be described by a
presentation $G = \langle X; R \rangle$ in terms of generators and defining
relators. In this case, the \emph{group alphabet} is $\Sigma = X \cup X^{-1}$.  Anisimov~\cite{anisimov} introduced the
fruitful point of view of considering the Word Problem of $G = \langle X; R \rangle$ as  the  formal language
$\WP(G:X;R) = \{w \in \Sigma^*: w = 1_G\}$. Although the Word Problem is generally  a very complicated
set, Anisimov asked what one could say about the group $G$ if $\WP(G:X;R)$ is a regular or context-free language in
the usual sense of formal language theory.  He showed that a finitely generated group has regular Word Problem if and
only if the group is finite. An important class of groups is the class of virtually free groups, that is, groups having
a free subgroup of finite index.  Muller and Schupp~\cite{MS1} showed that a finitely generated group has context-free
Word Problem if and only if the group is virtually free.

The basic geometric object associated with a finitely generated group $G = \langle X;R \rangle$, its Cayley graph
$\Gamma(G:X;R)$, was already defined by Cayley~\cite{cayley} in 1878. Intuitively, an \emph{end} (a notion due to
Hopf~\cite{hopf} and  Freudenthal~\cite{freudenthal}) of a locally finite graph is a way to go to infinity in the graph.  The number of ends of a connected graph $\Gamma$ with origin $v_0$ is the limit, as $n$ goes to infinity,
of the number of infinite connected components of $\Gamma \setminus \Gamma_n$, where the $n$-ball $\Gamma_n$ consists
of all vertices and edges on paths of length less than or equal to $n$ starting at $v_0$. The number of ends of a
finitely generated group is the number of ends of its Cayley graph. (It is not obvious, but true, that this number
depends only on the group and not on the particular presentation chosen.)  The proof of the characterization of groups
with context-free Word Problem depends heavily on the Stallings structure theorem~\cite{stallings}, which shows that
finitely generated groups with more than one end must have a particular algebraic structure.

It turns out that the connection between ends and context-freeness is much deeper than just  the  case of groups.
It is well-known ~\cite{chiswell, harrison, HU} that a formal language is context-free if and only if it is the
language accepted by some pushdown automaton.  The concept of a \emph{finitely generated graph} gives a common
framework in which one can discuss both Cayley graphs of finitely generated groups and complete transition graphs of
various kinds of automata, in particular the complete transition graph of a pushdown automaton.

Instead of considering
the number of ends of a finitely generated graph $\Gamma$, one can consider the number $c(\Gamma)$ of labelled
graph isomorphism classes of connected components of $\Gamma \setminus \Gamma_n$ over \emph{all} components and
all $n \ge 1$. Say that $\Gamma$
has \emph{finitary end-structure} if $c(\Gamma) < \infty$.  Muller and Schupp~\cite{MS2} proved that a finitely generated
graph has finitary end-structure if and only if $\Gamma$ is isomorphic to the complete transition graph $\Gamma(M)$ of
some pushdown automaton $M$.

One of the most powerful positive results about decision problems in logic is Rabin's theorem~\cite{rabin} that
the second--order monadic theory of the rooted infinite binary  tree $T_2$ is decidable.  This theory,  $S2S$, is the
\emph{theory of two successor functions}, as we now explain.  We consider the infinite binary tree as the rooted tree
with root $v_0$ and right successor edges labelled by $1$ and left successor edges labelled by $0$. The second--order
monadic logic of $T_2$ has variables ranging over arbitrary sets of vertices.  We have two set-valued successor functions:
if $S$ is a set of vertices and $a \in \{0,1\}$ then  $Sa = \{va : v \in S\}$. There is also the
relation symbol $\subseteq$ for set inclusion and a constant symbol $v_0$ for the origin. There are the usual
quantifiers $\forall, \exists$ and the Boolean connectives  $\land$ (and), $\lor$ (or),  and  $\neg$ (negation).
Some formulations include individual variables for single vertices, but sets with a  single element are definable,
 as is equality.  The great power of this language is that one can quantify over arbitrary sets of vertices.

   The characterization of graphs with finitary end structure shows that such graphs are ``very treelike''. Indeed,
such a graph $\Gamma$ contains a regular subtree of finite index, in the sense that there is a subtree $T$ defined
by a finite automaton and a fixed bound $D \geq 0$ such that every vertex in $\Gamma$ is within distance $D$ of
some vertex in the subtree $T$. From this fact, it is possible to reduce questions about the monadic theory of
$\Gamma$ to questions about the monadic theory of the tree $T$. It then follows from Rabin's theorem that
the monadic theory of the complete transition graph of any pushdown automaton is decidable.   In particular,
if $G = \langle X; R \rangle$ is any finitely generated presentation of a virtually free finitely generated group
then the monadic second-order theory of its Cayley graph $\Gamma(G:X;R)$ is decidable. There are finitely generated
graphs which do not have finitary end structure but whose monadic theories are decidable.   However,
 Kuske and  Lohrey~\cite{KL} have recently proved that if the
monadic theory of the Cayley graph of a finitely generated group is decidable then the group must be virtually free.


    There is  an interesting application of the decidability of the monadic second-order theory of Cayley graphs of
context-free groups to  the theory of cellular automata on groups.  The following definition is actually a straightforward generalization of von Neumann's  concept~\cite{neumann} of cellular automata on the grid on integer
lattice points in the plane, that is, the Cayley graph of
$\mathbb{Z}^2$.  Let $G$ be a group and $\Sigma$ a finite set and denote by $\Sigma^G$ the set of all maps
$\alpha \colon G \to \Sigma$.  Equip $\Sigma^G$ with the action of $G$ defined by
$$g(\alpha)(h) = \alpha(g^{-1}h) \mbox{ \ for all \ } \alpha \in \Sigma^G \mbox{ \ and \ }  g,h \in G.$$
Then one says that
a map $\CC \colon \Sigma^G \to \Sigma^G$ is a \emph{cellular automaton}
provided there exists a \emph{finite} subset $M \subset G$ and a map $\mu \colon \Sigma^M \to \Sigma$ such that
\begin{equation}
\label{e:ca1}
\CC(\alpha)(g) = \mu((g^{-1}\alpha)\vert_M)
\end{equation}
for all $\alpha \in \Sigma^G$ and $g \in G$, and where $(\cdot)\vert_M$ denotes the restriction to $M$. One is often
interested in determining whether or not a cellular automaton is surjective (respectively,  injective, bijective).
In particular, the following decision problem naturally arises: given a finite subset $M \subset G$
and a map $\mu \colon \Sigma^M \to \Sigma$, is the  associated cellular automaton $\CC \colon \Sigma^G \to \Sigma^G$
defined in \eqref{e:ca1}  surjective (respectively injective,  bijective) or not?   Amoroso and  Patt~\cite{amoroso}
proved in 1972 that if $G = \Z$ the above problem is decidable. If follows  from the decidability of the monadic
second-order theory of Cayley graphs of context-free groups  that the problem for cellular automata defined over
virtually-free groups is decidable. On the other hand,  Kari~\cite{kari1,kari2,kari3} proved that if $G = \Z^d$, $d \geq 2$,
this problem is undecidable. His  proof is based on  Berger's~\cite{berger} undecidability result for the
\emph{Domino Problem for Wang tiles}.

   In 1960  B\"uchi~\cite{buchi} proved that the monadic theory of $\mathbb N$ with one successor function, $S1S$, is
decidable by introducing finite automata working on infinite words.  Monadic sentences are too complicated to deal with
directly and the idea is to effectively associate with each monadic sentence $\phi$ a finite automaton $\A_{\phi}$ such that
$\phi$ is true if and only if the language $L(\A_{\phi})$ accepted by $\A_{\phi}$ is nonempty.
Of course, one must carefully define what it means for an automaton to accept an infinite word. Rabin used automata working on
infinite trees to establish  a similar correspondence between sentences of $S2S$ and the Emptiness Problem for tree automata.

   The theory of automata working on infinite inputs is thus crucial to studying monadic theories, but proving theorems
about such automata is difficult.  The best way to understand such automata is in terms of infinite games of perfect information
as introduced by  Gale and  Stewart~\cite{GS}. Let $\Sigma$ be a finite alphabet and  let $\Sigma^\N$ denote the set
of all  infinite words $w = a_1a_2 \cdots a_n \cdots$ over $\Sigma$ (all the infinite words which we consider
are infinite to the right).  Let $\mathcal W$ be a subset of  $\Sigma^\N$.
We consider the following game between Player I and Player II: Player I chooses a letter $\sigma_1 \in \Sigma$ and  Player II then
chooses a letter  $\sigma_2 \in \Sigma$.  Continuing indefinitely, at step $n$ Player I chooses a letter  $\sigma_{2n-1} \in \Sigma$
and Player II then chooses a letter  $\sigma_{2n} \in \Sigma$.  This sequence of choices defines an infinite word $w \in \Sigma^{\N}$.
Player I wins the game  if $w \in \mathcal W$ and Player II wins otherwise.
The basic question about such games is whether or not one of the players has a winning strategy, that is,
a map  $\phi: \Sigma^* \to \Sigma$ such that when a finite word $u$ has already been played, the player using the strategy then plays
$\phi(u) \in \Sigma$ and always wins.  Using the Axiom of Choice, it is possible to construct winning sets such that
neither player has a winning strategy, but this cannot happen if the set $\mathcal W$ is not ``too complicated''.
An important theorem of  Martin~\cite{martin1, martin2} shows that if the set $\mathcal W$ is a Borel set
then  one of the two players must have a winning strategy.

   To apply infinite games to automata, given an automaton $M$ one defines the \emph{acceptance game}  $\mathcal G (M,w)$ for $M$
on  an infinite input $w \in \Sigma^\N$.   The first player wins if $M$ accepts $w$
while the second player wins if $M$ rejects.  In the case of automata, the winning condition of the acceptance game is
at the second level of the Borel hierarchy so one of the players has a winning strategy. This  essentially proves closure of
under complementation  of regular languages in $\Sigma^\N$.   The situation is similar for  automata on the binary tree.
The celebrated ``Forgetful Determinacy Theorem'' of  Gurevich and  Harrington~\cite{GH} states that a
fixed finite amount of memory, \emph{the later appearance record}, is all that a winning strategy needs to take into account.

The paper is organized as follows. In Section~\ref{sec:LGA} we review the notions of regular, context-free, and computably  enumerable
 languages together with the parallel notions of grammars and their associated classes of automata: finite-state automata,
pushdown automata, and  Turing machines.   Section~\ref{sec:CG} is devoted to presentations of finitely generated groups  and their
associated Cayley graphs. We  consider the  Word Problem for a finitely generated  group as a  formal language.
 We prove  Anisimov's  characterization of  groups with regular Word Problem and present the
Muller-Schupp characterization of groups with context-free Word Problem. We also discuss
some applications of formal language theory to subgroups  and present
Haring-Smith's characterization of basic groups in terms of their Word Problem.
In Section~\ref{sec:FGG} we consider the notion of a finitely generated graph and  the number of ends
of a finitely generated graphs together with Stallings Structure Theorem and the notion of accessibility. We then consider
the notion of finitely generated graphs with finitary end-structure
 and their characterization as complete transition graphs of pushdown automata.
Section~\ref{s:SOML} is devoted to second-order monadic logic where  we discuss   B\"uchi's theorem
on the decidability of second-order monadic theory $S1S$ and   Rabin's theorem on
the decidability of second-order monadic theory, $S2S$, of the infinite binary tree.
We then discuss  the  decidability of second-order monadic theory for complete transition graphs of pushdown automata.
We  consider  the  classical Domino Problem and its undecidability due to Berger and Robinson. After generalizing the Domino Problem to
finitely generated groups, we show that it is decidable  for virtually free groups.
In Section~\ref{sec:CA} we consider the Surjectivity, Injectivity,  and Bijectivity problems for cellular automata on
finitely generated groups and its decidability for virtually free groups. The last section is devoted to finite automata on
infinite inputs and   the  work of B\"uchi, of Rabin, and of Muller and Schupp.
We then discuss  infinite games of perfect information, the theorems of Davis and Martin, and the
Forgetful Determinacy theorem of Gurevich and Harrington.

\section{Languages, Grammars, and Automata}\label{sec:LGA}

\subsection{The free monoid over a finite alphabet}
\label{monoid}
Let $\Sigma$ be a finite \emph{alphabet}, that is, a finite set of \emph{letters}.
A \emph{word} on $\Sigma$ is any element of the set
$$
\Sigma^* = \bigcup_{n =0}^\infty \Sigma^n,
$$
where $\Sigma^n = \{a_1a_2\cdots a_n: a_k \in \Sigma, 1 \leq k \leq n\}$.
The number $|w| =n$ is the \emph{length} of the word $w = a_1a_2\cdots a_n$.
The unique word  of length zero  is denoted by $\varepsilon$ and is called the \emph{empty word}.
\par
The \emph{concatenation} of two words $w = a_1 a_2\cdots a_n \in \Sigma^n$ and $w' = a_1'a_2' \cdots a'_m \in \Sigma^m$ is the word
$ww' \in \Sigma^{n + m}$ defined by
\begin{equation}
\label{e;concatenation}
ww' = a_1a_2 \cdots a_na_1'a_2'\cdots a_m'.
\end{equation}
We have $\varepsilon w =w\varepsilon =w$  and $(ww')w'' = w(w'w'')$ for all $w,w',w'' \in \Sigma^*$.
Thus, $\Sigma^*$ is a monoid under  the concatenation product with identity element  the empty word $\varepsilon$.
The monoid $\Sigma^*$ satisfies the following universal mapping property: if $M$ is any  monoid,
then every map $f \colon \Sigma \to M$ uniquely extends to a monoid homomorphism $\varphi \colon \Sigma^* \to M$.
Due to this property,  $\Sigma^*$ is the \emph{free monoid} over $\Sigma$.
\par
Let $u,w$ be two words over $\Sigma$. One says that $u$ is a \emph{subword} of $w$
if there exist $u_1, u_2 \in \Sigma^*$ such that $w = u_1uu_2$.
\par
A \emph{language} over $\Sigma$ is a  subset $L \subset \Sigma^*$.

\subsection{Context-free languages}
   In this section, we discuss the class of context-free languages introduced by Chomsky~\cite{chomsky}.
\par
A \emph{context-free grammar} is a quadruple $\GG = (V,\Sigma,P,S_0)$,
where  $V$ is a finite set  of \emph{variables}, disjoint from
the finite alphabet $\Sigma$  of \emph{terminal symbols}. The variable
$S_0 \in V$ is the \emph{start symbol}, and $P \subset V \times (V \cup \Sigma)^*$ is a
finite set of \emph{production rules}. We write $S \vdash u$
 if $(S,u) \in P$.  For $v, w \in (V \cup \Sigma)^*$,
we write $v \then w$ if $v = v_1Sv_2$ and $w=v_1uv_2$, where
$u, v_1, v_2 \in (V \cup \Sigma)^*$ and $S \vdash u$.
The expression $v \then w$ is a single \emph{derivation step}, and it is called \emph{rightmost} if $v_2 \in \Sigma^*$.
A \emph{derivation} is a sequence $v=w_0,w_1, \dots,w_n=w \in (V \cup \Sigma)^*$ such
that $w_i \then w_{i+1}$ for each $i = 0, \dots, n-1$ and we then write $v \thens w$. A
\emph{rightmost derivation} is one where each step is rightmost.
It can be easily shown that if $v \thens w$ with $w\in \Sigma^*$, then there exists a
rightmost derivation $v \thens w$.
For $S \in V$, we consider the language
$L_S = \{w \in \Sigma^* : S \thens w \}$.
The \emph{language generated by}
$\GG$ is
$$
L(\GG) := L_{S_0} = \{w \in \Sigma^*: S_0 \thens w \}.
$$
A \emph{context-free language} is a language generated by a
context-free grammar.

\begin{example}[Dyck's language]
\label{example:dyck}
\rm{The  language of all correctly balanced   expressions involving several types of parentheses is  in some sense the ``primordial'' context-free language.
Let $n \ge 1$ and $\Sigma = \{a_1, \bar{a}_1, \dots, a_n, \bar{a}_n\}$. Consider the grammar $\GG$ with one single variable $S_0$ and productions
$S_0 \vdash \varepsilon \mbox{ and } S_0 \vdash a_i S_0 \bar{a}_i S_0, \ i=1,\dots, n$.
The language $L(\GG)$ generated by the grammar $\GG$ is called the \emph{Dyck language}.
Thinking  of the  $a_i$'s (resp. $\bar{a}_i$'s) as $n$ different ``open'' (resp. ``closed'') parenthesis symbols, then  $L(\GG)$ consists of all correctly nested parenthesis expressions over these symbols.
For example,
$$
\begin{array}{ll}
S_0 &\vdash a_2S_0{\bar a}_{2}S_0 \then a_2S_0{\bar a}_{2}a_1S_0{\bar a}_{1}S_0
\then a_2S_0{\bar a}_{2}a_1S_0{\bar a}_{1}\\
&\then a_2S_0{\bar a}_{2}a_1a_2S_0{\bar a}_{2}S_0{\bar a}_{1}
\then a_2S_0{\bar a}_{2}a_1a_2S_0{\bar a}_{2}{\bar a}_{1}
\then a_2S_0{\bar a}_{2}a_1a_2{\bar a}_{2}{\bar a}_{1}\\
&\then a_2{\bar a}_{2}a_1a_2{\bar a}_{2}{\bar a}_{1}
\end{array}
$$
is the unique rightmost derivation of $a_2\bar{a}_2a_1a_2\bar{a}_2{\bar a}_{1} \in L(\GG)$.}
\end{example}

A context-free grammar $\GG = (V,\Sigma,P,S_0)$ and its associated language $L(\GG)$
are called \emph{linear} if every production rule in $P$ is of the form
$S \vdash v_1Tv_2$ or $S \vdash v$, where $v,v_1,v_2 \in \Sigma^*$ and
$S, T \in V$. If  in this
situation one always has $v_2 = \varepsilon$ (the empty word), then the grammar and
language are called \emph{right linear}.
Similarly, the grammar and language are  \emph{left linear} if one always has $v_1 = \varepsilon$. It is well known
(cf.~\cite{chiswell, harrison, HU}) that both left linear and right linear grammars generate the same class of languages, namely, the class of  \emph{regular} languages.

\begin{example}[Palindromes]
\label{e:palindromes}
\rm{Let $\Sigma$ be a finite alphabet. A word $w = a_1a_2 \cdots a_n$ is a \emph{palindrome} provided that
$a_i = a_{n-i+1}$ for all $i=1,2,\ldots,n$, that is,
 $w$ is the same read both forwards and backwards.
We denote by $L_{\text{\rm pal}}(\Sigma)$ the language consisting of all palindromes over the alphabet $\Sigma$.
For example, $L_{\text{\rm pal}}(\{a\}) = \{a\}^* = \{\varepsilon, a,aa,aaa,\ldots\}$ and
$$
L_{\text{\rm pal}}(\{a,b\}) = \{\varepsilon, a,b,aa,bb,aaa,aba,bab,bbb,aaaa, abba,baab,bbbb, \ldots \}.
$$
Consider the grammar $\GG$ with a unique variable $S_0$ and productions of the form
$S_0 \vdash \varepsilon, S_0 \vdash a$ and $S_0 \vdash a S_0 a$, for each $a \in \Sigma$.
Then $\GG$ is a linear grammar and $L(\GG) = L_{\text{\rm pal}}(\Sigma)$. It follows that the language
consisting of all palindromes is linear.}
\end{example}

\begin{example}[The free group]
\label{example:freegroup}
Let $X$ be a finite set and denote by $F_X$ the free group based on $X$. (If $n$ denotes the cardinality of $X$ we shall also denote $F_X$ by $F_n$ and
 refer to it as to the \emph{free group of rank $n$}.) Let ${X}^{-1}$ be a disjoint copy of $X$ and set $\Sigma = X \cup {X}^{-1}$.
 We denote by $x \mapsto {x}^{-1}$ the involutive map on $\Sigma$ exchanging $X$ and ${X}^{-1}$ so  $({x}^{-1})^{-1} = x$ for all  $x \in X$. A word $w \in \Sigma^*$ is
\emph{reduced} if it contains no subword of the
form $xx^{-1}$ or $x^{-1}x$ for $x \in X$. For example, if $x, y \in X$ are distinct, then the
words $\varepsilon, x, xy, xy^{-1}, xy^{-1}x^{-1}$ are reduced, while $xx^{-1}, x^{-1}xy$ are not.
We denote by $L_{\text{\rm red}}(\Sigma) \subset \Sigma^*$ the language consisting of all reduced words. It is well known that
every element of $F_X$ has a unique representative as a reduced word in $L_{\text{\rm red}}(\Sigma)$.

Consider the grammar $\GG = (V,\Sigma,P,S_0)$ where $V = \{S_0\} \cup \{S_x: x \in \Sigma\}$
and $P$ consists of the productions of the form
$$
S_0 \vdash \varepsilon \mbox{ and } S_0 \vdash xS_x \mbox{ for all } x \in \Sigma
$$
and
$$
S_x \vdash \varepsilon \mbox{ and } S_x \vdash yS_y \mbox{ for all } y \in \Sigma \setminus \{x^{-1}\}
$$
for all $x \in \Sigma$. Note that $\GG$ is a right-linear grammar and that $L(\GG) = L_{\text{\rm red}}(\Sigma)$. Thus,
the language of all reduced words over $\Sigma$ is regular.
\end{example}

Returning to a general context--free grammar $\GG$, for a given variable
$S \in V$, we define the \emph{degree of ambiguity}, $d_S(w)$, of a word
$w \in \Sigma^*$ as the number of  different rightmost derivations
$S \thens w$. We have $d_S(w) > 0$ if and only if $w \in L_S$.
The grammar is called \emph{unambiguous} if $d_{S_0}(w) = 1$ for all $w \in L(\GG)$. Otherwise,
if there exists $w \in L(\GG)$ such that $d_{S_0}(w) > 1$, the grammar is
called \emph{ambiguous}.  A context-free language $L$ is called \emph{unambiguous} if it is generated by
some unambiguous grammar and \emph{inherently ambiguous} if all context-free grammars
generating $L$ are ambiguous. It is a fact that  there exist inherently ambiguous
context-free languages (cf.~\cite{HU}).

\subsection{Growth of context-free languages}
Let $\Sigma$ be a finite alphabet and $L \subset \Sigma^*$ a language. 

The \emph{growth function} of $L$ is the map $\gamma_L \colon \N \to \N$ defined by
$$
\gamma_L(n) = |\{w \in L: |w| \leq n\}|, \ \ \ n \in \N.
$$

Note that
\[
\gamma_L(n) \leq \gamma_{\Sigma^*}(n) = \sum_{k=0}^n \vert \Sigma \vert^k = \frac{\vert \Sigma \vert^{n+1} - 1}{\vert \Sigma \vert - 1} \leq \vert \Sigma \vert^{n+1} = C  \vert \Sigma \vert^{n}
\]
for all $n \in \N$ where $C \geq \vert \Sigma \vert$. It follows that there exist $C > 0$ and $a > 1$ such that
\begin{equation}
\label{e:not-superexponential}
\gamma_L(n) \leq C  a^n
\end{equation}
for all $n \in \N$.

The {\it growth rate} of  $L$ is the number
\begin{equation}\label{entL}
\lambda(L) = \limsup_{n \to \infty} |\{w \in L: |w| \leq n\}|^{\frac{1}{n}}.
\end{equation}
On says that $L$ is of {\it exponential growth} if $\lambda(L) > 1$. Otherwise,  if $\lambda(L) = 1$,  then $L$ is of
{\it sub-exponential growth}.
Note that $L$ is of exponential growth if and only if there exists $a > 1$ such that
$\gamma_L(n) \geq a^n$ for all $n \in \N$. A  language $L$ is said to be of \emph{polynomial growth} provided that
there exist an integer $d \geq 0$ and a constant $C > 0$ such that
$\gamma_L(n) \leq C + Cn^d$ for all $n \in \N$.
Finally, one says that $L$ is of \emph{intermediate growth} if its growth is
sub-exponential but not polynomial.
Note that a language cannot be of ``super-exponential growth'' by virtue of
\eqref{e:not-superexponential}.

Bridson and Gilman \cite{bridson} and, independently, Incitti \cite{incitti},  proved that the growth of a context-free language is either polynomial or exponential. An explicit algorithm
for determining this alternative is presented in \cite{TCSaam}.
On the other hand, Grigorchuk and Mach\`\i \ \cite{GM} presented an example of an
\emph{indexed language} of intermediate growth. (The class of indexed languages,
introduced by A.~Aho, properly contains the class of context-free languages and,
in turn, is properly contained in the class of computably enumerable languages.)

One says that the language $L$ is {\it growth-sensitive} if
$$
\lambda(L^F) < \lambda(L)
$$
for every non-empty $F \subset \Sigma^*$ consisting of
subwords of elements of $L$, where
$$
L^F = \{w \in L : \mbox{no $\;v \in F\;$ is a subword of $\;w$} \}.
$$

It is a well known fact, which can be deduced from the Perron-Frobenius theory
(see \cite{CMS1} for an alternative proof), that regular languages are growth-sensitive.
Ceccherini-Silberstein and Woess \cite{CW1,C2} (see also \cite{CSW2}) extended this result to all unambiguous ergodic context-free languages. (Here ``ergodicity'' corresponds to
strong connectedness of the dependency graph (in the sense of Kuich \cite{kuich}) associated with an unambiguous context-free grammar generating the language.)


\subsection{Finite automata}
\label{sec:fa}

A \emph{nondeterministic finite automaton} is a 5-tuple $\A = (Q,\Sigma, \delta, q_0, F)$ where $Q$ is a nonempty finite set
of \emph{states}, $\Sigma$ is a finite alphabet,
$q_0 \in Q$ is the \emph{initial state}, $F \subset Q$ is the set of \emph{final
states}, and the map
$$\delta \colon Q \times \Sigma \to \PP(Q)$$
is the \emph{transition function}. (As usual,  $\PP(Q)$ denotes the set of all subsets of $Q$.)
The automaton works as follows. When reading a word $w \in \Sigma^*$, letter by letter, from left to right, it can change
 its state according to the transition function.  A \emph{run} of $\A$ on a word $w = \sigma_1 \sigma_2 \cdots \sigma_n$ is a function
$\rho: \{0,1,\ldots,n + 1\} \to Q$ such that $\rho(0) = q_0$ and $\rho(i + 1) \in \delta(\rho(i), \sigma_i)$ for $i = 0,1,\ldots,n$.
A word $w = \sigma_1 \sigma_2 \cdots \sigma_n \in \Sigma^*$ is \emph{accepted} by $\A$
if there exists a run $\rho$ of $\A$ on $w$  such that  $ \rho(n + 1) \in F$.
In short, $\A$ accepts $w$ if there is a sequence of choices allowed by the transition function such that $\A$ is in a final state  after reading the word $w$.
The set of all words $w \in \Sigma^*$ accepted by $\A$ is called the
\emph{language accepted} by $\A$ and it is denoted by~$L(\A)$.

The automaton $\A$ is said to be \emph{deterministic} if
$|\delta(q,a)| \leq 1$ for all $q \in Q$ and $a \in \Sigma$, where $|\,\cdot\,|$ denotes cardinality.

The following is a fundamental characterization of regular languages (see, e.g.~\cite{chiswell, harrison, HU}).
\begin{theorem}
\label{t:characterization-regular}
Let $\Sigma$ be a finite alphabet and $L \subset \Sigma^*$ be a language.
Then $L$ is regular (that is, it is generated by a left-linear (equivalently, by a
right-linear) grammar) if and only if it is accepted by a deterministic finite  automaton.
\end{theorem}

\begin{example}[The free group]
\label{example:free-group}
Let $X$ be a finite set with $\Sigma =  X \cup {X}^{-1}$ and the map $x \mapsto {x}^{-1}$ as in Example~\ref{example:freegroup}.
Let $\A = (Q,\Sigma,\delta,q_0,F)$ be the finite state automaton with state set
$Q = \{q_0\} \cup \{q_{x}: x \in \Sigma\}$, $F = Q$ (all states are terminal), and where the transition function is defined by
\[
\begin{split}
\delta(q_0,x) & = q_x\\
\delta(q_x,y) & =
\begin{cases} q_y & \mbox{ if } y \neq {x}^{-1}\\
\varnothing & \mbox{ otherwise.}
\end{cases}
\end{split}
\]
for all $x,y \in \Sigma$.
It is immediate to see that the language accepted
by the automaton $\A$ consists of all reduced words over the alphabet $\Sigma$, that is,
$L(\A) = L_{\text{\rm red}}(\Sigma)$.
\end{example}

Graphically, one represents a finite automaton $\A = (Q, \Sigma,\delta, q_0,  F)$ as
a labelled graph. (See Section~\ref{s:lg} for more on labelled graphs).
The vertex set is $Q$ and, for every $p \in Q$
and $a \in \Sigma$, there is an oriented edge from $p$ to $q$, with label $a$,
for all $q \in \delta(p,a)$.
The initial state is denoted by an ingoing arrow into it and a double circle is drawn around each final state.
In Figure~\ref{fig:automaton-free-group} we represented the automaton $\A$ recognizing the language $L_{\text{\rm red}}(\Sigma)$ of reduced words on $\{x,y, x^{-1}, y^{-1}\}$.

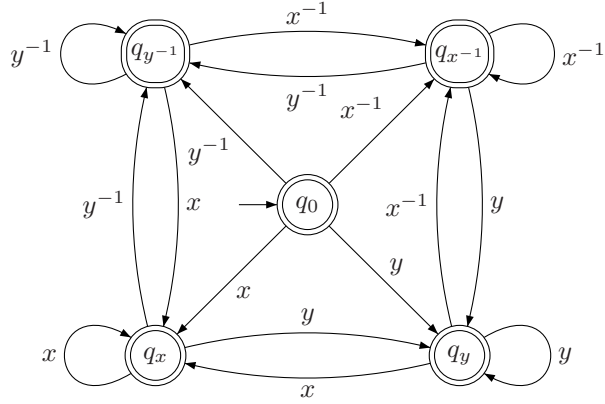
\begin{figure}[h!]
\begin{center}
\begin{picture}(0,50)(0,-20)
\node[Nmarks=r](A)(-20,-20){$q_x$}
\node[Nmarks=r](B)(20,-20){$q_y$}
\node[Nmarks=r,Nh=9,Nw=9](C)(20,20){$q_{x^{-1}}$}
\node[Nmarks=r,Nh=9,Nw=9](D)(-20,20){$q_{y^{-1}}$}
\node[Nmarks=ir](E)(0,0){$q_0$}
\drawedge(E,A){$x$}
\drawedge(E,B){$y$}
\drawedge(E,C){$x^{-1}$}
\drawedge(E,D){$y^{-1}$}
\drawedge[curvedepth=3](A,B){$y$}
\drawedge[curvedepth=3](A,D){$y^{-1}$}
\drawedge[curvedepth=3](B,A){$x$}
\drawedge[curvedepth=3](D,A){$x$}
\drawedge[curvedepth=3](B,C){$x^{-1}$}
\drawedge[curvedepth=3](C,B){$y$}
\drawedge[curvedepth=3](C,D){$y^{-1}$}
\drawedge[curvedepth=3](D,C){$x^{-1}$}
\drawloop[loopangle=180](A){$x$}
\drawloop[loopangle=0](B){$y$}
\drawloop[loopangle=0](C){$x^{-1}$}
\drawloop[loopangle=180](D){$y^{-1}$}
\end{picture}
\end{center}
\caption{The finite automaton accepting the reduced words of $F_{\{x,y\}}$.}\label{fig:automaton-free-group}
\end{figure}


\subsection{Pushdown automata}
A \emph{pushdown automaton} is a $7$-tuple
$\MM = (Q,\Sigma,Z, \delta, q_0, F, z_0)$, where $Q$ is a nonempty finite set of \emph{states}, $\Sigma$ is a finite alphabet, called
the \emph{input alphabet}, $Z$ is a finite set of \emph{stack symbols}, $q_0 \in Q$ is the \emph{initial state},
$F \subset Q$ is the set of \emph{final states}, and $z_0 \in Z \cup \{\varepsilon\}$ is the \emph{start symbol}.
Finally, the \emph{transition function} is a map
$$\delta \colon Q \times (\Sigma \cup \{\varepsilon\}) \times (Z\cup \{\varepsilon \}) \to \PP_{\text{\rm fin}}(Q \times Z^*)
$$
where
$\PP_{\text{\rm fin}}(Q \times Z^*)$ stands for the set of all finite subsets of $Q \times Z^*$.

\begin{figure}[h!]
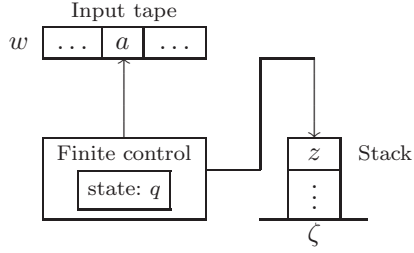


$$
\begin{array}{rcccccccccl}
&\multicolumn{3}{c}{\textup{\footnotesize Input tape}}&&&&&&&\\
\cline{2-4}
w&\multicolumn{1}{|c}{\dots} & \multicolumn{1}{|c}{a} & \multicolumn{1}{|c|}{\dots} &&&&&&&\\
\cline{2-4}\cline{7-8}
&\multicolumn{3}{c}{{\Bigg \uparrow}}&&&\multicolumn{1}{|c}{}&\multicolumn{2}{c}{{\Bigg \downarrow}}&&\\
\cline{2-4}\cline{8-9}
&\multicolumn{3}{|l|}{\textup{\footnotesize Finite control}}&&& \multicolumn{1}{|c}{}& \multicolumn{2}{|c|}{z} && \hspace{-0.3cm}\textup{\footnotesize Stack}\\
\cline{5-6}\cline{8-9}
&\multicolumn{3}{|c|}{\framebox(12,5){\mbox{\footnotesize state:\ $q$}}}&&&&\multicolumn{2}{|c|}{\vdots}&&\\
\cline{2-4}\cline{7-10}
&&&&&&&\multicolumn{2}{c}{\zeta}&&\\
\end{array}
$$
\caption{Representation of a pushdown automata. The input tape contains the word $w$ and its current letter is $a$. The stack contains the word $\zeta$ starting by the letter $z$.}\label{fig:pda}
\end{figure}

The automaton is represented in Figure~\ref{fig:pda} and works in the following way.   The automaton
reads a word $w \in \Sigma^*$ from the  input tape, letter by letter, from
left to right. At any time, it is in some state
$q \in Q$, and the stack contains a word $\zeta \in Z^*$.
If the current letter of $w$ is $a$, the state is $q$ and the top
symbol of the stack word $\zeta$ is $z$, then it performs one of the following
transitions:
\begin{enumerate}[\it (i)]
\item $\MM$ can move to the next position on the input tape.  If the letter read is $a$, $\MM$ selects some $(q',\zeta') \in \delta(q,a,z)$, changes to state $q'$,
   and replaces the rightmost symbol $z$ of $\zeta$ by $\zeta'$. If there are no more letters on the input tape the machine halts.

or, without advancing the tape,

\item $\MM$ can select some $(q',\zeta') \in \delta(q,\varepsilon,z)$, changes to state $q'$,
remain at the current position on the input tape  and replace the rightmost symbol $z$ of $\zeta$ by $\zeta'$.
Note that $\MM$ can make several successive moves of this type without advancing the tape.
Transitions of this type are called \emph{$\epsilon$-transitions}.

\end{enumerate}

If both $\delta(q,a,z)$ and $\delta(q,\varepsilon,z)$ are empty then $\MM$ halts.

Note that, in  general, a pushdown automaton is \emph{nondeterministic}
in the sense that it has more than one choice of a possible transition.
A  pushdown automaton $\MM$ is  \emph{deterministic} if for any
$q \in Q$, $a \in \Sigma$ and  $z \in Z \cup \{\varepsilon\}$, it has at most
one option of what to do next, that is,
$$
|\delta(q,a,z)| + |\delta(q,\varepsilon,z)| \leq 1.
$$

Since we are interested in groups, our convention is that the automaton is
allowed to continue to work when the stack is
empty, i.e., when $\zeta = \varepsilon$. Then the automaton acts
in the same way as before, by changing to state $q'$ and putting $\zeta'$ in the stack if it advances the tape and selects
$(q',\zeta') \in \delta(q,a,\varepsilon)$ in case {\it (i)}, or by  making an $\epsilon$-transition
$(q',\zeta') \in \delta(q,\varepsilon,\varepsilon)$ in case {\it (ii)}.
This convention is different from that of many authors, for example~\cite{HU},
who require  the automaton to halt on an empty stack.

Let $w \in \Sigma^*$, $q \in Q$, and $\zeta \in Z^*$. We write
$\MM \underset{w}{\overset{*}\vdash} (q, \zeta)$ if,  starting at the
initial state $q_0$ and with only $z_0$ in the stack, it is possible for the
automaton $\MM$ (after finitely many transitions) to be in state $q$ with $\zeta$ written on the stack, after reading the input $w$. If $q \in F$ and $\zeta = \varepsilon$ we say
that $\MM$ \emph{accepts} $w$. The \emph{language accepted} by $\MM$ is then defined by
$$
L(\MM) := \{w \in \Sigma^*: \MM \underset{w}{\overset{*}\vdash} (q, \varepsilon) \mbox{ for some } q \in F\}.
$$

\begin{example}
Every finite automaton $\A$ may be viewed as a pushdown automaton. Indeed, if $\A = (Q,\Sigma,\delta,q_0,F)$, consider the pushdown automaton
$\MM = (Q,\Sigma,Z,\delta',q_0,F,\varepsilon)$, where $Z = \varnothing$
and the transition function $\delta' \colon Q \times (\Sigma \cup \{\varepsilon\}) \times \{\varepsilon\} \to \PP_{\text{\rm fin}}(Q \times \{\varepsilon\})$ is defined by setting
$$
\delta'(q,a,\varepsilon) =  \{(q',\varepsilon): q' \in \delta(q,a)\}
$$
for all $q \in Q$ and $a \in \Sigma$.
It is clear that $L(\A) = L(\MM)$.
Note that $\MM$ is deterministic whenever $\A$ is deterministic.
\end{example}

The following is a fundamental characterization of context-free languages (see~\cite{chiswell, harrison, HU}).
\begin{theorem}[Chomsky]
\label{t:characterization--free}
Let $\Sigma$ be a finite alphabet and $L \subset \Sigma^*$ be a language.
Then $L$ is context-free (that is, it is generated by a context-free grammar) if and only if it is accepted by a pushdown automaton.
Moreover, $L$ is unambiguous if and only if it is accepted by a deterministic
pushdown automaton.
\end{theorem}

Note that since there exist inherently ambiguous context-free languages (which therefore
are not accepted by any deterministic pushdown automaton), it follows  that nondeterministic pushdown automata
are strictly more powerful than deterministic ones.

\begin{example}[The Dyck language revisited]
Let $n \ge 1$ and $\Sigma = \{a_1, \bar{a}_{1}, a_2, \bar a_2,$ $\dots,$ $a_n, \bar{a}_{n}\}$.
Consider the deterministic pushdown automaton $\MM = (Q,\Sigma,Z, \delta,$ $q_0, F, z_0)$
with $Q = \{q_0\} = F$, $Z = \Sigma$, $z_0 = \varepsilon$ and
$\delta \colon \{q_0\} \times (\Sigma \cup \{\varepsilon\}) \times (\Sigma \cup \{\varepsilon\}) \to \PP_{\text{\rm fin}}(\{q_0\} \times \Sigma^*)$
defined by setting
$$
\delta(q_0,a,z) =
\begin{cases}
 \{(q_0,\varepsilon)\} & \mbox{ if } a = {\bar z}\\
\{(q_0,za)\} & \mbox{ otherwise}
\end{cases}
$$
for all $a, z \in \Sigma \cup \{\varepsilon\}$. (We use
the convention that $\bar \varepsilon = \varepsilon$.)
Then it is easy to check that $L(\MM)$ is the Dyck language defined in Example~\ref{example:dyck}.
\end{example}

\subsection{Turing machines, computable and computably enumerable languages}
\label{s:Turing}
One of the great accomplishments of twentieth century mathematics was the formalization of the idea of  being ``computable''.  Probably the clearest model
is Turing's concept  of a \emph{Turing machine}~\cite{turing},
which one can consider as an idealized digital computer.  Several other
definitions were proposed  in the 1930's and 1940's  and all of these  definitions  have  been shown to
be equivalent.  The Turing machine model of computation is  the one still used in studying computational complexity, where
one wants to investigate how difficult it is to calculate something.

\begin{thesis}[The Church-Turing Thesis]
Any function intuitively thought to be computable is computable by a Turing machine.
\end{thesis}

Seventy years of research have led to the general acceptance of the Church-Turing Thesis.  By
the word ``algorithm'' we therefore mean a  Turing machine.

We give a brief description of how a Turing machine works. This description is illustrated in Figure~\ref{fig:tur}.  For a careful detailed discussion see~\cite{chiswell,cooper,HU}. A Turing machine $\T$ consists of the following:
\begin{itemize}
\item A \emph{tape} which is divided into consecutive \emph{cells} or \emph{squares}  and which is infinite to the right.
Thus the Turing machine always has enough tape for any  computation, that is, it has  unlimited memory.
There is a \emph{tape alphabet} $\Gamma$ which contains  a special  \emph{blank symbol} $\blank$.  The \emph{input alphabet} is
$\Sigma \subset \Gamma \setminus \{\blank\}$.
Each cell contains a symbol from the tape alphabet and
initially, all but finitely many cells contain the blank symbol $\blank$.
\item A \emph{reading head} that can read and write symbols on the tape and then move one cell to the right
or one cell to the left. Symbols $L$ and $R$ stand for ``left'' and ``right'', respectively.
\item A finite set $Q$ of \emph{control states} with an  \emph{initial state} $q_0 \in Q$ and a \emph{halting state}  $H \in Q$.
\item A   \emph{program} or  \emph{transition function}  $\delta: Q \times \Gamma  \to Q \times \Gamma \times \{L,R\}$.
 There is only one type of instruction and  Turing machines are  thus the ultimate in ``reduced instruction set architecture''.
If  $\delta(q,\gamma) = (q',\gamma',L/R)$ then the machine immediately halts if $q' = H$. Otherwise the machine does the following operations in sequence:
\begin{itemize}
\item replace the symbol $\gamma$ by the  symbol $\gamma'$, which may be the same as $\gamma$ or may be the blank $\blank$,
\item move the reading head one cell  to the left (on $L$) or one cell  to the right (on $R$),
\item assume the new state $q' \in Q$.
\end{itemize}
\end{itemize}
\begin{figure}[h!]
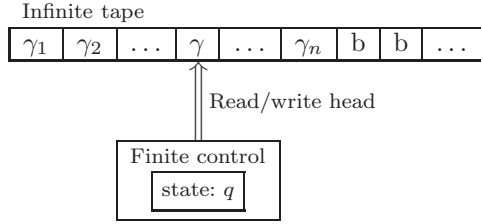

$$
\begin{array}{ccccccccc}
\multicolumn{9}{l}{\textup{\footnotesize Infinite tape}}\\
\hline
\multicolumn{1}{|c}{\gamma_1} & \multicolumn{1}{|c}{\gamma_2} & \multicolumn{1}{|c}{\dots} & \multicolumn{1}{|c}{\gamma} & \multicolumn{1}{|c}{\dots} & \multicolumn{1}{|c}{\gamma_n} & \multicolumn{1}{|c}{\blank} & \multicolumn{1}{|c|}{\blank} & \dots\\
\hline
&&&  \multicolumn{6}{l}{{\Bigg \Uparrow}\textup{\footnotesize Read/write head}}\\
\cline{3-5}
&& \multicolumn{3}{|l|}{\textup{\footnotesize Finite control}}&&&&\\
&& \multicolumn{3}{|c|}{\framebox(12,5){\mbox{\footnotesize state:\ $q$}}}&&&&\\
\cline{3-5}
\end{array}
$$
\caption{Representation of a Turing machine.}\label{fig:tur}
\end{figure}

A word $w \in \Sigma^*$ is \emph{written}
on the tape if it occupies the leftmost cells of the tape.
It is understood that all the cells that are on the right of the cell containing the last letter of $w$  contain the
blank symbol $\blank$.

Turing machines can be regarded either as \emph{calculators} of functions or as \emph{enumerators}.

\begin{definition}  Let $\Sigma_1$ and $\Sigma_2$ be finite alphabets.
A function $f : \Sigma_1^* \to \Sigma_2^*$ is \emph{computable} if
there exists a Turing machine $\T$ which, when started in its initial state with the reading head at the left end of the tape
and a word $w \in \Sigma_1^*$ written on the tape, eventually
halts  with  $f(w)\in \Sigma_2^*$ written on the tape.

A set $L \subseteq \Sigma^*$ is \emph{computable} if its characteristic function $\chi_L : \Sigma^* \to \{0,1\}$ is computable.
\end{definition}

Note that a Turing machine which calculates a function is required to halt on all inputs.  In general,  a Turing machine
with input alphabet $\Sigma$ may not halt on all inputs.

\begin{definition} A set $L \subseteq \Sigma^*$ is \emph{computably enumerable} if there exists a Turing machine $\T$ with
input alphabet $\Sigma$ such that $\T$ halts on input $w$ if and only if $w \in L$. We say that $\T$ \emph{enumerates} or \emph{accepts} $L$.
\end{definition}

Thus computably enumerable languages are exactly the halting sets of Turing machines.  The following lemma is a basic fact about
computability.

\begin{lemma}  A set $L \subseteq \Sigma^*$ is computable if and only if both $L$ and
its complement $\lnot L = \Sigma^* \setminus L$ are computably enumerable.
\end{lemma}
\begin{proof}  A basic principle of constructing  Turing machines is that a Turing machine $\T$ can be always be used as
a subroutine in a larger machine $\widehat{\T}$. If $L$ is computable, let $\T$ compute the characteristic function $\chi_L$ of $L$.
The machine $\widehat{\T}$ enumerating $L$ works as follows.  On input $w$, the machine $\widehat{\T}$ uses $\T$ to compute
$\chi_L(w)$.  If $w \in L$ then $\widehat{\T}$ halts.  If $w \notin L$ then $\widehat{\T}$ goes into a loop and never halts.
The machine enumerating the complement $\lnot L$ works similarly.

Conversely, suppose that $\T_1$ and $\T_2$ enumerate  $L$ and $\lnot L$ respectively. The machine $\overline{\T}$ computing $L$
uses the basic technique of ``bounded simulation''.  On input $w$, the machine $\overline{\T}$ begins successively enumerating positive
integers $n$.  When $n$ is enumerated, $\overline{\T}$ simulates both $\T_1$ and $\T_2$ on input $w$ for $n$ steps and sees if either machine
halts in $n$ steps.   Since $L$ and $\lnot L$ are complements, exactly one of $\T_1$ or $\T_2$ will eventually halt on input $w$.
When one of them halts,   $\overline{\T}$ then erases its tape and writes $1$ if $\T_1$ halted and $0$ if $\T_2$ halted.
\end{proof}

Note that in order to be able to prove that a problem is \emph{not} computable, it is necessary to have a complete
list of all possible means of computation.   We can assume that the input alphabet of a Turing machine contains the symbols $0$ and $1$.
It is not difficult to effectively assign a unique binary number $g(\T)$ to each Turing machine $\T$ (see~\cite{HU}).
The \emph{Halting Problem} for Turing machines is the following problem:  given a  Turing machine $\T$ and an input $w \in \{0,1\}^*$,
does the machine $\T$ halt on input $w$?  Turing~\cite{turing} showed that the Halting Problem is not computable.
Once one has a non-computable language  $L$, one can use ``reduction'' to show that a language $L'$ is not computable by
showing that $L$ is reducible to $L'$ in the sense that if $L'$ were computable then $L$ would be computable.  All non-computability results eventually go back to the Halting Problem.

\section{Finitely generated groups,  Cayley graphs, and the Word Problem}\label{sec:CG}
\subsection{Labelled graphs}
\label{s:lg}
A \emph{labelled graph} is a triple $\Gamma = (V,E,\Sigma)$,  where $V = V(\Gamma)$ is the set of \emph{vertices},
$\Sigma$ is a finite {alphabet}, and $E = E(\Gamma) \subset V \times \Sigma \times V$ is the set of
\emph{oriented, labelled edges}.

Let $\Gamma = (V,E,\Sigma)$ be a labelled graph.

We say that $\Gamma$ is \emph{finite} if its vertex set $V$ is finite and thus the edge set $E$ is also finite.

Given an edge $e= (u,a,v) \in E$ its \emph{label} is $\lambda(e) := a \in \Sigma$, its \emph{initial vertex} is $o(e) := u \in V$, and
its \emph{terminal vertex} is $t(e) := v \in V$.  We say that $e$ is \emph{outgoing}
from $u$ and \emph{ingoing} into $v$.  An edge $e$ can be visualized as an arrow from $o(e)$ to $t(e)$.

For $v \in V$ we denote by $\partial^o(v) \in [0,\infty]$
(resp. $\partial^t(v) \in [0,\infty]$) the number (possibly infinite) of edges outgoing from (resp. ingoing into) $v$.
The quantity $\partial(v) = \partial^o(v) + \partial^t(v) \in [0,\infty]$ is  the \emph{degree} of $v$.
An edge of the form $(v,a,v)$ is called a \emph{loop} at $v$ and is both an outgoing edge and an ingoing edge at $v$,
and so contributes $2$ to $\partial(v)$.
If $\partial(v) < \infty$ for all $v \in V$ one says that $\Gamma$ is \emph{locally finite}.
If the degrees of the vertices of $\Gamma$ are uniformly bounded, that is $\sup_{v \in V} \partial(v) < \infty$,
one says that $\Gamma$ has \emph{bounded degree}.

Suppose that $\Sigma$ is equipped with an involution $a \mapsto \bar{a}$. We then say that
$\Gamma$ is \emph{symmetric} if for each edge $e=(u,a,v) \in E$,
the \emph{inverse edge} $e^{-1}=(v,\bar{a},u)$ also belongs to $E$.
The  drawing convention\label{drawing convention} for symmetric graphs is that one draws only one directed edge (with the corresponding label) choosing between $e$ and $e^{-1}$.

Note that if $\Gamma$ is symmetric, we clearly have $\partial^o(v) = \partial^t(v)$ for each $v \in V$. If, in addition, there exists $d \in \N$ such that $d = \partial^o(v) = \partial^t(v)$ for all $v \in V$, one says that $\Gamma$ is \emph{regular} of \emph{degree} $d$.

We say that $\Gamma$ is \emph{deterministic} if at every vertex all outgoing edges
have distinct labels.

Note that our definition of a labelled graph allows \emph{multiple edges}, i.e., distinct edges of the form $e_1 = (u, a_1, v)$ and
$e_2 = (u, a_2, v)$, but this implies that $a_1 \neq a_2$. Thus, two edges must coincide if they have the  same initial vertex, the same terminal vertex, and the same label.

A \emph{subgraph} of $\Gamma$ is a labelled graph $\overline{\Gamma} = (\overline{V},\overline{E},\overline{\Sigma})$
such that $\overline{V} \subset V$, $\overline{E} \subset E$ and $\overline{\Sigma} \subset \Sigma$.

Let $\Gamma' = (V',E',\Sigma)$ be another labelled graph with the same label alphabet $\Sigma$.
A \emph{labelled graph-homomorphism} from $\Gamma$ to $\Gamma'$ is a map $\varphi \colon V \to V'$
such that $(\varphi(u),a,\varphi(v)) \in E'$ for all $(u,a,v) \in E$.
A \emph{labelled graph-isomorphism}  from $\Gamma$ to $\Gamma'$ is a bijective labelled graph-homomorphism from
$\Gamma$ to $\Gamma'$ such that the inverse map $\varphi^{-1} \colon V' \to V$ is also a labelled graph-homomorphism
from $\Gamma'$ to $\Gamma$.  Note that if $\varphi$ is a labelled graph-isomorphism  from $\Gamma$ to $\Gamma'$,
then the map $\psi \colon E \to E'$ defined by $\psi(u,a,v) = (\varphi(u),a,\varphi(v))$, for all $(u,a,v) \in E$,
is bijective with inverse map $\psi^{-1} \colon E' \to E$
given by $\psi^{-1}(u',a,v') = (\varphi^{-1}(u'),a,\varphi^{-1}(v'))$, for all $(u',a,v') \in E'$.

A \emph{path} in $\Gamma$ is a sequence $\pi = (e_1, e_2, \dots, e_n)$ of edges such
that $ o(e_{i+1}) = t(e_i) $ for $i=1,2,\ldots, n-1$.  We extend our notation for initial and terminal vertices to paths.  The vertex  $o(\pi) := o(e_1)$  is the  \emph{initial} vertex of $\pi$ and  $t(\pi) := t(e_n)$ is the \emph{terminal} vertex of $\pi$. We then says
that $\pi$ \emph{starts} at $o(\pi)$ and \emph{ends} at $t(\pi)$, equivalently it \emph{connects} $o(\pi)$ to $t(\pi)$. An edge $e \in E$ such that $o(e) = t(e)$ is called
a \emph{loop}. For every vertex $v \in V$, we also allow the \emph{empty path} starting and ending at $v$.

One says that $\Gamma$ is \emph{strongly connected} provided that for all vertices $u, v \in V$ there exists a path connecting $u$ to $v$. If $\Gamma$ is symmetric, the (obviously reflexive and transitive) relation in $V$ defined by $u \sim v$ provided that there exists a path in $\Gamma$ connecting $u$ to $v$ is also symmetric and therefore an equivalence
relation. Then the corresponding equivalence classes are called the \emph{connected components} of $\Gamma$; clearly, $\Gamma$ is strongly connected if and only if there exists a unique such a connected component.

Let $\pi = (e_1, e_2, \ldots, e_n)$ be a path.
The number  $|\pi| =n$ of edges is the \emph{length} of the path.
The \emph{label} of $\pi$ is  $\lambda(\pi) := \lambda(e_1) \lambda(e_2) \cdots \lambda(e_n) \in \Sigma^*$.
The empty path has length $0$ and is labelled by the empty word  $\varepsilon$.
If $t(\pi) = o(\pi)$ one says that $\pi$ is \emph{closed}.
If the vertices $o(e_1),t(e_1),t(e_2),\ldots,$ $t(e_n)$ are all distinct, then the path is called \emph{simple}.
If $\pi$ is closed, contains an edge  and its vertices are all distinct with the exception of $o(e_1)=t(e_n)$, then $\pi$ is called a \emph{cycle}.


Denote by $\Pi_{u,v}(\Gamma)$ the set of all paths $\pi$ in $\Gamma$ with initial
vertex $o(\pi) = u$ and terminal vertex $t(\pi) = v$.  More generally, given a subset $F \subset V$
we set $\Pi_{u,F}(\Gamma) := \bigcup_{v \in F} \Pi_{u,v}(\Gamma)$.
For $u \in V$ and $F \subset V$ we define the language
$$
L_{u,F}(\Gamma) := \{\lambda(\pi): \pi \in \Pi_{u,F}(\Gamma)\} \subset \Sigma^*.
$$
Note that $L_{u,F}(\Gamma)$ may  be empty.
\par
Suppose that a given vertex $v_0 \in V$ of $\Gamma$ is fixed as  \emph{origin} (or \emph{root} or \emph{basepoint}).
One then says that $\Gamma =(V,E,\Sigma,v_0)$ is a \emph{rooted} labelled graph.
A \emph{rooted} labelled graph-homomorphism (resp. \emph{rooted} labelled graph-isomorphism) from a rooted labelled graph
$\Gamma$ into a rooted labelled graph $\Gamma'$ is a labelled graph-homomorphism (resp. labelled graph-isomorphism)
$\varphi \colon V \to V'$ such that $\varphi(v_0) = v_0'$, where $v_0' \in V'$ is the root of $\Gamma'$.

\begin{example}[The rooted infinite binary tree $T_2$]
Let $\Sigma = \{0,1\}$.
Consider the rooted labelled graph $\Gamma = (\Sigma^*,E,\Sigma,\varepsilon)$
where $E = \{(v,a,va): v \in \Sigma^*, a \in \Sigma\}$.
The vertex corresponding to the empty word $\varepsilon$ is the root of $\Gamma$.
Note that for every vertex $v \in V$ one has $\partial^o(v) = 2$.

The graph $\Gamma$ is a rooted, directed tree called the \emph{rooted infinite binary tree} and it is denoted by $T_2$.
Figure~\ref{binaryTREE} illustrates it.
\end{example}
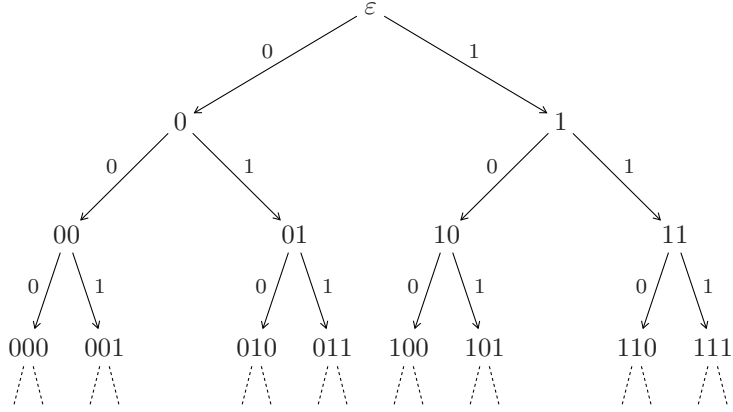
\begin{figure}[h!]
\begin{center}
\begin{picture}(0,60)(0,0)
\gasset{Nframe=n,Nadjust=w,Nh=5,ATangle=30,ATLength=1,ATlength=0,AHangle=30,AHLength=1,AHlength=0}
\node(E)(0,55){$\varepsilon$}
\node(0)(-25,40){$0$}
\node(1)(25,40){$1$}
\node(00)(-40,25){$00$}
\node(01)(-10,25){$01$}
\node(10)(10,25){$10$}
\node(11)(40,25){$11$}
\node(000)(-45,10){$000$}
\node(001)(-35,10){$001$}
\node(010)(-15,10){$010$}
\node(011)(-5,10){$011$}
\node(100)(5,10){$100$}
\node(101)(15,10){$101$}
\node(110)(35,10){$110$}
\node(111)(45,10){$111$}

\node(0000)(-42.5,0){$$}
\node(0001)(-47.5,0){$$}
\node(0010)(-37.5,0){$$}
\node(0011)(-32.5,0){$$}
\node(0100)(-17.5,0){$$}
\node(0101)(-12.5,0){$$}
\node(0110)(-7.5,0){$$}
\node(0111)(-2.5,0){$$}
\node(1000)(2.5,0){$$}
\node(1001)(7.5,0){$$}
\node(1010)(12.5,0){$$}
\node(1011)(17.5,0){$$}
\node(1100)(32.5,0){$$}
\node(1101)(37.5,0){$$}
\node(1110)(42.5,0){$$}
\node(1111)(47.5,0){$$}

\footnotesize
\drawedge[ATnb=1,AHnb=0](0,E){$0$}
\drawedge[ATnb=1,AHnb=0](00,0){$0$}
\drawedge[ATnb=1,AHnb=0](10,1){$0$}
\drawedge[ATnb=1,AHnb=0](000,00){$0$}
\drawedge[ATnb=1,AHnb=0](100,10){$0$}
\drawedge[ATnb=1,AHnb=0](010,01){$0$}
\drawedge[ATnb=1,AHnb=0](110,11){$0$}
\drawedge(E,1){$1$}
\drawedge(0,01){$1$}
\drawedge(1,11){$1$}
\drawedge(00,001){$1$}
\drawedge(01,011){$1$}
\drawedge(10,101){$1$}
\drawedge(11,111){$1$}

\drawedge[AHnb=0,dash={0.4}0](000,0000){}
\drawedge[AHnb=0,dash={0.4}0](000,0001){}
\drawedge[AHnb=0,dash={0.4}0](001,0010){}
\drawedge[AHnb=0,dash={0.4}0](001,0011){}
\drawedge[AHnb=0,dash={0.4}0](010,0100){}
\drawedge[AHnb=0,dash={0.4}0](010,0101){}
\drawedge[AHnb=0,dash={0.4}0](011,0110){}
\drawedge[AHnb=0,dash={0.4}0](011,0111){}
\drawedge[AHnb=0,dash={0.4}0](100,1000){}
\drawedge[AHnb=0,dash={0.4}0](100,1001){}
\drawedge[AHnb=0,dash={0.4}0](101,1010){}
\drawedge[AHnb=0,dash={0.4}0](101,1011){}
\drawedge[AHnb=0,dash={0.4}0](110,1100){}
\drawedge[AHnb=0,dash={0.4}0](110,1101){}
\drawedge[AHnb=0,dash={0.4}0](111,1110){}
\drawedge[AHnb=0,dash={0.4}0](111,1111){}
\end{picture}
\end{center}
\caption{The rooted infinite binary  tree $T_2$.}\label{binaryTREE}
\end{figure}

\begin{example}[The graph underlying a finite state automaton]
\label{e:underlying}
Let $\A = (Q,\Sigma, \delta, q_0,F)$ be a finite state automaton.
Consider the labelled graph $\Gamma = (V,E,\Sigma)$ where $V = Q$ and $E \subset V \times \Sigma \times V$ is defined by
$$
E = \{(u,a,v): u \in V, a \in \Sigma, \mbox{ and } v \in \delta(u,a)\}.
$$
Note that $\Gamma$ is deterministic if and only if $\A$ is deterministic.
The language $L(\A) \subset \Sigma^*$ accepted by $\A$ can be reinterpreted as the
language consisting of all words of the form $\lambda(\pi)$, where $\pi$ is a path in $\Gamma$ starting at the initial state $q_0$ and terminating at some final state in $F$.  In symbols:
$$
L(\A) = L_{q_0,F}(\Gamma) =  \{\lambda(\pi): \pi \in \Pi_{q_0,F}(\Gamma)\}.
$$
\end{example}

\subsection{Presentations and Cayley Graphs}
\label{s:PCG}

A \emph{finitely generated group presentation} is a pair $\langle X; R \rangle$, where $X$ is a finite set of \emph{generators},
the \emph{group alphabet} is  $\Sigma = X \cup X^{-1}$ where
$X^{-1}$ is a disjoint copy of $X$, and the set $R$ of \emph{defining relators} is a subset of $\Sigma^*$.
We denote by $a \mapsto a^{-1}$ the involutive map on $\Sigma$ exchanging $X$ and $X^{-1}$.

Two words $u,v \in \Sigma^*$ are said to be \emph{equivalent}, written $u \approx v$,
if it is possible to transform $u$ into $v$ by a finite sequence of insertions or
deletions of either the defining relators $r \in R$ or  the \emph{trivial relators} of the form $xx^{-1}$ and  $x^{-1}x$, with $x \in X$.
The concatenation product on the free monoid $\Sigma^*$ (cf. Equation~\eqref{e;concatenation}) induces a group structure on the set $G = \Sigma^*/\approx$ of equivalence classes whose identity element is the class of the
empty word $\varepsilon$. Moreover, if $w = a_1a_2 \cdots a_n \in \Sigma^*$, the inverse of the class of $w$ is the class
of the element $w^{-1} \in \Sigma^*$ defined by $w^{-1} = a_n^{-1} \cdots a_2^{-1}a_1^{-1}$.
One  says that $\langle X; R \rangle$  is a \emph{presentation} of the group $G$ and
one writes $G = \langle X; R \rangle$. When the defining relators $r \in R$ are of the
form $r = u_rv_r^{-1}$ for some $u_r,v_r \in \Sigma^*$ one often writes
$G = \langle X; u_r = v_r, r \in R \rangle$ and  refers to the equations  $u_r = v_r$,
$r \in R$, as the \emph{defining relations}.

A presentation $\langle X; R \rangle$ where both $X$ and the set $R$ of relators is finite is called a
\emph{finite presentation}. A group admitting a finite presentation is called  \emph{finitely presentable}.

Given a presentation $G = \langle X; R \rangle$, if $F_X$ denotes the free group based on $X$ and $N$ is the normal closure
of $R$ in $F_X$ then the  group
homomorphism $F_X \to G$ sending each $x \in X$ to  its $\approx$-equivalence class in
$\Sigma^*$  induces a group isomorphism $F_X/N \to G$.

\begin{example}
\label{ex:mult-table}
\begin{enumerate}[(a)]
\item Let $G = \{g_1,g_2,\ldots,g_n\}$ be a finite group  where $g_1$ is the identity element. The \emph{multiplication table presentation}
of $G$ is the presentation
$$G = \langle g_2,g_3,\ldots,g_n; g_ig_j = g_{k(i,j)}, i,j=2,3,\ldots,n\rangle,$$
where $g_{k(i,j)}$ is  the product of $g_i$ and $g_j$ determined from the multiplication table of $G$.  This example shows that every finite group has a finite presentation.

\item In multiplicative notation, the infinite cyclic group has a presentation
$\Z = \langle x \rangle$ with one generator and no defining relations.

\item[(b')] More generally, the free group based on a finite set $X$ has a presentation
$F_X = \langle X \rangle$ with generating set $X$ and no defining relations.

\item In multiplicative notation, the free abelian group of rank two has a presentation
$$\Z^2 = \langle x,y; [x,y]\rangle,$$ where $[x,y] = x^{-1}y^{-1}xy$ is
the commutator of $x$ and $y$.

\item[(c')] More generally, the free abelian group based on a finite set $X$ has
presentation $$\langle X; [x,y], x,y \in X \rangle.$$ In this case, the normal closure
$N$ of $R = \{[x,y], x,y \in X\}$ in the free group $F_X$ based on $X$ is the \emph{commutator} (or \emph{derived}) subgroup of $F_X$.

\item Let $G = (\Z/2\Z)*(\Z/2\Z)*(\Z/2\Z)*(\Z/2\Z)$ be the free product of four copies of
the group $\Z/2\Z$ with two elements. Let $x,y,z$ and $w$ be the nontrivial
elements in each copy of $\Z/2\Z$ (so that $x = x^{-1}, y = y^{-1}$, etc). Then the
corresponding presentation is $G = \langle x,y,z,w; x^2,y^2,z^2,w^2\rangle$.
\end{enumerate}
\end{example}

The fundamental geometric object associated with a finitely generated group was defined by  Cayley~\cite{cayley} in 1878.

\begin{definition}[Cayley graph]
\label{d:cayley}
Let $G = \langle X;R \rangle$  be a finitely generated group.
The \emph{Cayley graph} of $G$ with respect to the presentation $\langle X; R \rangle$ is the labelled graph $\Gamma = \Gamma(G:X;R)$
whose vertex set is $V(\Gamma) = G$, the set of labelled, directed edges is
$$
E(\Gamma) = \{(g,x,gx):  g \in G, x \in \Sigma\},
$$
and the label alphabet is $\Sigma = X \cup X^{-1}$.
\end{definition}

Let $\Gamma = \Gamma(G:X;R)$ be a Cayley graph. Then $\Gamma$ is often regarded as a rooted graph with basepoint $v_0 = 1_G$ and  is strongly connected:
between any two vertices $u$ and $v$ there is at least one  path from $u$ to $v$.
Note that a word $w \in \Sigma^*$ labels a closed path in $\Gamma(G:X;R)$ if and only if $w$ represents the identity in $G$.
Moreover, $\Gamma$ is symmetric (with respect to the involution $a \mapsto a^{-1}$
on $\Sigma$) and  $|X|$-regular. If $h \in G$ then the map $\mu_h : G \to G$, defined by  $\mu_h(g) = hg$ for all $g \in G$ is a labelled graph automorphism of $\Gamma$.
Thus a Cayley graph is \emph{homogeneous} in the sense that given any two vertices there is a lebelled graph automorphism taking the first vertex
to the second.

\begin{example}
\begin{enumerate}[(a)]
\item In Figure~\ref{CGK} we illustrate the Cayley graph of the Klein $4$-group $\Z/2\Z \times \Z/2\Z$ with respect to the multiplication table presentation
$\langle x,y,z ; x^2=y^2=z^2=1, xy=z=yx, xz=y=zx, yz=x=zy \rangle$.
\begin{figure}[h!]
\begin{center}
\gasset{Nframe=n,Nadjust=w,AHnb=1,AHangle=30,AHLength=1,AHlength=0}
\begin{picture}(0,39.6)(0,-17.3)
\node(1)(0,-2){$1$}
\node(x)(0,17.3){$x$}
\node(y)(20,-17.3){$y$}
\node(z)(-20,-17.3){$z$}
{\footnotesize
\drawedge(1,x){$x$}
\drawedge[ELside=r](1,y){$y$}
\drawedge(1,z){$z$}
\drawedge(x,y){$z$}
\drawedge[ELside=r](y,z){$x$}
\drawedge(z,x){$y$}
}
\end{picture}
\end{center}
\caption{The Cayley graph of the Klein $4$-group $\Z/2\Z \times \Z/2\Z$ with respect
to the multiplication table presentation $\langle x,y,z ; x^2=y^2=z^2=1, xy=z=yx, xz=y=zx, yz=x=zy \rangle$.}\label{CGK}
\end{figure}
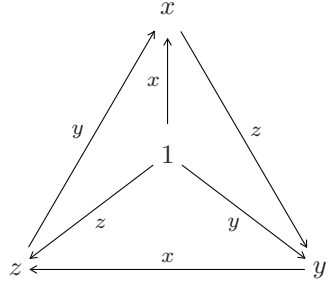

\item The Cayley graph $\Gamma(\Z: x)$ is described in Figure~\ref{CGZ}.
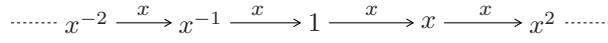
\begin{figure}[h!]
\begin{center}
\gasset{Nframe=n,AHnb=0,Nadjust=w}

\begin{picture}(0,5)(0,0)
\node(1)(0,0){$1$}
\node(x)(15,0){$x$}
\node(xx)(30,0){$x^2$}
\node(x-)(-15,0){$x^{-1}$}
\node(x-x-)(-30,0){$x^{-2}$}
\node(11)(-41,0){}
\node(12)(40,0){}
\drawedge[dash={0.4}0](x-x-,11){}
\drawedge[dash={0.4}0](12,xx){}
{\footnotesize
\gasset{AHnb=1,AHangle=30,AHLength=1,AHlength=0}
\drawedge(x-x-,x-){$x$}
\drawedge(x-,1){$x$}
\drawedge(1,x){$x$}
\drawedge(x,xx){$x$}}
\end{picture}
\end{center}
\caption{The Cayley graph of the group $\mathbb{Z} = \langle x \rangle$}\label{CGZ}
\end{figure}

\item[(b')] The Cayley graph $\Gamma(F_2 : x,y)$ is described in Figure~\ref{CGF2}.
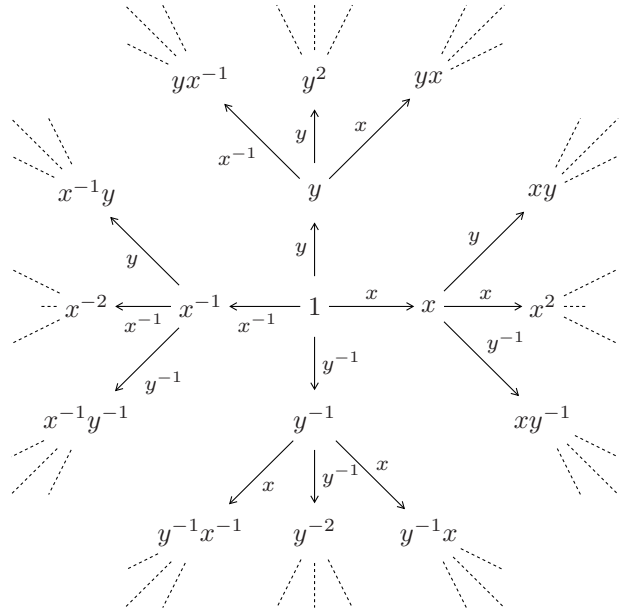
\begin{figure}[h!]
\begin{center}
\gasset{Nframe=n,AHnb=0,Nadjust=w}

\begin{picture}(0,85)(0,-40)
\node(1)(0,0){$1$}
\node(x)(15,0){$x$}
\node(y)(0,15){$y$}
\node(y-)(0,-15){$y^{-1}$}
\node(x-)(-15,0){$x^{-1}$}
\node(xx)(30,0){$x^2$}
\node(xy-)(30,-15){$xy^{-1}$}
\node(xy)(30,15){$xy$}
\node(yx-)(-15,30){$yx^{-1}$}
\node(yy)(0,30){$y^2$}
\node(yx)(15,30){$yx$}
\node(x-y)(-30,15){$x^{-1}y$}
\node(x-x-)(-30,0){$x^{-2}$}
\node(x-y-)(-30,-15){$x^{-1}y^{-1}$}
\node(y-x)(15,-30){$y^{-1}x$}
\node(y-y-)(-0,-30){$y^{-2}$}
\node(y-x-)(-15,-30){$y^{-1}x^{-1}$}

\gasset{Nframe=n,Nadjust=w,Nadjust=h,Nadjustdist=-1,AHnb=0}
\node(yx-1)(-25,35){}
\node(yx-2)(-25,40){}
\node(yx-3)(-20,40){}
\drawedge[dash={0.4}0](yx-,yx-1){}
\drawedge[dash={0.4}0](yx-,yx-2){}
\drawedge[dash={0.4}0](yx-,yx-3){}
\node(yx1)(20,40){}
\node(yx2)(25,40){}
\node(yx3)(25,35){}
\drawedge[dash={0.4}0](yx,yx1){}
\drawedge[dash={0.4}0](yx,yx2){}
\drawedge[dash={0.4}0](yx,yx3){}
\node(yy1)(0,40){}
\node(yy2)(-5,40){}
\node(yy3)(5,40){}
\drawedge[dash={0.4}0](yy,yy1){}
\drawedge[dash={0.4}0](yy,yy2){}
\drawedge[dash={0.4}0](yy,yy3){}
\node(xy-1)(40,-20){}
\node(xy-2)(40,-25){}
\node(xy-3)(35,-25){}
\drawedge[dash={0.4}0](xy-,xy-1){}
\drawedge[dash={0.4}0](xy-,xy-2){}
\drawedge[dash={0.4}0](xy-,xy-3){}
\node(xy1)(40,20){}
\node(xy2)(40,25){}
\node(xy3)(35,25){}
\drawedge[dash={0.4}0](xy,xy1){}
\drawedge[dash={0.4}0](xy,xy2){}
\drawedge[dash={0.4}0](xy,xy3){}
\node(xx1)(40,5){}
\node(xx2)(40,0){}
\node(xx3)(40,-5){}
\drawedge[dash={0.4}0](xx,xx1){}
\drawedge[dash={0.4}0](xx,xx2){}
\drawedge[dash={0.4}0](xx,xx3){}
\node(x-y-1)(-40,-20){}
\node(x-y-2)(-40,-25){}
\node(x-y-3)(-35,-25){}
\drawedge[dash={0.4}0](x-y-,x-y-1){}
\drawedge[dash={0.4}0](x-y-,x-y-2){}
\drawedge[dash={0.4}0](x-y-,x-y-3){}
\node(x-y1)(-40,20){}
\node(x-y2)(-40,25){}
\node(x-y3)(-35,25){}
\drawedge[dash={0.4}0](x-y,x-y1){}
\drawedge[dash={0.4}0](x-y,x-y2){}
\drawedge[dash={0.4}0](x-y,x-y3){}
\node(x-x-1)(-40,5){}
\node(x-x-2)(-40,0){}
\node(x-x-3)(-40,-5){}
\drawedge[dash={0.4}0](x-x-,x-x-1){}
\drawedge[dash={0.4}0](x-x-,x-x-2){}
\drawedge[dash={0.4}0](x-x-,x-x-3){}
\node(y-x1)(25,-35){}
\node(y-x2)(25,-40){}
\node(y-x3)(20,-40){}
\drawedge[dash={0.4}0](y-x,y-x1){}
\drawedge[dash={0.4}0](y-x,y-x2){}
\drawedge[dash={0.4}0](y-x,y-x3){}
\node(y-x-1)(-25,-35){}
\node(y-x-2)(-25,-40){}
\node(y-x-3)(-20,-40){}
\drawedge[dash={0.4}0](y-x-,y-x-1){}
\drawedge[dash={0.4}0](y-x-,y-x-2){}
\drawedge[dash={0.4}0](y-x-,y-x-3){}
\node(y-y-1)(-5,-40){}
\node(y-y-2)(0,-40){}
\node(y-y-3)(5,-40){}
\drawedge[dash={0.4}0](y-y-,y-y-1){}
\drawedge[dash={0.4}0](y-y-,y-y-2){}
\drawedge[dash={0.4}0](y-y-,y-y-3){}

{\footnotesize
\gasset{AHnb=1,AHangle=30,AHLength=1,AHlength=0}
\drawedge(1,x){$x$}
\drawedge(1,y){$y$}
\drawedge(1,y-){$y^{-1}$}
\drawedge(1,x-){$x^{-1}$}
\drawedge(x,xy){$y$}
\drawedge(x,xy-){$y^{-1}$}
\drawedge(x,xx){$x$}
\drawedge(y,yy){$y$}
\drawedge(y,yx){$x$}
\drawedge(y,yx-){$x^{-1}$}
\drawedge(x-,x-y){$y$}
\drawedge(x-,x-y-){$y^{-1}$}
\drawedge(x-,x-x-){$x^{-1}$}
\drawedge(y-,y-x){$x$}
\drawedge(y-,y-y-){$y^{-1}$}
\drawedge(y-,y-x-){$x$}}
\end{picture}
\end{center}
\caption{The Cayley graph of the free group $F_2 = \langle x, y \rangle$}\label{CGF2}
\end{figure}

\item The Cayley graph $\Gamma(\Z^2: x,y ; [x,y])$ is described in Figure~\ref{CGZ2}.
\begin{figure}[h!]
\begin{center}
\gasset{Nframe=n,AHnb=0,Nadjust=w}

\begin{picture}(0,95)(0,-45)
\node(1)(0,0){$1$}
\node(x)(15,0){$x$}
\node(xx)(30,0){$x^2$}
\node(y)(0,15){$y$}
\node(yy)(0,30){$y^2$}
\node(x-)(-15,0){$x^{-1}$}
\node(x-x-)(-30,0){$x^{-2}$}
\node(y-)(0,-15){$y^{-1}$}
\node(y-y-)(0,-30){$y^{-2}$}
\node(x-y)(-15,15){$x^{-1}y$}
\node(xy)(15,15){$xy$}
\node(x-y-)(-15,-15){$x^{-1}y^{-1}$}
\node(xy-)(15,-15){$xy^{-1}$}

{\footnotesize
\gasset{AHnb=1,AHangle=30,AHLength=1,AHlength=0}
\drawedge(x-x-,x-){$x$}
\drawedge(x-,1){$x$}
\drawedge(1,x){$x$}
\drawedge(x,xx){$x$}

\drawedge(y,yy){$y$}
\drawedge(1,y){$y$}
\drawedge(y-,1){$y$}
\drawedge(y-y-,y-){$y$}
\drawedge(x-y,y){$x$}
\drawedge(y,xy){$x$}
\drawedge(x-y-,y-){$x$}
\drawedge(y-,xy-){$x$}
\drawedge(x-,x-y){$y$}
\drawedge(x-y-,x-){$y$}
\drawedge(x,xy){$y$}
\drawedge(xy-,x){$y$}
}

\node(3)(-15,30){}
\node(4)(0,45){}
\node(5)(15,30){}
\node(9)(-30,15){}
\node(10)(30,15){}
\node(11)(-45,0){}
\node(12)(45,0){}
\node(13)(-30,-15){}
\node(14)(30,-15){}
\node(18)(-15,-30){}
\node(19)(0,-45){}
\node(20)(15,-30){}
\drawedge[dash={0.4}0](3,x-y){}
\drawedge[dash={0.4}0](4,yy){}
\drawedge[dash={0.4}0](5,xy){}
\drawedge[dash={0.4}0](9,x-y){}
\drawedge[dash={0.4}0](10,xy){}
\drawedge[dash={0.4}0](11,x-x-){}
\drawedge[dash={0.4}0](12,xx){}
\drawedge[dash={0.4}0](13,x-y-){}
\drawedge[dash={0.4}0](14,xy-){}
\drawedge[dash={0.4}0](18,x-y-){}
\drawedge[dash={0.4}0](19,y-y-){}
\drawedge[dash={0.4}0](20,xy-){}

\drawedge[dash={0.4}0](3,yy){}
\drawedge[dash={0.4}0](5,yy){}
\drawedge[dash={0.4}0](10,xx){}
\drawedge[dash={0.4}0](18,y-y-){}
\drawedge[dash={0.4}0](20,y-y-){}
\drawedge[dash={0.4}0](9,x-x-){}
\drawedge[dash={0.4}0](13,x-x-){}
\end{picture}
\end{center}
\caption{The Cayley graph of the group $\mathbb{Z}^2 = \langle x, y ; xy = yx \rangle$}\label{CGZ2}
\end{figure}
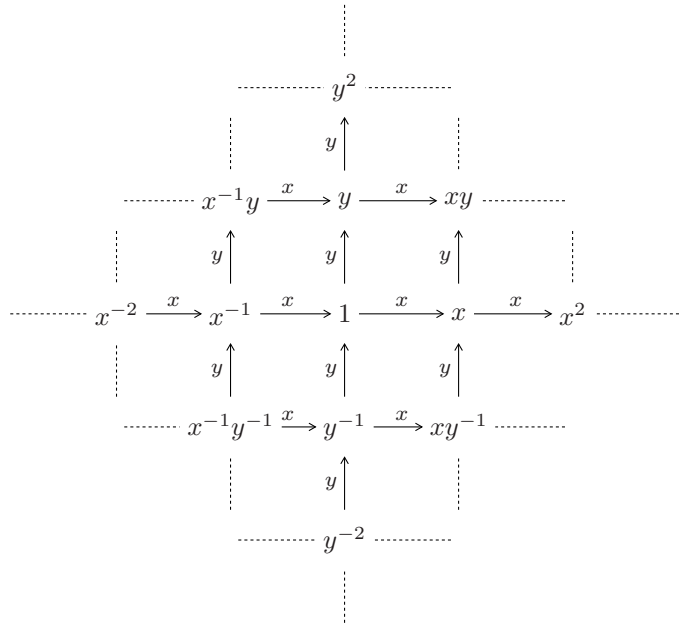

\item The Cayley graph $\Gamma(G: x,y,z,w; x^2,y^2,z^2,w^2)$, where $G =  (\Z/2\Z)*(\Z/2\Z)*(\Z/2\Z)*(\Z/2\Z)$, is described in Figure~\ref{CGC2*C2*C2*C2}.
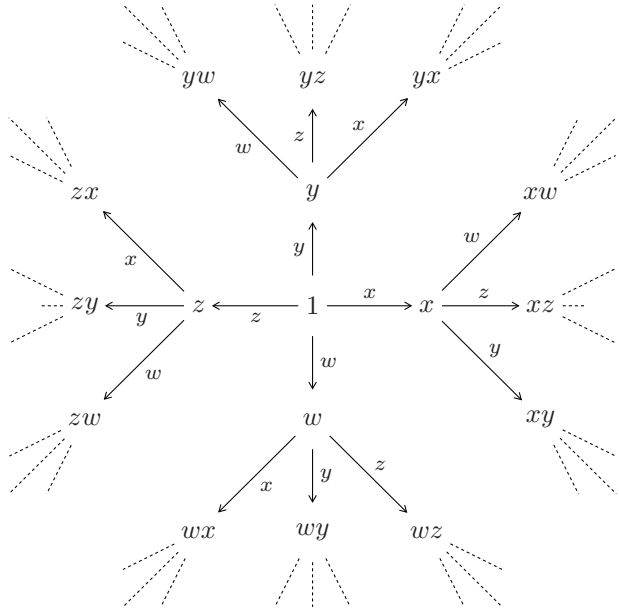
\begin{figure}[h!]
\begin{center}
\gasset{Nframe=n,AHnb=0,Nadjust=w}
\begin{picture}(0,85)(0,-40)
\node(1)(0,0){$1$}
\node(x)(15,0){$x$}
\node(y)(0,15){$y$}
\node(w)(0,-15){$w$}
\node(z)(-15,0){$z$}
\node(xz)(30,0){$xz$}
\node(xy)(30,-15){$xy$}
\node(xw)(30,15){$xw$}
\node(yw)(-15,30){$yw$}
\node(yz)(0,30){$yz$}
\node(yx)(15,30){$yx$}
\node(zx)(-30,15){$zx$}
\node(zy)(-30,0){$zy$}
\node(zw)(-30,-15){$zw$}
\node(wz)(15,-30){$wz$}
\node(wy)(-0,-30){$wy$}
\node(wx)(-15,-30){$wx$}
{\footnotesize
\gasset{AHnb=1,AHangle=30,AHLength=1,AHlength=0}

\drawedge(1,x){$x$}
\drawedge(1,y){$y$}
\drawedge(1,w){$w$}
\drawedge(1,z){$z$}

\drawedge(x,xy){$y$}
\drawedge(x,xw){$w$}
\drawedge(x,xz){$z$}

\drawedge(y,yw){$w$}
\drawedge(y,yx){$x$}
\drawedge(y,yz){$z$}

\drawedge(z,zy){$y$}
\drawedge(z,zw){$w$}
\drawedge(z,zx){$x$}

\drawedge(w,wx){$x$}
\drawedge(w,wy){$y$}
\drawedge(w,wz){$z$}
}
\gasset{Nframe=n,Nadjust=w,Nadjust=h,Nadjustdist=-1,AHnb=0}

\node(yw1)(-25,35){}
\node(yw2)(-25,40){}
\node(yw3)(-20,40){}
\drawedge[dash={0.4}0](yw,yw1){}
\drawedge[dash={0.4}0](yw,yw2){}
\drawedge[dash={0.4}0](yw,yw3){}

\node(yx1)(20,40){}
\node(yx2)(25,40){}
\node(yx3)(25,35){}
\drawedge[dash={0.4}0](yx,yx1){}
\drawedge[dash={0.4}0](yx,yx2){}
\drawedge[dash={0.4}0](yx,yx3){}

\node(yz1)(-5,40){}
\node(yz2)(0,40){}
\node(yz3)(5,40){}
\drawedge[dash={0.4}0](yz,yz1){}
\drawedge[dash={0.4}0](yz,yz2){}
\drawedge[dash={0.4}0](yz,yz3){}

\node(xy1)(40,-20){}
\node(xy2)(40,-25){}
\node(xy3)(35,-25){}
\drawedge[dash={0.4}0](xy,xy1){}
\drawedge[dash={0.4}0](xy,xy2){}
\drawedge[dash={0.4}0](xy,xy3){}

\node(xw1)(40,20){}
\node(xw2)(40,25){}
\node(xw3)(35,25){}
\drawedge[dash={0.4}0](xw,xw1){}
\drawedge[dash={0.4}0](xw,xw2){}
\drawedge[dash={0.4}0](xw,xw3){}

\node(xz1)(40,5){}
\node(xz2)(40,0){}
\node(xz3)(40,-5){}
\drawedge[dash={0.4}0](xz,xz1){}
\drawedge[dash={0.4}0](xz,xz2){}
\drawedge[dash={0.4}0](xz,xz3){}

\node(zw1)(-40,-20){}
\node(zw2)(-40,-25){}
\node(zw3)(-35,-25){}
\drawedge[dash={0.4}0](zw,zw1){}
\drawedge[dash={0.4}0](zw,zw2){}
\drawedge[dash={0.4}0](zw,zw3){}

\node(zx1)(-40,20){}
\node(zx2)(-40,25){}
\node(zx3)(-35,25){}
\drawedge[dash={0.4}0](zx,zx1){}
\drawedge[dash={0.4}0](zx,zx2){}
\drawedge[dash={0.4}0](zx,zx3){}

\node(zy1)(-40,5){}
\node(zy2)(-40,0){}
\node(zy3)(-40,-5){}
\drawedge[dash={0.4}0](zy,zy1){}
\drawedge[dash={0.4}0](zy,zy2){}
\drawedge[dash={0.4}0](zy,zy3){}

\node(wz1)(25,-35){}
\node(wz2)(25,-40){}
\node(wz3)(20,-40){}
\drawedge[dash={0.4}0](wz,wz1){}
\drawedge[dash={0.4}0](wz,wz2){}
\drawedge[dash={0.4}0](wz,wz3){}

\node(wx1)(-25,-35){}
\node(wx2)(-25,-40){}
\node(wx3)(-20,-40){}
\drawedge[dash={0.4}0](wx,wx1){}
\drawedge[dash={0.4}0](wx,wx2){}
\drawedge[dash={0.4}0](wx,wx3){}

\node(wy1)(-5,-40){}
\node(wy2)(0,-40){}
\node(wy3)(5,-40){}
\drawedge[dash={0.4}0](wy,wy1){}
\drawedge[dash={0.4}0](wy,wy2){}
\drawedge[dash={0.4}0](wy,wy3){}

\end{picture}
\end{center}
\caption{The Cayley graph of the group $(\mathbb{Z}/2\mathbb{Z}) * (\mathbb{Z}/2\mathbb{Z}) * (\mathbb{Z}/2\mathbb{Z}) * (\mathbb{Z}/2\mathbb{Z}) = \langle x, y, z, w ; x^2 = y^2 = z^2 = w^2 = 1 \rangle$}\label{CGC2*C2*C2*C2}
\end{figure}
\end{enumerate}

Note that the Cayley graphs in (c) and (d) are $4$-regular directed trees and they are isomorphic as directed graphs. However they are not isomorphic as directed labelled graphs.
\end{example}

Let $G = \langle X;R \rangle$ be a finitely generated presentation and let
$\Gamma = \Gamma(G:X;R)$ be the corresponding Cayley graph. When equipped with the metric
$\dist \colon V \times V \to [0,\infty)$
defined by $\dist(u,v) = \min\{|\pi(u,v)|: \pi \in \Pi_{u,v}\}$,  $\Gamma$  is a discrete metric
space. Denote by $B_n = \{g \in G: \dist(g,1_G) \leq n\}$ the \emph{ball} of radius
$n$ centered at $1_G$. The map $\gamma = \gamma(G:X;R) \colon \N \to \N$ defined by $\gamma(n) = |B_n|$ for all $n \in \N$ is called the \emph{growth function} of $G$ with respect to the given presentation. Since $\gamma$ is subadditive (i.e.
$\gamma(n+m) \leq \gamma(n)\gamma(m)$ for all $n,m \in \N$), by a well known result of
Fekete the limit
$$\lambda = \lambda(G:X;R) = \lim_{n \to \infty} \sqrt{\gamma(n)},$$
exists and $1 \leq \lambda < \infty$.
This limit is called the \emph{growth rate} of $G$ with respect to the given presentation,

That  $\lambda = 1$ is a condition independent of the particular presentation.
If $G = \langle X';R' \rangle$ is another finitely generated presentation of $G$ and $\gamma'$ is
the corresponding growth function, then
$\lambda' = \lim_{n \to \infty} \sqrt{\gamma'(n)}$ equals $1$ if and only if $\lambda$ does. If $\lambda = 1$ one says that
the group $G$ has \emph{subexponential growth}.
Otherwise, the group $G$ is said to have \emph{exponential growth}.
All finite groups, all finitely generated abelian groups, and, more generally, all nilpotent groups have subexponential growth.
On the other hand, if $F_X = \langle X \rangle$ is a finitely generated free group, then $\lambda = 2|X|-1$ so that
$F_X$ has exponential growth if $|X| \geq 2$.

\subsection{The Word Problem}

In a remarkable paper in 1911, twenty years before the development of the theory of computability,
 Dehn~\cite{dehn11} posed  three  fundamental decision problems in group theory:
the \emph{Word Problem}, the \emph{Conjugacy Problem}, and the \emph{Isomorphism Problem}.
 (See also the expository article by  de la Harpe~\cite{delaharpe}.)
Dehn viewed the Word Problem as the following algorithmic problem: given a finitely generated group presentation
$G = \langle X; R \rangle$ find an algorithm which, when given a word $w \in \Sigma^*$,
decides, in a finite number of steps, whether or not $w$ represents the identity element of $G$.
In 1912 Dehn~\cite{dehn12} solved this problem for the \emph{fundamental group of a closed orientable surface}:
$$
G_h = \langle a_1, b_1, a_2, b_2, \ldots, a_h, b_h; \prod_{i=1}^h [a_i,b_i]\rangle
$$
where $h\geq 2$ is the genus of the surface.

Given  a finite group presentation $G = \langle X; r_1, \ldots,r_k \rangle$,
let $R$ be the \emph{symmetrized} set generated by the given relators, that is, $R$ consists of all
cyclic permutations of the $r_i$ and their inverses.  Then  $ \langle X; R \rangle$ is also a presentation of $G$.
The original presentation is  a \emph{Dehn presentation} if every nontrivial word $w$ equal to the identity in $G$ contains a
subword $u$ such that some $r \in R$ has the form $r = uv$ where $|u| > |v|$.  This says that every nontrivial word
equal to the identity contains more than half of a cyclic permutation  of the given relators or their  inverses.

Although we usually do not write the trivial relators,  if $X$ is a finite set and $F_X$ is the free group based on $X$, then
a Dehn presentation of $F_X$ is given by $\langle X; xx^{-1}, x^{-1}x, x \in X\rangle$.

Now, every group admitting a Dehn presentation has solvable Word Problem. Indeed, if $G = \langle X; r_1,\ldots,r_k \rangle$
is a Dehn presentation,  let $R$ be the symmetrized set of relators generated by the $r_i$.
We  then have the following algorithm, now called \emph{Dehn's algorithm},  to decide whether or not   $w \in \Sigma^*$ represents
the identity element  in $G$:
\begin{enumerate}[\it Step 1)]
\item if $w = \varepsilon$ then $w$ does represent $1_G$, otherwise go to the next step;
\item if $w$ contains a subword $u$  where   for some $r \in R,  r = uv$ with $|u| > |v|$,
 then replace $u$ by $v^{-1}$  and go to Step 1.  Otherwise, $w$ does not represent $1_G$.
\end{enumerate}
 Note that since each step in the algorithm strictly reduces the length of the word being considered,
Dehn's algorithm takes only linearly many steps and thus works in linear time, which  is  the best possible complexity result.
The Cayley graph of the surface group $G_h, h \ge 2$ is the dual graph of  the regular tessellation of the
hyperbolic plane by $4h$-gons.   Dehn used hyperbolic geometry to show that the presentation of $G_h$ given above
is  a Dehn presentation  and thus $G$ has solvable Word Problem.  The quest to extend Dehn's algorithm to a larger class of groups led to the development of \emph{small cancellation theory} which, among many other things, gives  some simple sufficient conditions for a presentation to be a Dehn presentation.  (See~\cite{schupp1} for a survey.) This then
led to Gromov's~\cite{gromov} remarkable development of the theory of \emph{word-hyperbolic} groups.
As mentioned before,  the Cayley graph of a finitely generated group becomes  a metric space by
defining the distance between two vertices as the minimal length of a path
connecting them and considering each edge as isometric to the unit interval.  The \emph{thin triangle condition}
then captures many of the features of hyperbolic geometry.
One of the characterizations of a group $G$ being word-hyperbolic is exactly that it has some Dehn presentation.
 (See~\cite{gromov} and  also~\cite[Chapter III.$\Gamma$, Theorem 2.6.]{BH}.)

Solvability of the Word Problem was extended to all \emph{one-relator groups} by  Magnus~\cite{magnus} in 1932.
  We do not, however,  know any  bound on the complexity of solving the Word Problem  over the class of all one-relator groups.
A  theorem of  Newman~\cite{newman, LS} shows that any  one-relator presentation
 of the form $G = \langle X; w^n \rangle$ with  $n \ge 2$ is a Dehn presentation.

 It was independently shown by Novikov~\cite{novikov} in 1955 and by Boone~\cite{boone} in 1958 that there exist
finitely presented groups $G = \langle X; R \rangle$ with unsolvable  Word Problem.  In order to prove this basic result it is necessary to code the Halting Problem for Turing machines into the Word Problem of the group.
The unsolvability of the Word Problem is the foundation of all the unsolvability results in group theory and topology.

\subsection{The Dehn function}
Let
\begin{equation}
\label{e:dehn-function-pres}
G = \langle X;R\rangle
\end{equation}
be a finite presentation of $G$.
Let $\Sigma = X \cup X^{-1}$ denote the associated group alphabet and suppose that
$w \in \Sigma^*$ satisfies $w \approx \varepsilon$, that is, $w = 1_G$ in $G$.
This is equivalent to saying that the reduced form of $w$ belongs to the normal closure
$N$ of $R$ in $F_X$, the free group based on $X$.  This in turn is equivalent to the existence of  an  expression
\begin{equation}
\label{e:dehn-area}
w = u_1 r_1 u_1^{-1} u_2 r_2 u_2^{-1} \cdots u_m r_m u_m^{-1}
\end{equation}
where $m \in \N$, $u_i \in \Sigma^*$ and $r_i \in R^{\pm 1}$, $i=1,2,\ldots,m$.
Then  the \emph{area} of $w$ (with respect to the given presentation \eqref{e:dehn-function-pres}), denoted $\Area(w)$,
is the smallest $m \geq 0$ such that an  expression of the form above  holds.
The \emph{Dehn function} associated with the presentation \eqref{e:dehn-function-pres}
is the map $\Dehn \colon \N \to \N$ defined by
$$
\Dehn(n) = \max\{\Area(w): w \in \Sigma^*, w \approx \varepsilon, |w| \leq n\}.
$$

Cannon observed the following (see also \cite[Theorem 2.1]{gersten}).

\begin{theorem}
A finitely presented group presentation $G = \langle X ; R \rangle$ has a computable  Dehn function  if and only if the group $G$ has solvable Word Problem.
\end{theorem}

   It is not difficult to show that if $w = 1_G$ in $G$ then in an expression \eqref{e:dehn-area}
 the length of all the conjugating elements $u_i$ can be  bounded by $|w|$.  Thus if we can calculate $\Dehn(|w|) = b$ we
can try all possible products of the form \eqref{e:dehn-area} with $m \le b$ and all $|u_i| \le |w|$ and
check whether any of these products equals $w$ in the free group.

\subsection{The Word Problem as a formal language}

Anisimov~\cite{anisimov} in 1972 introduced the fruitful idea of viewing the Word Problem as a formal language,
a point of view which we now adopt.

\begin{definition}
Let $G = \langle X;R\rangle$ be a finitely generated group presentation.
The \emph{Word Problem} of $G$, relative to the given presentation, is the
language
$$
\WP(G:X;R) = \{w \in \Sigma^*: w \approx \varepsilon\},
$$
where  $\Sigma$ is the group alphabet as usual.
One says that the $G$ has \emph{regular} (resp. \emph{context-free}, resp. \emph{computable}) Word Problem with
respect to the given presentation if  $\WP(G:X;R)$ is a regular (resp. context-free, resp. computable) language.
\end{definition}

Note that, the Word Problem for a finitely generated group presentation $G = \langle X; R \rangle$ is solvable (in the sense of Dehn)
if and only if the language $\WP(G:X;R) \subset \Sigma^*$ is computable.

\begin{observation}[Invariance and finitely generated subgroups]
\label{r:invariance}
\rm{It is easy to see that the classification above of the Word Problem as a formal language is  actually a property of the group
and does not depend on the particular presentation considered. Indeed, the complexity of the Word Problem of a finitely generated
group bounds the complexity of the Word Problems of all its finitely generated subgroups.
For, suppose that $G = \langle X;R \rangle$ has a Word Problem of a given type and that
$H = \langle Y; S \rangle$ is a finitely generated presentation of a group isomorphic to a finitely generated subgroup  of $G$.
Let $\phi : H \to G$ be an injective homomorphism  and for $y \in Y$ denote by $w_{y} \in \Sigma^*$ a representative of the image  $\phi(y)$.
So, whether a finite automaton, pushdown automaton or Turing machine $\mathfrak{M}$ accepts the Word Problem for the first presentation,  we can
construct a machine $\mathfrak{M}'$ of the same type which, on reading a letter  $(y)^{\pm 1} \in Y \cup Y^{-1}$ simulates the sequence of
transitions of $\mathfrak{M}$ on reading the word $(w_{y})^{\pm 1} \in \Sigma^*$.}

As a consequence, we say that a finitely generated group $G$ is  \emph{context-free} provided
that the Word Problem $\WP(G:X;R)$ relative to some (equivalently, every) finitely generated presentation $G = \langle X; R \rangle$ is context-free.
\end{observation}

Anisimov~\cite{anisimov} characterized groups with regular Word Problem.

\begin{theorem}[Anisimov]
\label{t:anisimov}
Let $G = \langle X;R \rangle$ be a finitely generated group.
Then $G$ has regular Word Problem  if and only if $G$ is finite.
\end{theorem}

\begin{proof} Suppose that $G$ is finite. Consider the deterministic finite automaton $\A = (Q, \Sigma, \delta, q_0, F)$ where $Q = G$, $1_G$ is both the initial state $q_0$ and the unique element in $F$, and $\delta \colon Q \times \Sigma \to {\mathcal P}(Q)$ is given
by
$$
\delta(q,a) = qa
$$
for all $q \in Q$ and $a \in A$. Note that the graph underlying $\A$ (cf. Example \ref{e:underlying}) is the Cayley graph $\Gamma(G;X;R)$.
Then $\A$ accepts exactly the Word Problem of $G$.

Conversely, if $G$ is infinite there are arbitrarily long words $w \in \Sigma^*$
such that no nontrivial subword of $w$ is equal to the identity in $G$.
Suppose that $\A = (Q, \Sigma, \delta, q_0, F)$ is a deterministic finite automaton with alphabet the group alphabet $\Sigma$ and let $n$ denote the cardinality of its state set $Q$. Taking a word $w$ as above and
such that  $|w| \geq n+1$, then there exist $q \in Q$ and words $w_1, w_2, w_3 \in \Sigma^*$ satisfying
$w = w_1w_2w_3$ with $w_2 \neq \varepsilon$  such that $\A$, when reading $w$, is in the same state $q$
after reading the initial segment $w_1$ and after reading $w_1w_2$.  Then if $\A$ is in the state $q' \in Q$ after reading $w_1w_{1}^{-1}$  it is in  the same state after reading $w_1w_2w_{1}^{-1}$.  But the first word equals the identity in $G$
while the second word does not.  It follows that $\A$ cannot accept the Word Problem of $G$.
\end{proof}
\subsection{Context-free groups}
Recall that a group is context-free if it is finitely generated and its Word Problem
with respect to some (equivalently, every) finitely generated presentation is a
context-free language.

\begin{example}[The Word Problem for the free group]
\label{ex:free-cf}
Let $X$ be a finite set and let $G = F_X$ be the free group based on $X$.
Recall from Example~\ref{example:freegroup} that $G$ is in one-to-one correspondence with the set $L_{\text{\rm red}}(\Sigma)$ of all reduced words over
the alphabet $\Sigma = X \cup {X}^{-1}$. We adopt the convention that $\bar \varepsilon = \varepsilon$.
Consider the one-state deterministic pushdown automaton $\MM = (\{q_0\}, \Sigma, \Sigma,  \delta, q_0, \{q_0\}, \varepsilon)$. The automaton starts
with empty stack, accepts by empty stack, and the transition function is defined by
$$
\delta(q_0,a,z) =
\begin{cases}
(q_0,\varepsilon) & \mbox{ if } a = \bar z\\
(q_0,za) & \mbox{ otherwise}
\end{cases}
$$
for all $a,z \in \Sigma$.
It is clear that $L(\MM) = \WP(G:X;R)$, so the Word Problem for $G$ is context-free.
It follows that free groups are context-free.
\end{example}

For the next example we need the following well-known result (see~\cite[Lemma 6.1]{HU}).
\begin{lemma}
[The Pumping Lemma for context-free languages]
\label{l;pumping}
Let $\Sigma$ be a finite alphabet. Let $L \subset \Sigma^*$ be a context-free language.
Then there exists a positive integer $N = N(L)$ such that if $w \in L$ and $|w| \geq N$, then we can find $u,v,z,s,t \in \Sigma^*$ such that
$w = uvzst$, $|v|+|s| \geq 1, |vzs| \leq N$ and $uv^nzx^ny \in L$ for all $n \geq 0$.
\end{lemma}

With the notation from the above lemma, we  say that the word $uv^nzs^nt$ is obtained
from $w$ by \emph{pumping}  the subwords $v$ and $s$.

\begin{example}[The Word Problem for the free  abelian group of rank $2$]
\label{ex:free-abelian-not-cf}
Let $G = \Z^2$ with presentation $\langle x,y; [x,y]\rangle$.
Then $x^m y^m x^{-n} y^{-n} = 1$ in $G$ if and only if $m = n$.
We can now use Lemma~\ref{l;pumping}
to show that $L = \WP(G:X;R)$  is not context-free. Suppose by contradiction that $L$ is context-free and let $N = N(L)$ be the
corresponding positive integer. Consider the word
$w =x^{N+1}y^{N+1}x^{-(N+1)} y^{-(N+1)}$.
We clearly have $w \in L$. However, there are no subwords
$u,v,z,s,t$ of $w$ satisfying the conditions  described in the Pumping Lemma.
Indeed, from $|vzs| \leq N$ we deduce that $vzs$ is a subword of one of the following forms:
{\it (i)}~$x^m$, {\it (ii)} $x^py^q$, {\it (iii)} $y^h x^{-k}$, {\it (iv)} $x^{-p}y^{-q}$, or {\it (v)} $y^{-m}$, for suitable positive integers $m,p,q,h$ and $k$.
In all these cases, by pumping $n \geq 2$ times the subwords $v$ and $s$, we obtain a word $w'$ whose number of positive occurrences
of $x$ or of $y$ fails to equal the number of its negative occurrences  so that $w' \notin L$, contradicting the Pumping Lemma.
It follows that $L$ is not context-free. Therefore $\Z^2$ is not a context-free group.
\end{example}

\begin{proposition}
\label{p:finite-index-cf}
Let $G$ be a finitely generated group and $H$ a subgroup with $[G:H] < \infty$. Then $G$ is
context-free if and only if $H$ is context-free.
\end{proposition}
\begin{proof}
We have already seen the ``only if'' part in Observation~\ref{r:invariance}.
Conversely, let $H$ be a finite index subgroup of $G$ and suppose that it is context-free.
Recall the following general fact from group theory (sometimes called the Poincar\'e Lemma): a subgroup of finite index in a
finitely generated group $G$ contains a subgroup which is normal in $G$ and also of finite index,
 and which is therefore finitely generated.   Thus this  normal subgroup is also context-free
 if the ambient subgroup is context-free.
We can therefore  suppose that $H$ is normal in $G$.
Let $K = G/H$   be the corresponding finite quotient with  $\psi \colon G \to K$
the natural quotient map. Let  $K = \{k_1 = 1, k_2,\ldots,k_n\}$.
Let $H = \langle h_1, h_2,\ldots,h_m: R \rangle$ be a presentation of $H$.
Since $H$ is normal in $G$, if $\bar k_i \in G$ is such that $\psi(\bar k_i) = k_i$ we have relations of the form
$$
{\bar k_r} h_{j}^\eta \bar k_{r}^{-1} = w(r,j,\eta)
$$
where $\eta = \pm 1$ and $w(r,j,\eta)$ is a word in the generators $h_i$ and their
inverses.  Because $H$ is a normal subgroup we also have the relations
$$
\bar k_r\bar k_s = z(r,s)\bar k_{t(r,s)}
$$
where $z(r,s)$ is a word in the generators $h_i$ and their inverses
determined by the relation $k_r k_s = k_{t(r,s)}$ in the multiplication table of $K$.
So a presentation of $G$ is
\[
G = \langle \bar k_2,\ldots,\bar k_n, h_1,h_2,\ldots,h_m; {\bar k_r}^{-1} h_{j}^\eta \bar k_r = w(r,j,\eta),  \bar k_r\bar k_s = z(r,s)\bar k_{t(r,s)}, R \rangle,
\]
where $r,s = 2,\ldots, n$, $j = 1,2,\ldots,m$, and $\eta = \pm 1$.

\par
Let $\MM$ be a pushdown automaton accepting the Word Problem of $H$ for its presentation
above. The idea of constructing a pushdown automaton $\widehat \MM$ to accept the Word Problem of $G$ for the above presentation
is very simple. On reading a word $w$, the automaton $\widehat \MM$ uses extra master control states to keep track of the image
$\psi(w)$  in $K$ and uses the stack to simulate $\MM$ on the Word Problem of $H$.   The automaton $\widehat \MM$ starts with
empty stack in  the master control state  corresponding to $1_K$.
If  $\widehat \MM$ is in the master control state corresponding to $k_r$ and  $\widehat \MM$ reads a letter $\overline{k_s}$ it
uses a sequence of auxiliary  states to simulate $\MM$ reading the word $z(r,s)$ and then changes to the master control state
corresponding to $k_t$ where  $k_t = k_r k_s$ in $K$.
If $\widehat \MM$ reads a letter $h_j^{\eta}$ while in the master
control state corresponding to $k_r$ in the quotient group it uses a series of auxiliary states to simulate $\MM$ reading the word
$w(r,j,\eta)$.  Finally,  $\widehat \MM$ accepts by having empty stack and master control  state corresponding to $1_K$.
\end{proof}

\begin{definition}
A group $G$ is \emph{virtually free} if $G$ contains a free subgroup $H$ of finite index in $G$.
\end{definition}

\begin{corollary}
\label{c:vf-cf}
A finitely generated virtually free group is context-free.
\end{corollary}

\begin{proof} Let $G$ be a finitely generated virtually-free group and let $H \subset G$
be a free subgroup of finite index. Then $H$ is finitely generated and, as seen in Example
\ref{ex:free-cf}, context-free. By the ``if'' part of the previous proposition, we have that $G$ is context-free as well.
\end{proof}

Muller and Schupp~\cite{MS1} proved the following characterization
of groups with context-free Word Problem.
\begin{theorem}[Muller-Schupp]
\label{t:MS-characterization context-free}
Let $G$ be a finitely generated group. Then $G$ is context-free if and only
if $G$ is virtually free.
\end{theorem}

\begin{remark}{\rm
In \cite{CW} Ceccherini-Silberstein and Woess introduced and studied the concept of a
\emph{context-free pair of group}. Such a pair $(G,K)$ consists of a finitely generated group
$G = \langle X; R \rangle$ together with  a subgroup $K \subset G$ for which the language consisting of
all words over $\Sigma^* = X \cup X^{-1}$ representing an element in $K$ is context-free.
(When $K$ reduces to the identity element, this clearly specializes to the above definition
of $G$ to be a context-free group.) These investigations were extended by Woess in \cite{W-II} who applied them to the study of random walk asymptotics yielding a complete proof of the local limit theorem for return probabilities on any context-free group.}
\end{remark}

\subsection{Subgroups and embeddability}
We briefly mention some applications of formal language theory to subgroups and embeddability.

\begin{definition}
Let $G = \langle X;R \rangle$  be a finitely generated group with group alphabet
$\Sigma = X \cup X^{-1}$. Let $\psi \colon \Sigma^*  \to G$ be the natural map. Let $S \subset G$ be a subset.
 An \emph{enumeration} of $S$ is a subset $L \subset \Sigma^*$ such that $\psi(L) = S$. Then one says that
$L$ is a \emph{regular} (resp. \emph{context-free}, resp. \emph{computable}) enumeration provided that $L$
is  a regular (resp. context-free, resp. computably  enumerable) language.
\end{definition}

Anisimov and Seifert~\cite{AS} proved in 1975 the following theorem.

\begin{theorem}[Anisimov-Seifert]
\label{t:anisimov-seifert}
Let $G$ be a finitely generated group and let $H \subset G$
be a subgroup of $G$. Then $H$ is finitely generated if and only if $H$ has a regular enumeration.
\end{theorem}

Anisimov and Seifert also proved that context-free groups are finitely presentable, a fact  used in the proof of the characterization theorem.  The following more general result is  due to Frougny, Sakarovitch, and Schupp~\cite{FSS}.

\begin{theorem}[Frougny-Sakarovitch-Schupp]
Let $G$ be a finitely generated group and let $N \subset G$
be a normal subgroup of $G$. Then $N$ is finitely generated as a normal subgroup
(that is, $N$ equals the normal closure of a finite set of elements
of $G$) if and only if $N$ has a context-free enumeration.
\end{theorem}

\begin{definition} A \emph{computably enumerable presentation} (also called  a \emph{recursive presentation}) is a group presentation
$G = \langle X; R \rangle$ where the set $X$ of generators is finite and the set $R$ of defining relators is computably enumerable.
\end{definition}

Recall that a group $H$ is said to be \emph{embeddable} into a group $G$ provided there
exists an injective homomorphism $\psi \colon H \to G$.
The  remarkable Higman Embedding Theorem~\cite{higman} shows that the connection between group theory and computability is intrinsic.

\begin{theorem}[Higman]
A finitely generated group $H$ is embeddable into some finitely presented group  if and only if $H$ admits a
computably enumerable  presentation.
\end{theorem}

\subsection{Basic groups and simple languages}

We next consider a special subclass of deterministic context-free languages.

\begin{definition}
Let $\Sigma$ be a finite alphabet. A language $L \subset \Sigma^*$ is called
\emph{simple} if it is accepted by a $1$-state deterministic pushdown automaton
which accepts by empty stack and is required to halt when it empties its stack.
\end{definition}

The convention that the automaton accepting a simple language halts on empty stack
makes a simple language $L$ \emph{prefix-free}, that is, if $w= uv \in L $ with $u$ and $v$ nontrivial then $u \notin L$.
The main reference for simple languages is Harrison~\cite{harrison}.

Recall that given a language $L \subset \Sigma^*$, the \emph{Kleene star} of $L$ is the language $L^*$
over $\Sigma$ defined by
\begin{equation}
\label{e:kleene}
L^* = \{w_1w_2 \cdots w_n: w_i \in L \textup{ where } i=1,2,\ldots,n \textup{ and } n=0,1,2,\ldots\}.
\end{equation}
In other words, $L^*$ is the submonoid of $\Sigma^*$ generated by $L$.

Since we are interested in groups, the convention that the automaton must halt on empty stack
is rather unnatural.  Note that the language accepted by a $1$-state deterministic pushdown automaton which is
not required to halt on empty stack  is the Kleene star $L^*$, of a simple language $L$ (see Equation \eqref{e:kleene}).

\begin{example}
\label{e:mult-table-basic}
We show that the Word Problem for a finite group with respect to its multiplication table presentation is the Kleene star
of a simple language.
Let $G = \{g_1,g_2, \ldots, g_n\}$ (with $g_1 = 1_G$) be a finite group and consider its multiplication table presentation
\[
G = \langle g_2, \ldots, g_n; g_i g_j = g_{k(i,j)}\rangle
\]
(see Example~\ref{ex:mult-table}.(a)). Let $\MM$ be the deterministic single state pushdown automaton whose input alphabet
and stack alphabet are the  set $\{g_2,\ldots,g_n\}$ of non-identity elements of $G$. The automaton $\MM$ starts with
empty stack and will always have at most one symbol on the stack. If the stack is empty and $\MM$ reads $g_i$ then $\MM$
puts $g_i$ on the stack. If the symbol on the stack is $g_i$ and $\MM$ reads $g_j$ then $\MM$ replaces $g_i$ by the
product $g_{k(i,j)}$  if $g_i g_j $ is not the identity of $G$ and $\MM$ empties the stack otherwise.  It is clear that $\MM$ has empty stack
exactly when the product of the elements it has read so far is the identity, so
$  L(\MM)^* = \WP(G:g_2, \ldots, g_n; g_i g_j = g_{k(i,j)}) $.
\end{example}

\begin{definition}
A group $G$ is called \emph{basic} if it is the free product of
finitely many finite groups and  a free group of finite rank, i.e,
$G \cong G_1*G_2 * \cdots *G_k * F_n$, where $G_i$ is a finite group, $i=1,2,\ldots,k$, and $F_n$ is
the free group of rank $n$, with $k,n \geq 0$.
\end{definition}

Note that finite groups and finitely generated free groups are basic groups.
We saw in Example \ref{e:mult-table-basic} that the Word Problem of a finite group  with respect to the multiplication table presentation is the star of a simple language.
Analogously, it follows from Example \ref{ex:free-cf} that the Word Problem of a finitely generated free group with respect to the free presentation is the star of a simple language
as well.
More generally, if we take the ``canonical  presentation''  of a basic group given by the disjoint union of the multiplication table presentations of the finite factors and the free presentation of the free group, then the corresponding Word Problem is the star of a simple language.  In general, however, having a Word Problem which is the Kleene star of a simple language depends on the given presentation.
We give an example below (Example~\ref{bad-Z}).

Haring-Smith~\cite{HS} characterized groups whose Word Problem is the star of a simple language.

\begin{theorem}[Haring-Smith]
A finitely generated group $G$ is basic if and only if it
has a finitely generated presentation $G = \langle X; R \rangle$ such that
the corresponding Word Problem is the Kleene star of a simple language.
\end{theorem}

Haring-Smith~\cite{HS} also gave the following geometric characterization of basic groups.

\begin{theorem}[Haring-Smith]
A group $G$ is basic if and only if $G$ has a finitely generated presentation such that in the corresponding Cayley graph  $\Gamma $ the following holds: for every vertex $v \in V(\Gamma)$ there are only finitely many cycles trough $v$.
\end{theorem}

Indeed, the Word Problem for a given presentation is the star of a simple language if and only if its Cayley graph satisfies the above geometric condition.

\begin{example}[The modular group]
\label{modular}
A presentation of the modular group  $G = \PSL(2,\Z)$ $\cong$ $(\Z/2\Z) * (\Z/3\Z)$ is
$G = \langle x,y; x^2, y^3 \rangle$.
The corresponding Cayley graph $\Gamma$ is represented in Figure~\ref{CGC2*C3}.
\begin{figure}[h!]
\begin{center}
\gasset{Nframe=n,AHnb=0,Nadjust=w}
\begin{picture}(0,77)(0,-36)
\node(1)(-7,0){$1$}
\node(x)(7,0){$x$}
\node(y)(-18,10){$y$}
\node(y-)(-18,-10){$y^2$}
\node(xy)(18,10){$xy$}
\node(xy-)(18,-10){$xy^2$}
\node(xyx)(29,20){$xyx$}
\node(xy-x)(29,-20){$xy^2x$}
\node(yx)(-29,20){$yx$}
\node(y-x)(-29,-20){$y^2x$}
\node(xyxy)(40,30){$xyxy$}
\node(xyxy-)(40,10){$xyxy^2$}
\node(xy-xy)(40,-10){$xy^2xy$}
\node(xy-xy-)(40,-30){$xy^2x y^2$}

\node(yxy)(-40,30){$yxy$}
\node(yxy-)(-40,10){$yxy^2$}
\node(y-xy)(-40,-10){$y^2xy$}
\node(y-xy-)(-40,-30){$y^2x y^2$}

{\footnotesize
\gasset{AHnb=1,AHangle=30,AHLength=1,AHlength=0}
\drawedge(1,x){$x$}
\drawedge[ELside=r](1,y){$y$}
\drawedge[ELside=r](y-,1){$y$}
\drawedge(x,xy){$y$}
\drawedge(xy-,x){$y$}
\drawedge[ELside=r](y,y-){$y$}
\drawedge[ELside=r](y,yx){$x$}
\drawedge(xy,xy-){$y$}
\drawedge(xy,xyx){$x$}
\drawedge[ELside=r](y-,y-x){$x$}
\drawedge(xy-,xy-x){$x$}
\drawedge(xyx,xyxy){$y$}
\drawedge(xyxy-,xyx){$y$}
\drawedge(xyxy,xyxy-){$y$}
\drawedge(xy-x,xy-xy){$y$}
\drawedge(xy-xy-,xy-x){$y$}
\drawedge(xy-xy,xy-xy-){$y$}
\drawedge[ELside=r](yx,yxy){$y$}
\drawedge[ELside=r](yxy-,yx){$y$}
\drawedge[ELside=r](yxy,yxy-){$y$}
\drawedge[ELside=r](y-x,y-xy){$y$}
\drawedge[ELside=r](y-xy-,y-x){$y$}
\drawedge[ELside=r](y-xy,y-xy-){$y$}
}
\gasset{Nframe=n,Nadjust=w,Nadjust=h,Nadjustdist=-1,AHnb=0,Nh=0}
\node(2)(46,36){}
\node(3)(46,4){}
\node(4)(46,-4){}
\node(5)(46,-36){}
\node(6)(-46,36){}
\node(7)(-46,4){}
\node(8)(-46,-4){}
\node(9)(-46,-36){}

\drawedge[dash={0.4}0](xyxy,2){}
\drawedge[dash={0.4}0](xyxy-,3){}
\drawedge[dash={0.4}0](xy-xy,4){}
\drawedge[dash={0.4}0](xy-xy-,5){}
\drawedge[dash={0.4}0](yxy,6){}
\drawedge[dash={0.4}0](yxy-,7){}
\drawedge[dash={0.4}0](y-xy,8){}
\drawedge[dash={0.4}0](y-xy-,9){}

\end{picture}
\end{center}
\caption{The Cayley graph of the modular group $\mathbb{Z}/2\mathbb{Z} * \mathbb{Z}/3\mathbb{Z}
 = \langle x, y ; x^2 = y^3 = 1 \rangle$}\label{CGC2*C3}
\end{figure}

As illustrated in Figure~\ref{CGC2*C3}, for every vertex $v \in V(\Gamma)$ there are exactly two cycles through $v$, namely $(e_1,e_2,e_3)$ and $({e_3}^{-1},{e_2}^{-1}, {e_1}^{-1})$, where $e_1 = (v,y,vy)$, $e_2 = (vy,y,vy^2)$, and $e_3 = (vy^2,y,v)$. As usual, for an edge $e$ we denote by $e^{-1}$ the opposite edge (see the drawing convention for symmetric labelled graphs at page~\pageref{drawing convention}).
\end{example}

\begin{example}
\label{bad-Z}
Consider the presentation $\langle x,y;  y = x^2 \rangle$ of the infinite cyclic group.  In the Cayley graph
of this presentation there are infinitely many cycles through a vertex (see Figure~\ref{CGZ'}) and the Word Problem for this
presentation is not the star of a simple language.
\end{example}

\begin{figure}[h!]
\begin{center}
\gasset{Nframe=n,AHnb=0,Nadjust=w}
\begin{picture}(0,15)(0,-5)
\node(1)(0,0){$1$}
\node(x)(13,0){$x$}
\node(xx)(26,0){$x^2$}
\node(xxx)(39,0){$x^3$}
\node(xxxx)(52,0){$x^4$}

\node(x-)(-13,0){$x^{-1}$}
\node(x-x-)(-26,0){$x^{-2}$}
\node(x-x-x-)(-39,0){$x^{-3}$}
\node(x-x-x-x-)(-52,0){$x^{-4}$}

\node(10)(-60,6){}
\node(11)(-60,0){}
\node(12)(60,0){}
\node(13)(60,6){}

\drawedge[dash={0.4}0](11,x-x-x-x-){}
\drawedge[dash={0.4}0](12,xxxx){}
\drawedge[dash={0.4}0,curvedepth=-6](xxx,12){}
\drawedge[dash={0.4}0,curvedepth=6](x-x-x-,11){}
\drawedge[dash={0.4}0,curvedepth=2](xxxx,13){}
\drawedge[dash={0.4}0,curvedepth=-2](x-x-x-x-,10){}

{\footnotesize
\gasset{AHnb=1,AHangle=30,AHLength=1,AHlength=0}
\drawedge(x-,1){$x$}
\drawedge(x-x-,x-){$x$}
\drawedge(x-x-x-,x-x-){$x$}
\drawedge(x-x-x-x-,x-x-x-){$x$}
\drawedge(1,x){$x$}
\drawedge(x,xx){$x$}
\drawedge(xx,xxx){$x$}
\drawedge(xxx,xxxx){$x$}
\drawedge[curvedepth=6](1,xx){$y$}
\drawedge[curvedepth=-6,ELside=r](x,xxx){$y$}
\drawedge[curvedepth=6](xx,xxxx){$y$}
\drawedge[curvedepth=-6,ELside=r](x-,x){$y$}
\drawedge[curvedepth=6](x-x-,1){$y$}
\drawedge[curvedepth=-6,ELside=r](x-x-x-,x-){$y$}
\drawedge[curvedepth=6](x-x-x-x-,x-x-){$y$}
}
\end{picture}
\end{center}
\caption{The Cayley graph of the group $\mathbb{Z} = \langle x, y ; y = x^2\rangle$}\label{CGZ'}
\end{figure}

\section{Finitely generated graphs and ends}\label{sec:FGG}

\subsection{Finitely generated graphs}
We need a framework in which we can  discuss both Cayley graphs of finitely generated groups and complete transition graphs
of pushdown automata.  The following definition is from~\cite{MS2}.

\begin{definition}
A \emph{finitely generated graph} is a rooted labelled graph
$\Gamma = (V,E,\Sigma,v_0)$ with a uniform upper bound on the degrees of vertices, and which
is connected from $v_0$, that is, for every vertex $v \in V$, there is a directed path from $v_0$ to $v$.
\end{definition}

The Cayley graph of a finitely generated group is  clearly a finitely generated graph.  Other examples of finitely generated graphs are provided by the
complete transition graph of pushdown automata that we now define.

\begin{definition}[The complete transition  graph of a pushdown automaton]
Let $\MM = (Q,\Sigma,Z, \delta,$ $q_0,$ $F, z_0)$ be a pushdown automaton.
The \emph{complete transition graph} of $\MM$ is the labelled graph $\Gamma = \Gamma(\MM)$
defined as follows.
The initial vertex is the initial configuration $(q_0,z_0)$.
The vertex set $V$ is the subset of $Q \times Z^*$ consisting of all configurations
$(q,\zeta)$ which are reachable from the initial configuration on reading some possible input $w \in \Sigma^*$. In our previous notation,
$$V =  \{(q,\zeta) : \MM \underset{w}{\overset{*}\vdash} (q, \zeta), w \in \Sigma^* \}.$$
If $v = (q, \zeta)$ and $v' = (q', \zeta')$
are two vertices, then there is an oriented edge labelled by $a \in \Sigma$ from $v$
to $v'$ if and only if $\zeta = \zeta_0z$ with $z \in Z$ such that there exists $(q',\zeta_1) \in \delta(q,a,z)$ satisfying $\zeta' = \zeta_0\zeta_1$.
\end{definition}

Note that $\Gamma(\MM)$ is connected from $v_0$ by definition and that there is an
upper bound on the degrees of vertices.
Thus  the complete transition  graph $\Gamma(\MM)$ of a pushdown automaton $\MM$ is a finitely generated graph.

\begin{example}
Consider the deterministic pushdown automaton $\MM = (Q,\Sigma,Z,$ $\delta,$ $q_0,$ $F, z_0)$
with $Q = F = \{q_0\}$, $\Sigma = Z = \{0,1\}$, $z_0 = \varepsilon$,
and transition function defined by
$$\delta(q_0,a,z) = (q_0,za)$$
for all $a, z \in \{0,1\}$. Then the associated complete transition graph of $\MM$ is isomorphic to the rooted infinite binary  tree $T_2$ (see Figure~\ref{binaryTREE}).
\end{example}

\subsection{Ends of finitely generated graphs}
Let $\Gamma = (V,E, \Sigma,v_0)$ be a finitely generated graph.  Intuitively, an \emph{end} of $\Gamma$ is a way to
``go to infinity'' in $\Gamma$.  Although  a finitely generated graph is a directed graph,  in order to discuss
ends,  we need to consider undirected paths.  Let $\Gamma'$ be the graph obtained by considering $\Gamma$ as an undirected
graph. So if $(u,\sigma, v)$ is an edge of $\Gamma$ then both $(u,v)$ and $(v,u)$ are edges of $\Gamma'$. In short,
one now  ignores labels and  the orientation of  edges.
An  \emph{undirected path} in $\Gamma$ is a sequence of edges $(u_1,v_1),(v_1,v_2),\ldots,(v_i,v_{i+1}), \ldots,(v_n,v_{n+1})$ forming a path in $\Gamma'$.
\par
For a non-negative integer $n$, we denote by $\Gamma_n$ the subgraph of $\Gamma$ whose vertex set $V_n$ consists
of all vertices $v \in V$ such that there exists an  \emph{undirected path}  $\pi $ with $|\pi| \leq n$ from the
origin $v_0$ to $v$ and whose edge set  consists of the edges of $\Gamma$ between two
such vertices.  $\Gamma_n$  is called the \emph{ball} of radius $n$ centered at the basepoint $v_0$ of $\Gamma$.

Let $n$ be a non-negative integer. It follows from the finiteness of the degrees of the
vertices of $\Gamma$ that there are only finitely many connected components of $\Gamma \setminus \Gamma_n$.
 Let us denote them by $\Gamma_{n,1}, \Gamma_{n,2}, \ldots, \Gamma_{n,k(n)}$.
Let $e(n)$ be the number of \emph{infinite} connected components of
$\Gamma \setminus \Gamma_n$. Note that $0 \leq e(n) \leq k(n)$. Moreover, it is easy to see that $e(n)$ is a non-decreasing function of $n$. Thus the following limit exists in $\R \cup \{\infty\}$:
$$
e(\Gamma) = \lim_{n \to \infty} e(n).
$$
It is called the \emph{number of ends} of $\Gamma$.

\begin{example}
\label{e;ends}
\begin{enumerate}[(a)]
\item Let $\Gamma$ be a finite graph and fix an arbitrary vertex $v_0 \in V(\Gamma)$.
For every $n \geq 0$ one has $\Gamma \setminus \Gamma_n$ is finite and, in particular, has
no infinite connected components, that is, $e(n) = 0$. It follows that $e(\Gamma) = 0$.

\item  Let $\Gamma = T_2$ be the rooted infinite binary  tree.
Then for every non-negative integer $n$, the vertex set of the ball of radius $n$ centered at $v_0 = \varepsilon$ consists of  all words
in $\{0,1\}^*$ having  length at most $n$. Each connected component of $\Gamma \setminus \Gamma_n$  has vertex subset $V_w \subset V$
consisting of all words in $\{0,1\}^*$ with proper prefix $w$, where $w \in \{0,1\}^n$ is a word of length $n$. Since there are $2^n$ distinct words of length $n$ over the alphabet $\{0,1\}$, we have $e(n) = 2^n$ for all $n \geq 0$,
so that  $e(\Gamma) = \infty$.

\item  Let $\Gamma$ be the Cayley graph of the infinite cyclic group $\Z = \langle x \rangle$. Then for every non-negative integer $n$ the ball of radius $n$ centered at $v_0 = 1_{\mathbb Z} $ is the ``interval''
from $x^{-n}$ to  $x^{n}$. Thus, $\Gamma \setminus \Gamma_n$ consists of the two disjoint intervals $C_{< n} = \{x^m: m < -n\}$ and
$C_{>n} = \{x^m: m > n\}$. Thus $e(n) = 2$ for all $n \geq 1$ so that
$e(\Gamma) = 2$.

\item  Let $\Gamma$ be the Cayley graph of $\Z^2$ with respect to the presentation $\Z^2 = \langle x,y ; [x,y]\rangle$.
Then for every non-negative integer $n$ the ball of radius $n$ centered at the origin  is the ``square''
$\Gamma_n = \{x^py^q: |p|+|q| \leq n\}$. Thus, $\Gamma \setminus \Gamma_n$
consists of a single connected component, namely $C_{>n} = \{x^py^q: |p|+|q| > n\}$. Hence $e(n) = 1$ for all $n \geq 0$ so that $e(\Gamma) = 1$.
\end{enumerate}
\end{example}

\begin{remark}
If $G$ is a finitely generated  group then the number of ends of the Cayley graph of any finitely generated presentation of $G$ is the same.
Thus $e(G)$, the number of ends of $G$, is well defined and does not depend on the presentation. It is a fact that the number of ends
of any  finitely generated group is either $0,1,2$, or $\infty$.
We also remark that if $e(G) = \infty$ then $G$ contains nonabelian free groups
(see, e.g. \cite{karlsson, karlsson2, moon, woess1}).
\end{remark}

A very powerful result of  Stallings~\cite{stallings} is the Stallings Structure Theorem.

\begin{theorem}[Stallings]
Let $G$ be a  finitely generated group. Then $e(G) > 1$ if and only if one of the following
holds:
\begin{itemize}
\item $G$ admits a splitting $G = H*_C K$ as a free product with amalgamation, where $C$
is a finite proper subgroup of both $H$ and $K$;
\item $G$ admits a splitting  $G = \langle H,t; tC_1t^{-1} = C_2 \rangle$  as an HNN-extension, where $C_1$ and $C_2$ are
isomorphic finite subgroups of $H$.
\end{itemize}
\end{theorem}

The proof of the characterization of context-free groups as
finitely generated virtually free groups depends heavily on the Stallings Structure Theorem. A consequence of the geometric
characterization of context-free groups is that every finitely generated subgroup of a context-free group is either finite
or has more than one end. This opens the way to a proof by induction but needs the notion of accessibility.
A finitely generated group is \emph{accessible} if the process of taking repeated splittings as in Stallings' theorem must halt
after a finite number of steps.  That is, one splits $G$ as $H *_C K$ or as an HNN-extension $\langle H,t: tC_1t^{-1} = C_2\rangle$
according to the theorem and then splits $H$ and $K$ or just $H$ in the HNN case, etc.  Accessibility of context-free groups
is needed  to complete the characterization of context-free groups as virtually-free groups. (See Theorem~\ref{t:MS-characterization context-free}.)

Senizergues~\cite{senizergues2} proved the following result.

\begin{theorem}[Senizergues]
If $G$ is a context-free group then there are only finitely many conjugacy classes of finite subgroups of $G$.
\end{theorem}

Linnell~\cite{linnel} proved that any finitely generated group with only finitely many conjugacy classes of finite subgroups
is accessible. In conjunction with Senizergues' theorem this shows  that any context-free group is accessible.
 Dunwoody~\cite{dunwoody1} later proved that \emph{all} finitely presentable groups are accessible.  Recall that Anisimov and Seifert
proved that context-free groups are finitely presentable, (See the comments after Theorem~\ref{t:anisimov-seifert}.
Note that there exist finitely generated groups that are not accessible
(see~\cite{dunwoody2}).

\subsection {Graphs with finitary end structure}
We have seen that $e(\Z^2) = 1$ while $e(T_2) = \infty$.
Later, in the section on monadic logic, we shall see that there is a precise sense in which, from the point of view of logical complexity,
the Cayley graph of $\mathbb Z^2$ is infinitely more complicated than the rooted infinite binary  tree $T_2$.  So the \emph{number} of ends
is not a good measure of logical complexity but it turns out that we can still use ends to measure  complexity.

\begin{definition} Let $\Gamma$ be a finitely generated graph.
Denote by $c(\Gamma)$ the number of \emph{end-isomorphism classes} of connected components of $\Gamma \setminus \Gamma_n$ over \emph{all} components and all $n \geq 1$.   An \emph{end-isomorphism} between connected  components $C$ of $\Gamma \setminus \Gamma_n$
and $C'$ of $\Gamma \setminus \Gamma_{n'}$ is a labelled graph isomorphism which additionally maps the points of $\Gamma_n$ at distance $n$
from $v_0$ to  the points of $\Gamma_{n'}$ at distance $n'$ from $v_0$
(thus respecting the end structure).  Note that although we
undirected the graph to define the connected components, we are using the directed structure of $\Gamma$ to define end-isomorphisms.
\end{definition}

\begin{example} (compare with Example~\ref{e;ends}).
\begin{enumerate}[(a)]
\item Let $\Gamma$ be a finite graph. The number of all connected components of $\Gamma \setminus \Gamma_n$, $n \geq 1$, equals the number of all connected components of
$\Gamma \setminus \Gamma_1, \Gamma \setminus \Gamma_2, \ldots, \Gamma \setminus \Gamma_{d-1}$, where $d = \max\{\dist(v,v_0): v \in V(\Gamma)\}$, and is therefore finite.
It follows  that $c(\Gamma) < \infty$.

\item Let $\Gamma = T_2$ be the rooted infinite binary tree, say with label $0$ on left successor edges and label $1$ on right successor edges.
Then for every $n \in \N$ and every component $C$ of $\Gamma \setminus \Gamma_n$ the graph  $C$ is a rooted infinite binary  tree
isomorphic to $\Gamma$.  Thus  $c(\Gamma) = 1$.

\item Let $\Gamma$ be the Cayley graph of $\Z$ with respect to the standard presentation.
Recall that $\Gamma$ is the infinite line (see Figure~\ref{CGZ})
with a directed edge labelled by $x$ from vertex $x^n$ to vertex $x^{n+1} $ for all $n \in \mathbb Z$. If we remove a ball
$\Gamma_r, r \ge 1$, then there are always two components. Call these components the ``left'' component and the ``right `` component.
These two components  are not isomorphic as labelled graphs  since edges with label $x$ go from  vertex $x^n$ to vertex $x^{n+1}$.
However, all right components are isomorphic to each other and all left components are isomorphic to each other.  Thus $c(\Gamma) = 2$.

\item Let $\Gamma$ be the Cayley graph of $\Z^2$ with  presentation $\langle x,y; [x,y]\rangle$ (see Figure~\ref{CGZ2}).
Then, for every non-negative integer $n$ the ball of radius $n$ centered at the identity  is
the ``square''  $\Gamma_n = \{x^p y^{q}: |p|+|q| \leq n\}$. It is clear that
the graphs $\Gamma \setminus \Gamma_n$ are pairwise non-isomorphic  (look at the finite
boundaries!) so that $c(\Gamma) = \infty$.
\end{enumerate}
\end{example}

\begin{definition} A   finitely generated graph $\Gamma$ has \emph{finitary end-structure} if   $c(\Gamma) < \infty$.
A finitely generated graph is \emph{context-free}
if there exists a pushdown automaton $\MM$ such that $\Gamma$ is label-isomorphic to $\Gamma(\MM)$.
\end{definition}

It turns out that there is a characterization of finitely generated graphs with finitary end-structure.

\begin{theorem}[Muller-Schupp]
Let $\Gamma$ be a finitely generated graph. Then $\Gamma$ has finitary end-structure if and only if
$\Gamma$ is context-free.
\end{theorem}

The necessary condition of the theorem is the ``easy part'' while the sufficient condition
is  ``hard''.  An analysis of the proof shows that finitely generated graphs
$\Gamma$ with $c(\Gamma) < \infty$ are ``very treelike'' (see also \cite{woess3}).
Indeed, $\Gamma$ contains a rational subtree of finite index in the sense that there is a subtree $T$ of $\Gamma$ defined by a finite automaton
such that every vertex of $\Gamma$ is within a fixed distance from some vertex of $T$.  Putting the characterization of graphs
$\Gamma$ with $c(\Gamma) < \infty$ together with the characterization of context-free groups we have the following result.

\begin{corollary}
Let $G$ be a finitely generated group and let $\Gamma$ be the Cayley graph of any finitely generated presentation of $G$.
Then $c(\Gamma) < \infty$ if and only if $G$ is virtually free.
\end{corollary}

\section{Second-order monadic logic, the Domino Problem, and decidability}
\label{s:SOML}
\subsection{Second-order monadic logic and the theorems of B\"uchi and Rabin}

The reader is probably familiar with \emph{first-order logic} in which  the quantifiers
$\exists$ (there exists) and $\forall$ (for all) range only over individual elements of a given structure. The first-order language
for a structure includes the quantifiers, variables $x,y,z,\ldots$ for individual elements and the Boolean connectives  $\neg$ (negation),
$\lor$ (or), and $\land$ (and). There are  function and relation symbols for the operations  and relations of the structure,
including the relation of equality.
For more on first-order logic see  the monograph  by Enderton~\cite{enderton}.

\begin{example}[Group axioms]
 The usual  axioms which define a group are expressible in first-order logic.
A quadruple $\langle G,\ast, ^{-1} ,1_G \rangle $, where $G$ is a set with  a binary function symbol  $\ast$,  a unary function symbol  $^{-1}$,
and a 0-ary  constant symbol $1_G$,  defines a group provided that:
\begin{itemize}
\item  $\forall x \forall y \forall z [(x \ast y ) \ast z  = x \ast ( y \ast z)]$ (associative property);
\item  $\forall x [x \ast 1_G = 1_G \ast x = x]$ (existence of an identity element);
\item $\forall x [x \ast x^{-1} = x^{-1} \ast x = 1_G]$ (existence of inverse elements).
\end{itemize}
\end{example}

In \emph{monadic second-order} logic, one also has variables and quantifiers ranging over arbitrary subsets of the structure.
The term ``monadic'' refers to the fact that we can quantify only over subsets of the given structure, and not over relations.
Second-order logic with  variables for arbitrary relations is sometimes called \emph{full} second-order logic to distinguish
it from the monadic version.

\begin{example}[Peano axioms]
 Consider  the language of  \emph{second-order Peano axioms} for \emph{arithmetic} in which we have a unary function symbol $s$
for the successor function, a constant symbol $0$, the set membership symbol $\in$, the relation $\subseteq$ of set inclusion,
and equality relation for both individual and set variables.   The axioms are:

\begin{itemize}
\item $ \forall x \lnot [s(x) = 0 ]$
\item $ \forall y \exists x [ y \ne 0 \Rightarrow   y = s(x) ]$
\item $ \forall x \forall y [s(x) = s(y) \Rightarrow  x=y]$
\item $ \forall X [[0 \in  X  \land  \forall x(x \in X \Rightarrow s(x) \in X)] \Rightarrow  \forall y [y \in X]]$ \ (mathematical induction).
\end{itemize}

In standard second-order logic, these axioms define $\mathbb N$ with the successor function up to isomorphism.
This theory is sometimes denoted by $S1S$, the \emph{theory of one successor function}.
\end{example}

B\"uchi~\cite{buchi} introduced the theory of finite automata on infinite inputs to prove
the following result.

\begin{theorem}[B\"uchi]
\label{t:buchi}
The monadic second-order theory $S1S$ is decidable.
\end{theorem}

We next want to consider the monadic theory $S2S$ of two successor functions, that is, the monadic theory of the rooted infinite binary
 tree $T_2$.   Individual variables and quantifiers can actually be eliminated since when a set has exactly
one element is definable in the logic and we often adopt this point of view.
Also,  equality between sets is definable in terms of set inclusion.
The set of vertices of the  rooted infinite binary  tree  $T_2$ can  be viewed as the set $\{0,1\}^*$ of all finite words on $\{0, 1\}$.
We  have a constant for the root of the tree (which corresponds to the empty word $\varepsilon$) and
two set-valued  successor functions, $0$ and $1$.  If $S$ denotes a set of vertices then
$$ S0 = \{v0: v \in S\} \mbox{ \ and \ } S1 = \{v1: v \in S \}.   $$
We also have  the binary relation $\subseteq$ of set inclusion.

In 1969  Rabin~\cite{rabin} developed the theory of finite automata working on infinite trees and proved the following result.

\begin{theorem}[Rabin]
\label{t:rabin}
The monadic second-order theory $S2S$ is decidable.
\end{theorem}

As a consequence of Rabin's  theorem, the monadic second-order theory $SnS$ of $n$ successor functions is also decidable since it
can be interpreted in $S2S$. Note that the above theories are about the geometry of the underlying graph.
Analogously then, we can define the second-order monadic theory of any finitely generated graph $\Gamma = (V,E,\Sigma, v_0)$. We thus have again a constant for the origin of the graph $v_0$ and for each $a \in \Sigma$ we have a set-valued successor function where $Sa  = \{v \in V: \exists u \in S$ such that $(u,a,v) \in E\}$
for all $S \subset V$.

\subsection{The Domino Problem}
\label{sub-domino}
Rabin's theorem is one of the most remarkable positive results on decidability.
An important negative  result  is the unsolvability of the \emph{Wang Domino Problem} in the plane.  Whether or not it is possible
to tile the plane with copies of a fixed finite set of square tiles with colored edges was a question raised by Wang~\cite{wang} in the late 1950s.
Of course, when one places a tile next to another one,  the colors on the matching edges must be the same.  Wang
showed that the \emph{origin-constrained} problem is undecidable.  In this version there is a fixed  initial tile which must be
used first.   Indeed, fixing one tile is enough to show  that one can directly simulate the Halting Problem
for Turing machines in this context. Given a Turing machine $\T$ one can write down a set of tiles such that one can tile the entire plane
if and only if $\T$ halts when started with a blank tape.
The general Tiling Problem without an origin constraint  was proved undecidable
by  Berger \cite{berger} in 1966.   In 1971,  Robinson~\cite{robinson} found a simpler proof of the undecidability of the
general problem in the Euclidean plane.

This problem can be reformulated in terms of coloring vertices as follows.  Let $\Gamma$ be the Cayley graph of the standard presentation
$\mathbb Z^2 = \langle x,y ;  [x,y]\rangle$ of the free abelian group of rank $2$.  Let $C = \{c_1,c_2,\ldots,c_k\}$ be a finite
set of \emph{colors}.  The \emph{standard neighborhood} of a vertex $v$
in $\Gamma$ consists of $v$ and its four neighbors: $vx, vx^{-1}, vy$, and $vy^{-1}$ (see Figure~\ref{neighZ2}).

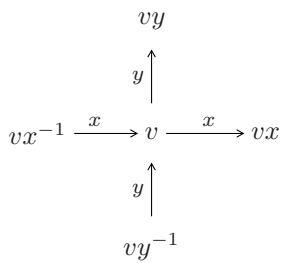
\begin{figure}[h!]
\begin{center}
\gasset{Nframe=n,AHnb=0,Nadjust=w}

\begin{picture}(0,35)(0,-15)
\node(1)(0,0){$v$}
\node(x)(15,0){$vx$}
\node(y)(0,15){$vy$}
\node(x-)(-15,0){$vx^{-1}$}
\node(y-)(0,-15){$vy^{-1}$}

{\footnotesize
\gasset{AHnb=1,AHangle=30,AHLength=1,AHlength=0}
\drawedge(x-,1){$x$}
\drawedge(1,x){$x$}
\drawedge(1,y){$y$}
\drawedge(y-,1){$y$}
}
\end{picture}
\end{center}
\caption{The standard neighborhood of a vertex $v$ in the Cayley graph of
$\mathbb Z^2 = \langle x,y ;  [x,y]\rangle$.}\label{neighZ2}
\end{figure}

We are also given a set $\mathcal F$ of \emph{forbidden patterns} where a \emph{pattern} $p \in {C}^5$ is a coloring of the vertices of the
standard neighborhood with colors from $C$.  The \emph{Domino Problem} for $\mathbb Z^2$ is the following decision problem:
given a pair $(C, \mathcal F)$ as above,
can all the vertices of the Cayley graph $\Gamma$ be colored so that there are no forbidden patterns? Note that since $\Gamma$ can be viewed as the dual graph of the tessellation by squares, this version is easily seen to be equivalent to the original formulation in terms of square tiles.

Our reformulation of the Domino Problem applies to an arbitrary finitely generated group $G$.
Also, the Domino Problem is easily expressible in terms of the monadic second-order logic of the Cayley graph $\Gamma$ of $G$ with respect to the given presentation.
A tuple $(C_1, C_2, \dots, C_k)$ of sets of elements of $G$ is a \emph{disjoint cover of $G$} if every element of $G$
belongs to exactly one of the $C_i$. (A disjoint cover differs from a partition only in
that some of the $C_i$ may be empty.)
We need only say that there is a disjoint cover  $(C_1,C_2, \ldots,C_ k)$ of the
vertices corresponding to the colors $c_1,c_2,\ldots,c_k$ such that there are no forbidden patterns. For example, if  the $i$-th pattern in $\mathcal F$ centered at $v$ has color $c_v$ at $v$ and colors $c_x, c_{x^{-1}}, c_y$,
and $c_{y^{-1}}$  at $vx,vx^{-1}, vy$, and $vy^{-1}$ respectively, we abbreviate this as $p_i$, and we must say that such a
pattern does not occur.
We can write this as:
\[
\exists C_1\exists C_2 \cdots \exists C_k \forall v [[\bigvee_{i} [v \in C_i]  \land [\bigwedge_{i < j} [v \in C_i \Rightarrow v \notin C_j ]]
 \land [\bigwedge_{p_i \in \mathcal F} \lnot p_i]].
\]

Note that from the point of view of logical complexity, measured in terms of alternation of quantifiers, the sentence above is very simple.
It consists of one block of existential set quantifiers followed by one universal individual quantifier and such sentences are already
undecidable.  There is thus a precise sense in which the monadic logic of the Cayley graph of  $\mathbb Z^2$ is infinitely more complicated
than the monadic logic of the infinite binary tree, where the entire monadic theory is decidable.


Recently,  Margenstern~\cite{margenstern} (see also~\cite{margenstern1} for a
shorter account) proved that the general Tiling Problem of the \emph{hyperbolic plane} is undecidable by
 using  a regular polygon as the basic shape of the tiles.  Robinson raised this problem  in the above mentioned paper
and  in 1978 he proved that the  origin-constrained problem is undecidable for the hyperbolic plane~\cite{robinson2}.
The fundamental group of a closed orientable surface of genus $2$ has a presentation
$G_2 = \langle a,b,c,d; [a,b][c,d] \rangle$. The corresponding Cayley graph induces
a tessellation of the hyperbolic plane by regular octagons and every vertex
is on  exactly eight such octagons (thus  the graph is self-dual).  We can reformulate Margenstern's undecidability result in group-theoretical language as follows.

\begin{theorem}[Margenstern]
The Domino Problem for the surface group $G_2$ is undecidable.
\end{theorem}


\subsection{Decidability of monadic second-order theory for context-free groups}
Recall that a finitely generated group $G$ has context-free Word Problem if and only if $G$ is virtually free (see Theorem~\ref{t:MS-characterization context-free}).
Now the Cayley graph of a finitely generated virtually free group has a regular tree of finite index. Namely, the subgraph corresponding to the Cayley graph of the free subgroup of finite index.
In this case one can reduce the monadic theory of $G$ to the monadic
theory of the subtree. As a consequence, we have the following result~\cite{MSca}.

\begin{theorem}[Muller-Schupp]
\label{t;MS-dec}
The monadic second-order theory of a Cayley graph of a context-free group is decidable.
\end{theorem}

\begin{corollary}
The Domino Problem for context-free groups is decidable.
\end{corollary}

Kuske and Lohrey~\cite{KL} have recently proved the converse to Theorem~\ref{t;MS-dec}.

\begin{theorem}[Kuske-Lohrey]
If the monadic second-order theory of a Cayley graph of a finitely generated group is decidable, then the group is context-free.
\end{theorem}

In the section on graphs with finitary end structure, we mentioned that all such graphs also have a regular subtree of finite index.
Thus we have the following result from~\cite{MSca}.

\begin{theorem}[Muller-Schupp]
Let $\Gamma $ be the complete transition graph of a pushdown automaton.  Then the monadic second-order theory of $\Gamma$ is decidable.
\end{theorem}

\section{Cellular Automata on Groups}\label{sec:CA}
Cellular automata were introduced by von Neumann~\cite{Burks,neumann} who used them to describe theoretical models of self-reproducing machines.
Although originally defined on the lattice of integer points in Euclidean plane, cellular automata can be defined over any group.

Let $G$ be a group, called the \emph{universe}, and  let $\Sigma$ be a finite alphabet
called the set of \emph{states} (or \emph{colors}). Denote by $\Sigma^G$ the set of all maps $\alpha \colon G \to \Sigma$,
called \emph{configurations}. When equipped with the prodiscrete topology, that is,
the product topology obtained by taking the
discrete topology on each factor $\Sigma$ of $\Sigma^G = \prod_{g \in G} \Sigma$, the configuration space becomes a compact, Hausdorff,
totally disconnected  topological space.
There is a natural continuous left action of $G$ on $\Sigma^G$ given by
$g\alpha(h) = \alpha(g^{-1}h)$ for all $g,h \in G$ and $\alpha \in \Sigma^G$.
This action is called the $G$-\emph{shift} on $\Sigma^G$.

\begin{definition}
A map $\CC \colon \Sigma^G \to \Sigma^G$ is called a \emph{cellular automaton}
provided there exists a finite subset $M \subset G$ and a map
$\mu \colon \Sigma^{M} \to \Sigma$ such that
\begin{equation}
\label{e:ca2}
\CC(\alpha)(g) = \mu((g^{-1} \alpha)\vert_{M})
\end{equation}
for all $\alpha  \in \Sigma^G$ and $g \in G$, where $(\cdot)\vert_M$ denotes the restriction to $M$. The subset $M \subset G$ is called a
\emph{local neighborhood} (or \emph{memory set}) for $\CC$ and $\mu$ is the associated \emph{local defining map}.
\end{definition}

\begin{example}[The majority action on $\Z$]
Consider $G = \Z$, $\Sigma = \{0,1\}$, $M = \{-1,0,1\}$ and $\mu\colon \Sigma^M \equiv \Sigma^3 \to \Sigma$ defined by
$$
\mu(a_{-1},a_0,a_1) =
\begin{cases} 1 & \mbox{ if } a_{-1}+a_0+a_1 \geq 2\\
0 & \mbox{ otherwise.}
\end{cases}
$$
Figure~\ref{fig:maj} illustrates the behavior of the corresponding cellular automaton $\CC \colon \Sigma^\Z \to \Sigma^\Z$.
Note that $\CC$ is surjective but not injective.
\begin{figure}[h!]
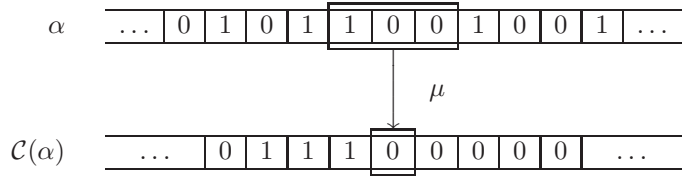

$$
\begin{array}{rcccccccccccccc}
\hhline{~~-----|===|-----}
\alpha&&\dots & \multicolumn{1}{|c}{0} & \multicolumn{1}{|c}{1} & \multicolumn{1}{|c}{0} & \multicolumn{1}{|c|}{1} & \multicolumn{1}{c}{1} & \multicolumn{1}{|c|}{0} & \multicolumn{1}{c|}{0} & \multicolumn{1}{c|}{1} & \multicolumn{1}{c|}{0} & \multicolumn{1}{c|}{0} & \multicolumn{1}{c|}{1} & \dots\\
\cline{3-7}\cline{11-15}
\hhline{~~~~~~~|===|~~~~~}
&&&&&&&&  {\Bigg \downarrow} &\mu&&&&&\\
\hhline{~~------|=|------}
\CC(\alpha) & &\multicolumn{2}{c}{\dots} & \multicolumn{1}{|c}{0} & \multicolumn{1}{|c}{1} & \multicolumn{1}{|c}{1} & \multicolumn{1}{|c|}{1} & \multicolumn{1}{c|}{0} & \multicolumn{1}{c|}{0} & \multicolumn{1}{c|}{0} & \multicolumn{1}{c|}{0} & \multicolumn{1}{c|}{0} & \multicolumn{2}{c}{\dots}\\
\cline{3-8}\cline{10-15}
\hhline{~~~~~~~~|=|~~~~~~}
\end{array}
$$
\caption{The cellular automaton defined by the majority action on $\Z$.}\label{fig:maj}
\end{figure}
\end{example}

\begin{example}[Hedlund's marker \cite{hedlund}]
Let $G = \Z$, $\Sigma = \{0,1\}$, $M = \{-1,0,1,2\}$ and $\mu \colon \Sigma^M \equiv \Sigma^4 \to \Sigma$ defined by
$$
\mu(a_{-1},a_0,a_1,a_2) =
\begin{cases} 1 - a_0 & \mbox{ if } (a_{-1},a_1,a_2) = (0,1,0)\\
a_0 & \mbox{ otherwise.}
\end{cases}
$$
The corresponding
cellular automaton $\CC \colon \Sigma^\Z \to \Sigma^\Z$ is a nontrivial involution of $\Sigma^\Z$.  It is described in Figure~\ref{fig:hed}.
\begin{figure}[h!]
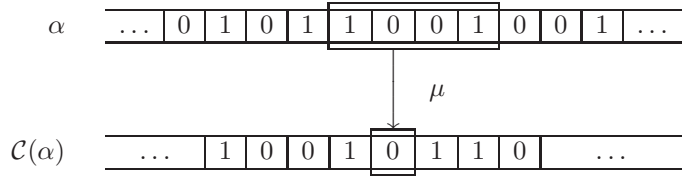

$$
\begin{array}{rcccccccccccccc}
\hhline{~~-----|====|----}
\alpha&&\dots & \multicolumn{1}{|c}{0} & \multicolumn{1}{|c}{1} & \multicolumn{1}{|c}{0} &
\multicolumn{1}{|c|}{1} & \multicolumn{1}{c}{1} & \multicolumn{1}{|c|}{0} & \multicolumn{1}{c|}{0} &
\multicolumn{1}{c|}{1} & \multicolumn{1}{c|}{0} & \multicolumn{1}{c|}{0} & \multicolumn{1}{c|}{1} &
\dots\\
\cline{3-7}\cline{11-15}
\hhline{~~~~~~~|====|~~~~}
&&&&&&&&  {\Bigg \downarrow} &\mu&&&&&\\
\hhline{~~------|=|------}
\CC(\alpha) & &\multicolumn{2}{c}{\dots} & \multicolumn{1}{|c}{1} & \multicolumn{1}{|c}{0} &
\multicolumn{1}{|c}{0} & \multicolumn{1}{|c|}{1} & \multicolumn{1}{c|}{0} & \multicolumn{1}{c|}{1} &
\multicolumn{1}{c|}{1} & \multicolumn{1}{c|}{0} & \multicolumn{3}{c}{\dots}\\
\cline{3-8}\cline{10-15}
\hhline{~~~~~~~~|=|~~~~~~}
\end{array}
$$
\caption{The cellular automaton defined by the Hedlund marker.}\label{fig:hed}
\end{figure}
\end{example}

\begin{example}[Conway's Game of Life]
Let $G = \Z^2$, $\Sigma = \{0,1\}$, $M = \{-1,0,1\}^2 \subset \Z^2$ and $\mu\colon \Sigma^M \to \Sigma$ given by
\begin{equation}
\label{GoL}
\mu(y) =\left\{
\begin{array}{ll}
1 &\mbox{if}\left\{ \begin{array}{l} \mbox{ } \displaystyle\sum_{m
\in M} y(m) = 3 \\
\mbox{or } \displaystyle\sum_{m \in M} y(m) = 4 \mbox{ and } y((0,0)) = 1
\end{array} \right. \\
0 &\mbox{otherwise}
\end{array} \right.
\end{equation}
for all $y \in \Sigma^M$.
The corresponding cellular automaton $\CC \colon \Sigma^{\Z^2} \to \Sigma^{\Z^2}$ describes the \emph{Game of Life} due to Conway. One thinks of an element $g$ of $G = \Z^2$ as a
``cell'' and the set $gM$ (we use multiplicative notation) as the set consisting of
its eight neighboring cells, namely the North, North-East, East, South-East, South,
South-West, West and North-West cells. We interpret  state $0$  as corresponding  to the \emph{absence} of life while state $1$
corresponds to the  \emph{presence} of life.  We  thus refer to  cells in state $0$ as
\emph{dead} cells and to cells  in state $1$ as  \emph{live} cells.
Finally, if $\alpha  \in \Sigma^{\Z^2}$ is a configuration at time $t$, then $\CC(\alpha)$ represents the evolution of the configuration  at time $t+1$.
Then the cellular automaton in \eqref{GoL} evolves as  follows.
\begin{itemize}
\item {\it Birth:} a cell that is dead at time $t$ becomes alive at time $t+1$ if and only if  three of its neighbors are alive at time $t$.
\item{\it Survival:}  a cell that is alive at time $t$ will remain
alive at time $t+1$ if and only if it has  exactly two or three
live neighbors at time $t$.
\item{\it Death by loneliness:} a live cell that has at most one live
neighbor at time $t$ will be dead at time $t+1$.
\item{\it Death by overcrowding:} a cell that is alive at time $t$
and has four or more live neighbors at time $t$, will be
dead at time $t+1$.
\end{itemize}
Figure~\ref{fig:gol} illustrates all these cases.
Note that $\CC$ is not injective and it can be shown that $\CC$ is not surjective either.
\end{example}
\begin{figure}[h!]
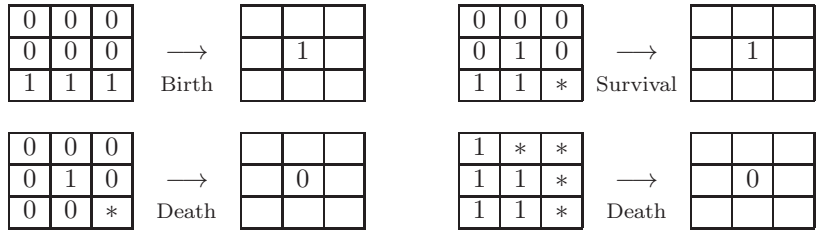

$$
\begin{array}{ccccccccccccccc}
\cline{1-3}\cline{5-7}\cline{9-11}\cline{13-15}
\multicolumn{1}{|c}{0} & \multicolumn{1}{|c}{0} & \multicolumn{1}{|c|}{0} & \phantom{\mbox{\footnotesize Survival}}&\multicolumn{1}{|c}{\phantom{1}} & \multicolumn{1}{|c}{} & \multicolumn{1}{|c|}{\phantom{1}}
&\phantom{lungo}&
\multicolumn{1}{|c}{0} & \multicolumn{1}{|c}{0} & \multicolumn{1}{|c|}{0} & &\multicolumn{1}{|c}{} & \multicolumn{1}{|c}{} & \multicolumn{1}{|c|}{}
\\
\cline{1-3}\cline{5-7}\cline{9-11}\cline{13-15}
\multicolumn{1}{|c}{0} & \multicolumn{1}{|c}{0} & \multicolumn{1}{|c|}{0} & \longrightarrow &\multicolumn{1}{|c}{} & \multicolumn{1}{|c}{1} & \multicolumn{1}{|c|}{}
&&
\multicolumn{1}{|c}{0} & \multicolumn{1}{|c}{1} & \multicolumn{1}{|c|}{0} & \longrightarrow &\multicolumn{1}{|c}{} & \multicolumn{1}{|c}{1} & \multicolumn{1}{|c|}{}
\\
\cline{1-3}\cline{5-7}\cline{9-11}\cline{13-15}
\multicolumn{1}{|c}{1} & \multicolumn{1}{|c}{1} & \multicolumn{1}{|c|}{1} & \mbox{\footnotesize Birth} &\multicolumn{1}{|c}{} & \multicolumn{1}{|c}{} & \multicolumn{1}{|c|}{}
&&
\multicolumn{1}{|c}{1} & \multicolumn{1}{|c}{1} & \multicolumn{1}{|c|}{*} & \mbox{\footnotesize Survival} &\multicolumn{1}{|c}{} & \multicolumn{1}{|c}{} & \multicolumn{1}{|c|}{}
\\
\cline{1-3}\cline{5-7}\cline{9-11}\cline{13-15}
\\
\cline{1-3}\cline{5-7}\cline{9-11}\cline{13-15}
\multicolumn{1}{|c}{0} & \multicolumn{1}{|c}{0} & \multicolumn{1}{|c|}{0} & &\multicolumn{1}{|c}{\phantom{1}} & \multicolumn{1}{|c}{} & \multicolumn{1}{|c|}{\phantom{1}}
&&
\multicolumn{1}{|c}{1} & \multicolumn{1}{|c}{*} & \multicolumn{1}{|c|}{*} & &\multicolumn{1}{|c}{\phantom{1}} & \multicolumn{1}{|c}{} & \multicolumn{1}{|c|}{\phantom{1}}
\\
\cline{1-3}\cline{5-7}\cline{9-11}\cline{13-15}
\multicolumn{1}{|c}{0} & \multicolumn{1}{|c}{1} & \multicolumn{1}{|c|}{0} & \longrightarrow &\multicolumn{1}{|c}{} & \multicolumn{1}{|c}{0} & \multicolumn{1}{|c|}{}
&&
\multicolumn{1}{|c}{1} & \multicolumn{1}{|c}{1} & \multicolumn{1}{|c|}{*} & \longrightarrow &\multicolumn{1}{|c}{} & \multicolumn{1}{|c}{0} & \multicolumn{1}{|c|}{}
\\
\cline{1-3}\cline{5-7}\cline{9-11}\cline{13-15}
\multicolumn{1}{|c}{0} & \multicolumn{1}{|c}{0} & \multicolumn{1}{|c|}{*} & \mbox{\footnotesize Death} &\multicolumn{1}{|c}{} & \multicolumn{1}{|c}{} & \multicolumn{1}{|c|}{}
&&
\multicolumn{1}{|c}{1} & \multicolumn{1}{|c}{1} & \multicolumn{1}{|c|}{*} & \mbox{\footnotesize Death} &\multicolumn{1}{|c}{} & \multicolumn{1}{|c}{} & \multicolumn{1}{|c|}{}
\\
\cline{1-3}\cline{5-7}\cline{9-11}\cline{13-15}
\end{array}
$$
\caption{The evolution of a cell in the Game of Life. The symbol $*$ represents any symbol in \{0,1\}.}\label{fig:gol}
\end{figure}

It easily follows from the definition that every cellular automaton
$\CC \colon \Sigma^G \to \Sigma^G$ is $G$-equivariant, i.e.,
$\CC(g\alpha) = g\CC(\alpha)$ for all $g \in G$ and $\alpha \in \Sigma^G$, and
is continuous with respect to the prodiscrete topology on $\Sigma^G$.
The Curtis-Hedlund Theorem (\cite{hedlund},~\cite[Theorem 1.8.1]{CSC}) shows that the converse is also true.

It immediately follows from topological considerations and the Curtis-Hed\-lund Theorem that a bijective cellular automaton
$\CC \colon \Sigma^G \to \Sigma^G$ is  \emph{invertible}, in the sense that the inverse map $\CC^{-1} \colon \Sigma^G \to \Sigma^G$
is also a cellular automaton.

A map $\CC \colon \Sigma^G \to \Sigma^G$ is called \emph{pre-injective} (a terminology
due to Gromov~\cite{gromov2}) if whenever two configurations $\alpha, \beta \in \Sigma^G$ differ
at only finitely many points  (that is, the set $\{g \in G: \alpha(g) \neq \beta(g)\}$ is finite)
and  $\CC(\alpha) = \CC(\beta)$, then $\alpha = \beta$. Clearly pre-injectivity is a weaker form of
injectivity.

 Moore and  Myhill proved that for $G = \Z^d$, $d \geq 1$, a cellular automaton
$\CC \colon \Sigma^G \to \Sigma^G$ is surjective if and only if it is pre-injective.
Necessity is due to Moore and sufficiency is due to Myhill.
This result is often called the \emph{Garden of Eden Theorem}.
Regarding a cellular automaton as a dynamical system with discrete time, a configuration which is not in the image of the
cellular automaton can only appear as an initial configuration, that is, at time $t = 0$.
This motivates the biblical terminology.
In 1993  Mach\`\i \ and  Mignosi~\cite{machi} extended the Garden of Eden theorem to finitely generated groups of
subexponential growth (cf. the end of Section \ref{s:PCG}) and, finally,  Ceccherini-Silberstein, Mach\`\i \ and  Scarabotti~\cite{CMS} (see also Gromov~\cite{gromov}) further extended it to
all amenable groups.

Recall that a group $G$ is said to be \emph{amenable},
a notion going back to von Neumann~\cite{neumanna}, if there exists a \emph{left-invariant
finitely additive probability measure} on $G$, that is, a map $m \colon {\mathcal P}(G) \to
[0,1]$ such that $m(G) = 1$, $m(A \cup B) = m(A) + m(B) - m(A \cap B)$ and $m(gA) = m(A)$,
for all $A,B \in {\mathcal P}(G)$ and $g \in G$. Finite groups, abelian groups, and more
generally solvable groups, groups of subexponential growth are amenable groups. On the
other hand the free nonabelian groups are non-amenable.

Based on examples due to Muller~\cite{muller}, in~\cite{CMS}
it is shown that if the group $G$ contains a free nonabelian group (and is therefore
non-amenable, since the class of amenable groups is closed under the operation of taking subgroups), then there exist examples of pre-injective (resp. surjective) cellular automata on $G$ which are not surjective (resp. not pre-injective).  Finally,  Bartholdi in 2010~\cite{bartholdi} (see also Theorem 5.12.1 in~\cite{CSC}) proved
the converse to the amenable version of Moore's theorem in~\cite{CMS}, namely that  if every surjective
cellular automaton $\CC  \colon \Sigma^G \to \Sigma^G$ is pre-injective, then the  group $G$ is amenable.
This yields a new characterization of amenability in  terms of cellular automata.

Following   Gottschalk~\cite{gottschalk}, we say that a group $G$ is \emph{surjunctive}
provided that for every  finite set $\Sigma$ every injective cellular automaton $\CC \colon \Sigma^G \to \Sigma^G$ is surjective
(and therefore bijective). It is an open problem to determine whether all groups are surjunctive or not.
 Lawton~\cite{lawton} (see also~\cite[Theorem 3.3.1]{CSC}) showed that all residually finite groups (in particular,  all virtually free groups) are surjunctive. Recall that a group
 is residually finite provided that the intersection of all its finite index subgroups reduces to the trivial group (see, e.g.~\cite[Chapter 2]{CSC}.
It immediately follows from the Garden of Eden Theorem for amenable groups that all
amenable groups are surjunctive.  Gromov~\cite{gromov2} and  Weiss~\cite{weiss}
(see also~\cite[Theorem 7.8.1]{CSC}) showed that all \emph{sofic} groups are surjunctive.
For the definition of soficity we refer to~\cite[Chapter 7]{CSC}. We only mention that the class of sofic groups contains  all residually finite groups and all amenable groups,
and that it is not known if there are any non-sofic groups.

One is often interested in determining whether a cellular automaton is injective or surjective.
In particular, the following question naturally arises: is it decidable,  given a
finite subset $M \subset G$ and a map $\mu \colon \Sigma^{M} \to \Sigma$, if the associated cellular automaton
$\CC \colon \Sigma^G \to \Sigma^G$ defined in \eqref{e:ca2}  is surjective or not?  Amoroso and  Patt~\cite{amoroso}
proved in 1972 that if $G = \Z$ the above Surjectivity Problem is decidable.
On the other hand,  Kari~\cite{kari1,kari2, kari3} proved that the similar problem for cellular automata with finite alphabet over
$\Z^d$, $d \geq 2$, is undecidable. His  proof is based on Berger's undecidability result for the Domino Problem (see Section~\ref{sub-domino}).
It follows  from the decidability of the monadic second-order theory of Cayley graphs of context-free groups (cf. Theorem~\ref{t;MS-dec}) that the Surjectivity Problem
for cellular automata defined over finitely generated virtually-free groups is decidable.

Indeed,  that the cellular automaton is surjective is expressed by
saying that for every disjoint cover $(C_1,C_2, \dots, C_n)$ of $G$ (where $C_i$ represents the points currently in state $a_i \in \Sigma$) there is a disjoint cover $(P_1,P_2, \dots, P_n)$ (the assignment of predecessor states) such that for
every vertex $v$, one has $v \in C_i$ if and only if the points in the neighborhood of $v$ are in the correct $P$-sets for the local defining map $\mu$ to assign
state $a_i$ to $v$. This fact is easily  expressible as  a monadic second-order sentence.  It similarly follows that
the Injectivity and Bijectivity Problems are decidable for cellular automata on finitely generated virtually-free groups.

The following natural question is open.

\begin{question}
Are there any finitely generated groups which are not virtually free but for which the Surjectivity, Injectivity or Bijectivity Problems are decidable?
\end{question}


\section{Finite automata on infinite inputs and infinite games of perfect information}\label{sec:games}

\subsection{B\"uchi acceptance and regular languages in $\Sigma^\N$}
As mentioned in the Introduction, monadic sentences are too complicated to deal with directly.  The theorems of B\"uchi (cf. Theorem~\ref{t:buchi}) and
of Rabin (cf. Theorem~\ref{t:rabin}) are  proved by developing a theory of finite automata working  on infinite words and infinite trees respectively.
Let  $w = w_0w_1 \dots w_i w_{i+1}\cdots \in \Sigma^\N$ be an infinite word. (All our infinite words are infinite to the right.)
In B\"uchi's original paper, a nondeterministic finite automaton working on a word $w \in \Sigma^{\mathbb{N}}$
is a tuple $\A = (Q,\Sigma, q_0, \delta , F)$ exactly as in the case
of automata on finite words (cf. Section~\ref{sec:fa}).
Thus, as usual, $Q$ is a finite set of states, $\Sigma$ is a finite alphabet, $q_0 \in Q$ is the initial state, $\delta : Q \times \Sigma \to \mathcal{P}(Q)$ is the transition function and $F \subseteq Q$ is a set of final states.
A \emph{run} of $\A$ on $w$ is a map $\rho: \mathbb N \to Q$ such that $\rho(0) = q_0$ and $\rho (i+1) \in \delta(\rho(i), w_i) $
for all $i \in \N$.  We must now define when the  automaton $\A$ \emph{accepts} $w \in \Sigma^\N$, which we write as $\A \vdash  w$.
The definition of \emph{B\"uchi acceptance} is that $\A \vdash w$ if there exists a run $\rho$ of $\A$ on $w$ such that some state
from $F$ occurs infinitely often.  As in the case of finite words, we call the set
$$
L(\A) = \{w \in \Sigma^\N : \A \vdash w\} \subset \Sigma^\N
$$
the \emph{language accepted} by $\A$.  A  subset $L \subseteq \Sigma^\N$ is a \emph{regular language} if it is the language
accepted by some finite automaton.

\begin{example}
\label{example-finite-b}
Let $\Sigma = \{a,b\}$. We describe a finite automaton which accepts those infinite words $w \in \Sigma^\N$ containing $b$ only a
finite number of times.
Let $\A = (Q,\Sigma, q_0, \delta , F)$ be a finite automaton
where $Q =\{q_b,q_a,q_c,q_r\}$, $q_0 = q_b$, $F = \{q_c\}$ and
\[
\begin{split}
\delta(q_b,a) &= q_a, \ \delta(q_b,b) = q_b, \\
\delta(q_a,a) &= \{q_a, q_c\}, \  \delta(q_a,b) = q_b,\\
\delta(q_c,a) &= q_c,  \  \delta(q_c,b) = q_r, \\
\delta(q_r,a) &= \delta(q_r,b) = q_r.
\end{split}
\]
The automaton is illustrated in Figure~\ref{AUTfinite-b} and it works in the following way.  When in state $q_b$,
the automaton goes to $q_a$ on reading $a$ and remains in $q_b$ on reading $b$.
On reading a $b$ in the  state $q_a$ it goes to state $q_b$.  On reading an $a$ in $q_a$
the automaton can either remain in state $q_a$ or ``guess'' that it will see no $b$'s in the future by going to the ``check''
state $q_c$.  In $q_c$ the automaton remains in $q_c$ as long as it sees only $a$'s
but goes to the reject state $q_r$ if it ever reads a $b$. Once in $q_r$ the automaton always remains in $q_r$ on either input.
Since $F = \{q_c\}$, in any accepting run the automaton must  have guessed at some time that no more $b$'s occur and must then
always remain in $q_c$, thus seeing no more $b$'s.
And for any $w \in \Sigma^\N$ containing only finitely many $b$'s there is an accepting run.
\end{example}

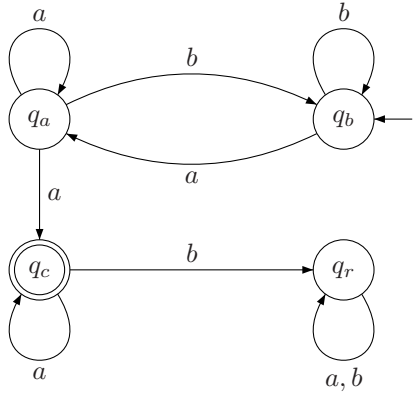
\begin{figure}[h!]
\begin{center}
\begin{picture}(0,50)(0,-30)
\node(1)(-20,0){$q_a$}
\node[Nmarks=i, iangle=0](2)(20,0){$q_b$}
\node[Nmarks=r](3)(-20,-20){$q_c$}
\node(4)(20,-20){$q_r$}
\drawloop(1){$a$}
\drawloop(2){$b$}
\drawedge[curvedepth=0](1,3){$a$}
\drawedge[curvedepth=6](2,1){$a$}
\drawedge[curvedepth=6](1,2){$b$}
\drawedge[curvedepth=0](3,4){$b$}
\drawloop[loopangle=270](3){$a$}
\drawloop[loopangle=270](4){$a,b$}
\end{picture}
\end{center}
\caption{The automaton accepting the infinite words $w \in \{a,b\}^\N$ containing only a
finite number of $b$'s.}\label{AUTfinite-b}
\end{figure}

The overall goal is to associate with each monadic sentence $\phi$ of $S1S$ a finite automaton $\A_{\phi}$ such that $\phi$ is
true if and only if $L(\A_{\phi}) \ne \varnothing$. In order to do this we need to establish the closure of
regular languages under the three operations of union, complementation, and projection.  These operations correspond to the logical
connectives $\lor, \lnot$, and $\exists$ respectively.   If $\Sigma$ and $\overline{\Sigma}$ are alphabets  and $\pi: \Sigma \to \overline{\Sigma}$ is a map
then $\pi$ induces a function $\widehat{\pi}: \Sigma^\N \to \overline{\Sigma}^\N$ by letter-by-letter substitution. If $L \subset \Sigma^\N$ is
a language, then $\widehat{\pi}(L) \subset \overline{\Sigma}^\N$ is the \emph{projection} of $L$ under $\pi$ and we need to know that if $L$ is a
regular language over $\Sigma$ then  $\widehat{\pi}(L)$ is a regular language over $\overline{\Sigma}$.

The closure of regular languages with respect to the operation of union is easy to establish in essentially any model of finite automata.
Also,   projection is ``easy'' for nondeterministic automata, even on infinite words,  and ``hard'' for deterministic automata.
Suppose that   $\pi: \Sigma \to \overline{\Sigma}$ is a function inducing the projection
$\widehat{\pi}: \Sigma^{\omega} \to \overline{\Sigma}^{\omega}$ and that  $\A = (Q, \Sigma, \delta, q_0, F)$ is a nondeterministic automaton with alphabet $\Sigma$.
To accept the projection of the language accepted by $\A$, we define a nondeterministic automaton   $\widehat{\A}$ which, on reading a
letter $\overline{a} \in \overline{\Sigma}$  can make any transition that $\A$ can make on any preimage of $\overline{a}$.  Formally,
\[
\widehat{\A} = \left(\mathcal{P}(Q), \overline{\Sigma}, \widehat{\delta}, \{q_0\}, \mathcal{P}(F) \right) \mbox{ \ where }
 \widehat{\delta}(S,\overline{a}) = \bigcup_{q \in S}\bigcup_{a \in \pi^{-1}(\overline{a})} \delta(q,a).
\]

Note that even if we started  with a deterministic automaton  $\A$,  the automaton $\widehat{\A}$ is nondeterministic.

\subsection{Muller acceptance}\label{Muller}
In general,  the closure of regular languages with respect to complementation is  ``hard'' for nondeterministic automata, and regular languages
in $\Sigma^\N$  recognized by using B\"uchi acceptance  generally require using a  nondeterministic automaton.
The power of automata on infinite inputs is very  sensitive to the acceptance condition used.
Muller~\cite{muller1} introduced the concept of \emph{Muller acceptance},
which is the most general type of acceptance commonly used.

\begin{definition}
A nondeterministic  \emph{Muller automaton} is a tuple $\A = (Q, \Sigma, \delta, q_0, \mathcal{F})$ where $Q, \Sigma, \delta$ and $q_0$ are
exactly as for a nondeterministic
finite automaton but  $\mathcal{F} \subset \mathcal{P}(Q)$.
Let $w \in \Sigma^\N$ be a word. If $\rho$ is a run of $\A$ on $w$ then we denote by $\Inf(\rho)$ the set of states occurring infinitely
often in $\rho$. Then $\A$ \emph{accepts}  $w$ if there exists a run $\rho$ of $\A$ on $w$ such that $\Inf(\rho) \in \mathcal{F}$.
\end{definition}

\begin{remark}  If we compare B\"uchi acceptance with Muller acceptance, we have that
the set of final states $F \subset S$ is now replaced by the ``accepting'' family
$\mathcal{F}$. Moreover $w \in \Sigma^\N$ is B\"uchi-accepted if $\Inf(\rho) \cap F \ne \varnothing$, while it is Muller-accepted if $\Inf(\rho) \in \mathcal{F}$.
\end{remark}

The following result was conjectured by Muller and then proved by McNaughton~\cite{mcnaughton}.

\begin{theorem}[McNaughton]
\label{t:McNaughton} For any nondeterministic automaton on infinite words using Muller acceptance, there is an equivalent deterministic automaton using Muller acceptance.
\end{theorem}

While  the negation of a B\"uchi acceptance condition is not a B\"uchi condition,
the negation of a Muller acceptance condition $\mathcal{F}$ is again a condition of the same type, namely the Muller condition defined by the accepting family $\mathcal{P}(Q) \setminus \mathcal{F}$.
For a deterministic automaton  $\A = (Q, \Sigma, \delta, q_0, \mathcal{F})$ using Muller acceptance to accept the language $L(\A)$ we have
\[
\Sigma^* \diagdown L(\A) = L(\lnot \A) \mbox { where }  \neg \A =   (Q, \Sigma, \delta, q_0, \mathcal{P}(Q) \setminus \mathcal{F}).
\]
In short, $\lnot \A$ is obtained from $\A$ by simply complementing the accepting family.

McNaughton's theorem thus proves that the class of regular languages of infinite words is closed under complementation.  Proving McNaughton's theorem
 from scratch is not easy and it is an  accident  that determinizing  the nondeterministic automaton of
Example~\ref{example-finite-b} is easy.

\begin{example}
\rm{Let $\Sigma = \{a,b\}$. We now present a deterministic finite automaton $\A$ using Muller acceptance which accepts exactly
those words $w \in \Sigma^\N$ containing $b$ infinitely often.
Let $\A = (Q,\Sigma,  \delta, q_0, \mathcal{F})$ be the finite automaton in which
$Q = \{q_a,q_b\}$, $\Sigma = \{a,b\}$, $q_0 = q_a$, $\mathcal{F} = \left\{\{q_b\}, \{q_a, q_b\}\right\}$ and
\[
\begin{split}
\delta(q_a,a) & = q_a, \ \delta(q_a, b) = q_b,\\
\delta(q_b,a) & = q_a, \ \delta(q_b,b) = q_b.
\end{split}
\]

The automaton is illustrated in Figure~\ref{AUTinfinite-b} and it works in the following way.
The states $q_a$ and $q_b$ record which letter has just been read.
On a word $w \in \Sigma^\N$ containing $b$ infinitely often the set of states
occurring infinitely often must be exactly  $\{q_a,q_b\}$ in the case that both letters occur infinitely often or $\{q_b\}$ in the
case that only $b$ occurs infinitely often.
Since $\mathcal{F}$ consists of these two sets, the automaton accepts exactly the desired words.
Note that  $\neg\A = (Q,\Sigma,  \delta, q_0, \{ \{q_a \} \})$
is a  deterministic automaton using Muller acceptance which accepts exactly those words containing $b$ only finitely many times (cf. Example~\ref{example-finite-b}).}
\end{example}

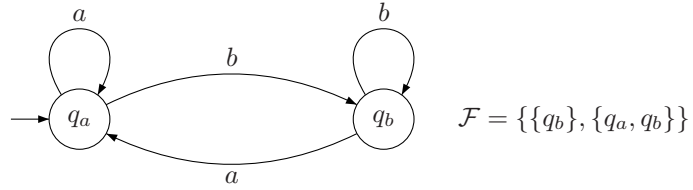
\begin{figure}[h!]
\begin{center}
\begin{picture}(0,20)(0,-5)
\node[Nmarks=i](1)(-20,0){$q_a$}
\node(2)(20,0){$q_b$}
\drawloop(1){$a$}
\drawedge[curvedepth=6](1,2){$b$}
\drawedge[curvedepth=6](2,1){$a$}
\drawloop(2){$b$}
\node[Nframe=n,AHnb=0,Nadjust=w](F)(45,0){$\mathcal{F} = \left\{\{q_b\}, \{q_a, q_b\}\right\}$}
\end{picture}
\end{center}
\caption{The automaton accepting the infinite words $w \in \{a,b\}^\N$ containing an
infinite number of $b$'s.}\label{AUTinfinite-b}
\end{figure}

Deciding the Emptiness Problem for non-deterministic Muller automata is easy.  Given $\A$ with underlying graph $\Gamma$, the language  $L(\A) \ne \varnothing$ if and only if there is a path in $\Gamma$ from the initial state to a cycle containing exactly the states in some set $S \in \mathcal{F}$.

\subsection{Rabin's theory}
We now turn to considering automata on the infinite binary tree $T_2$.
Recall that each vertex of $T_2$ is described by a finite word over the set $\{0,1\}$ of the two possible directions. For a nondeterministic automaton
with alphabet $\Sigma$  working on $T_2$, a possible input $\alpha$ consists of an element
$\alpha \in \Sigma^{T_2}$ which can be described as a copy of $T_2$ with all vertices labelled from
$\Sigma$.  In Rabin's model, a nondeterministic automaton is a 5-tuple $\A = (Q, \Sigma, \delta, q_0, \mathcal{F})$, where $Q$, $\Sigma$, $q_0$ and $\FF$ are defined as in Section~\ref{Muller}. The transition function is of the form
$\delta : Q \times \Sigma \to \mathcal{P}(Q \times Q)$.
The automaton starts at the root
$\varepsilon$ in the initial state $q_0$. A copy of the automaton at a vertex $v$ always sends one copy to the left successor of $v$ and one copy to the right successor of $v$.
\begin{example}\label{defRABINdelta}
If one has
\[
\delta(q_0, a) = \{ (q_1, q_3), (q_2,q_0) \},
\]
then when the automaton is in state $q_0$ reading the letter $a$, it can send one copy to the left in state $q_1$ \emph{and} one copy to the
right in state $q_3$,  \emph{or} it can send one copy to the left in state $q_2$  \emph {and} one copy to the right in state $q_0$.
Note that both ``and'' and ``or'' occur in the description of the transition function. This situation is illustrated in Figure~\ref{RABINdelta}.
\end{example}
\begin{figure}[h!]
\begin{center}
\begin{picture}(0,20)(0,-10)

\node(0)(0,0){$q_0$}
\node(1)(30,10){$q_1$}
\node(3)(30,-10){$q_3$}
\node(2)(-30,10){$q_2$}
\node[Nframe=n,Nw=-1,Nh=-1](00)(20,0){}
\node[Nframe=n,Nw=-1,Nh=-1](0_)(-20,0){}

\drawedge[AHnb=0,ATnb=2,ATangle=30,ATLength=3,ATlength=0](00,0){$a$}
\drawedge[curvedepth=2,dash={1}0](00,1){}
\drawedge[curvedepth=-2](00,3){}
\drawedge[AHnb=2,AHangle=30,AHLength=3,AHlength=0](0,0_){$a$}
\drawedge[curvedepth=-2,dash={1}0](0_,2){}
\drawedge[curvedepth=-6](0_,0){}
\end{picture}
\end{center}
\caption{An instance of the transition function in a Rabin automaton. The drawing convention is that the broken line visualizes the copy to the left, while the continuous line visualizes the copy to the right.}\label{RABINdelta}
\end{figure}
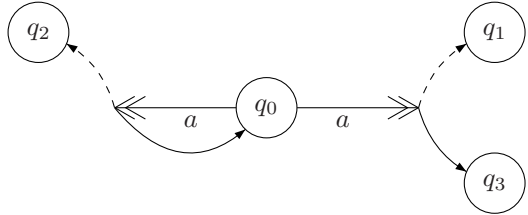

We must now define what it means for  an automaton $\A$  to accept an input $\alpha$, for which we write $\A \vdash \alpha$ as usual.  An infinite path $\pi$
through $T_2$ is a path starting at the origin $\varepsilon$ such that each vertex in $\pi$ has exactly one successor in $\pi$. Note that
$\pi \in \{0,1\}^{\mathbb N}$ and there are thus uncountably many distinct infinite paths through the tree.
A \emph{run} $\rho$ of $\A$ on $\alpha$ is an element in $Q^{T_2}$,
that is, a labelling of $T_2$ by states from $Q$  such that for each vertex $v \in T_2$ we have
$$(\rho(v0), \rho(v1)) \in \delta(\rho(v), \alpha(v)).$$
We will again use Muller acceptance although Rabin used a different but equivalent condition.
So we specify a family $\mathcal{F} \subseteq \mathcal{P}(Q)$.
Given a run $\rho$ and a path $\pi$, we define $\Inf(\rho, \pi)$ to be the set of states in $\rho$ which occur infinitely often along the path $\pi$.
Finally,
\[
\A \vdash \alpha  \mbox{ \ if \ } \exists \rho \forall \pi [\Inf(\rho, \pi) \in \mathcal{F}].\]
In short,  for every path $\pi$ the set of states occurring infinitely often along $\pi$ must be some set $S$ in the accepting family $\mathcal{F}$.
Note that $S$ can vary with different paths.

\begin{example}\label{RABINex} \rm{We extend Example~\ref{example-finite-b}. Suppose again that $\Sigma = \{ a,b \}$ and we now want an automaton which
accepts $\alpha \in \Sigma^{T_2}$ exactly if $\alpha$ contains some infinite path $\pi$ on which $b$ occurs only finitely often.
Let  $\A = (Q,\Sigma,  \delta, q_0, {\mathcal F})$
where $Q = \{q_a,q_b, q_d \}$, $q_0 = q_a$, ${\mathcal F} =\{\{q_a\}, \{q_d\}\}$, and
\[
\begin{split}
\delta(q_a,a) & = \{(q_d, q_a), (q_a,q_d) \},  \delta(q_a,b) = \{(q_d,q_b),(q_b,q_d)\},\\
\delta(q_b,a) & = \{(q_d, q_a) ,(q_a,q_d) \},  \delta(q_b,b) = \{(q_d,q_b),(q_b,q_d)\},\\  \delta(q_d,a) & = \delta(q_d,b) = \{(q_d,q_d)\}.
\end{split}
\]

The automaton is illustrated in Figure~\ref{RABINautomaton} and  works in the following way.  Its overall strategy is to make a nondeterministic choice of the path $\pi$. On reading an $a$ in state $q_a$,
the automaton sends a copy in the ``don't care'' state $q_d$ in one direction and a copy in $q_a$ in the other direction.
On reading a $b$ in  state $q_a$,  the automaton sends a copy in the ``don't care'' state $q_d$ in one direction
and a copy in $q_b$ the other direction.  The state $q_b$ functions similarly.
If the automaton is in the ``don't care'' state $q_d$, it is not on the chosen path and so sends copies in $q_d$ in both directions on
reading either letter. It is easy to see that $\A \vdash \alpha$  if and only if $\alpha$ does contain an infinite path with only finitely many $b$'s.}
\end{example}
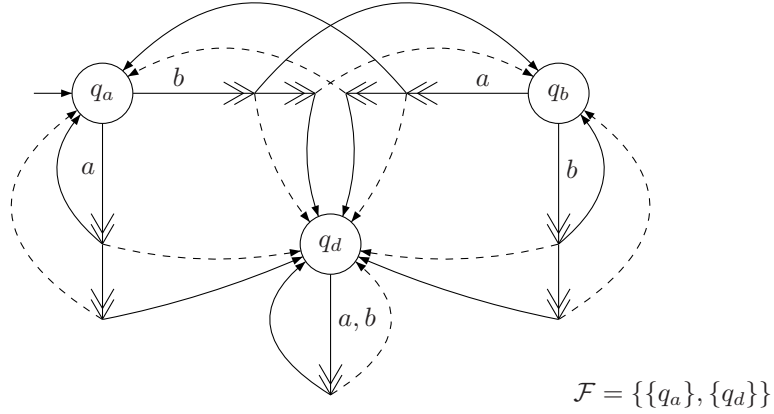
\begin{figure}[h!]
\begin{center}
\begin{picture}(0,50)(0,-40)
\node[Nmarks=i](A)(-30,0){$q_a$}
\node(D)(0,-20){$q_d$}
\node(B)(30,0){$q_b$}
\node[Nframe=n,Nw=-1,Nh=-1](A_)(-30,-20){}
\node[Nframe=n,Nw=-1,Nh=-1](B_)(30,-20){}
\node[Nframe=n,Nw=-1,Nh=-1](A__)(-30,-30){}
\node[Nframe=n,Nw=-1,Nh=-1](B__)(30,-30){}

\drawedge[AHnb=0,ATnb=2,ATangle=30,ATLength=3,ATlength=0](A_,A){$a$}
\drawedge[AHnb=0,ATnb=2,ATangle=30,ATLength=3,ATlength=0](A__,A_){}
\drawedge[curvedepth=12,dash={1}0](A__,A){}
\drawedge[curvedepth=-1](A__,D){}
\drawedge[curvedepth=6](A_,A){}
\drawedge[curvedepth=-2,dash={1}0](A_,D){}

\drawedge[AHnb=2,AHangle=30,AHLength=3,AHlength=0](B,B_){$b$}
\drawedge[AHnb=2,AHangle=30,AHLength=3,AHlength=0](B_,B__){}
\drawedge[curvedepth=-12,dash={1}0](B__,B){}
\drawedge[curvedepth=1](B__,D){}
\drawedge[curvedepth=-6](B_,B){}
\drawedge[curvedepth=2,dash={1}0](B_,D){}

\node[Nframe=n,Nw=-1,Nh=-1](AA)(-10,0){}
\node[Nframe=n,Nw=-1,Nh=-1](AAA)(-2,0){}
\node[Nframe=n,Nw=-1,Nh=-1](BB)(10,0){}
\node[Nframe=n,Nw=-1,Nh=-1](BBB)(2,0){}
\drawedge[AHnb=2,AHangle=30,AHLength=3,AHlength=0](A,AA){$b$}
\drawedge[AHnb=0,ATnb=2,ATangle=30,ATLength=3,ATlength=0](BB,B){$a$}
\drawedge[AHnb=2,AHangle=30,AHLength=3,AHlength=0](AA,AAA){}
\drawedge[AHnb=0,ATnb=2,ATangle=30,ATLength=3,ATlength=0](BBB,BB){}
\drawedge[curvedepth=-2,dash={1}0](AA,D){}
\drawedge[curvedepth=12](AA,B){}
\drawedge[curvedepth=-2](AAA,D){}
\drawedge[curvedepth=6,dash={1}0](AAA,B){}

\drawedge[curvedepth=2,dash={1}0](BB,D){}
\drawedge[curvedepth=-12](BB,A){}
\drawedge[curvedepth=2](BBB,D){}
\drawedge[curvedepth=-6,dash={1}0](BBB,A){}

\node[Nframe=n,Nw=-1,Nh=-1](DD)(0,-40){}
\drawedge[AHnb=2,AHangle=30,AHLength=3,AHlength=0](D,DD){$a, b$}
\drawedge[curvedepth=8](DD,D){}
\drawedge[curvedepth=-8,dash={1}0](DD,D){}
\node[Nframe=n,AHnb=0,Nadjust=w](F)(45,-40){$\mathcal{F} = \{\{q_a\}, \{q_d\}\}$}
\end{picture}
\end{center}
\caption{The Rabin automaton defined in Example~\ref{RABINex}.}\label{RABINautomaton}
\end{figure}

\subsection{Infinite games of perfect information}
Deterministic automata on trees are not very powerful and nondeterminism is essential.  Rabin's proof of the closure of regular languages
under complementation was  very difficult.  We now know that the best way to understand automata
on infinite inputs is in terms of infinite games of perfect information, as introduced by Gale and Stewart~\cite{GS}.

Let $\Sigma$ be a finite alphabet, let $\Sigma^\N$ denote the set of all  infinite words over $\Sigma$, and let $\mathcal W$ be a subset
of  $\Sigma^\N$.   We consider the following game between Player I and Player II.
Player I chooses a letter $\sigma_1 \in \Sigma$ and  Player II then chooses a letter
 $\sigma_2 \in \Sigma$.  Continuing indefinitely, at step $n$ Player I chooses a letter  $\sigma_{2n-1} \in \Sigma$ and Player II then chooses a
letter  $\sigma_{2n} \in \Sigma$.  The sequence of choices defines an infinite word $w = \sigma_1\sigma_2 \cdots \sigma_n \cdots \in \Sigma^\N$.
Player I wins the game if $w \in \mathcal W$ and Player II wins otherwise.
The basic question about such games is whether or not one of the players has a winning strategy, that is, a function  $\phi: \Sigma^* \to \Sigma$
such that when a finite word $u$ has already been played, the player then plays  $\phi(u) \in \Sigma$ and always wins.
Using the Axiom of Choice, it is possible to construct winning sets such that neither player has a winning strategy, but this cannot happen
if the set $\mathcal W$ is not ``too complicated''.

\begin{example} \rm{We  show that if the set $\mathcal{W}$ is countable and $|\Sigma| \ge 2$ then the second player has a winning strategy by
applying Cantor's diagonal argument.  Let $w_i = w_{i,1} w_{i,2}\cdots w_{i,n}\cdots$ be the $i$-th  word in $\mathcal{W}$.
On his turn, play $2k$, Player II simply plays a letter different from $w_{2k,2k}$.  Thus the word resulting from the set of plays is not in $\mathcal{W}$.
Note that this simple example shows that strategies need not at all be effectively computable. Since the $w_i$ are infinite words, even a single
such word need not be computable since $\mathcal{W}$ is an arbitrary countable subset of $\Sigma^\N$.}
\end{example}

The set $\Sigma^\N$ becomes a complete metric space by defining $\dist(v,w) = 2^{-j}$ for all
$v = v_1v_2 \cdots$ and $w = w_1w_2\cdots$, where $j$ is the least index such that $w_j \ne v_j$.
An important theorem of Martin~\cite{martin1, martin2} (see also~\cite[Sect. 20]{kechris} and~\cite[Sect. 6F]{moschovakis}) shows that if
the set $\mathcal W$ is a Borel set then one of the two players must have a winning strategy.
In applying infinite games to automata, one needs only consider winning conditions which are $F_{\delta, \sigma}$
and that  such games are determined was proven by  Davis~\cite{Davis} before Martin's general result.
Given an automaton $\A$ and an input $\alpha$, one defines the \emph{acceptance game} $\mathcal G (\A,t)$ for $\A$ on the input $\alpha$. The first player wins
if $\A$ accepts $\alpha$ while the second player wins if $\A$ rejects.

Muller and Schupp~\cite{MS3} defined \emph{alternating tree automata} as a generalization of nondeterministic automata working on trees.
In this model, the transition function has the form $\delta: Q \times \Sigma \to \mathcal{L}(Q \times \{0,1\}$, where $\mathcal{L}(Q \times \{0,1\}$ is the free distributive lattice generated by all possible pairs (state, direction).

\begin{example}We consider again a nondeterministic automaton in which
\[ \delta(q_0, a) = \{ (q_1, q_3), (q_2,q_0) \} \]
as in Example~\ref{defRABINdelta}.
In the lattice notation we can write this as
\[ \delta(q_0, a) = [( q_1, 0) \land (q_3,1)] \lor [(q_2,0) \land (q_0, 1)]. \]
Here the symbol $\lor$ stands for nondeterministic choice  and $\land$ means ``do both things''.

We  \emph{dualize} a transition function of an alternating tree automaton
by interchanging  $\land$ and $\lor$ as usual. For the example above we have:
\[
\widetilde{\delta}(q_0, a) = [( q_1, 0) \lor (q_3,1)] \land  [(q_2,0) \lor (q_0, 1)].
\]
Converting this expression to disjunctive normal form we have:
\[
\widetilde{\delta}(q_0, a) = [( q_1, 0) \land (q_2,0)] \lor  [( q_1, 0) \land (q_0,1)] \lor  [(q_3,1) \land (q_2,0)]
   \lor  [(q_3,1) \land (q_0, 1)].
\]
We interpret this as saying that when the automaton is in state $q_0$ reading the letter $a$ it has a choice of sending one copy to the
left in $ q_1$  \emph{and} another copy to the left in  $q_2$, \emph{or}  sending a copy to the left
in $ q_1$ \emph{and} a copy to the right in $q_0$, \emph{or} a copy to the right in  $q_3$ \emph{and} a copy to the left in $q_2$,
\emph{or}, finally,  a copy to the right in $q_3$ \emph{and} another copy to the right in  $q_0$.
This is not a nondeterministic automaton but it is a perfectly good alternating automaton. Note that the automaton can send multiple copies
in the same direction and is not required to send copies in all directions.  It must, of course, send at least one copy in some direction.
\end{example}

We now have a framework general enough to always be able to dualize.

\begin{definition}  Let $\A = (Q,\Sigma, \delta, q_0, \mathcal{F})$ be an alternating automaton on the rooted infinite binary  tree.
Then the \emph{dual automaton} of $\A$ is
\[
\widetilde{\A} = (Q,\Sigma, \widetilde{\delta}, q_0, \overline{\mathcal{F}})
\]
where $\widetilde{\delta}$ is obtained by dualizing the transition function $\delta$, and the accepting family is
$\overline{\mathcal{F}} = \mathcal{P}(Q) \setminus \mathcal{F}$.
\end{definition}

It is clear from the definition that the dual of $\widetilde{\A}$ is just $\A$.
One must carefully define the acceptance game $\GG(\A,t)$ of $\A$ on an input $\alpha$ (for details see~\cite{MS3}).  That this game is determined
follows from Davis' theorem. In the alternating framework, it is easy to check that a winning strategy for the second player
in $\GG(\A,t)$ is a winning strategy for the first player in the acceptance game $\GG(\widetilde{\A},t)$ for the dual automaton.
Thus complementation is easy for alternating automata and the following theorem is a consequence of pure determinacy.

\begin{theorem}[The Complementation Theorem]  If $\A$ is an alternating tree automaton accepting the language $L(\A)$ then the dual automaton
$\widetilde{\A}$ accepts the complementary language $\lnot L(\A)$.
\end{theorem}

Of course, something must be hard for alternating automata and it is the operation of projection.  The argument for nondeterministic automaton
fails completely because there may be multiple copies of the automaton at the same vertex of the tree. So we must prove that given an
alternating automaton, there is a nondeterministic automaton  accepting the same language.
Gurevich and Harrington~\cite{GH} made a fundamental contribution to understanding
automata on infinite inputs by showing that a winning strategy in the acceptance game for a nondeterministic automaton depends only on a
finite amount of memory called the \emph{later appearance record}.  This is called the \emph{Forgetful Determinacy Theorem} (see~\cite{GH, zeitman}).
Muller and Schupp~\cite{MS3} used the later appearance record to prove the  \emph{Simulation Theorem} which states that there is an effective
construction which, given an alternating automaton, produces a nondeterministic automaton accepting the same language.

Given the Complementation and Simulation theorems, most results have short conceptual proofs. As an illustration, we present a proof of McNaughton's theorem.

\begin{proof}[Proof of Theorem \ref{t:McNaughton}]
There is a natural notion of an automaton which is alternating but still deterministic.  Namely, one with no $\lor$'s in its transition function.  The Simulation Theorem
shows that if we start with a deterministic alternating automaton, then the simulating ordinary automaton is a deterministic automaton.
If $\A$ is a nondeterministic automaton on the line (i.e. $|D| =1$), using Muller acceptance, then $\A$ has only $\lor$'s in its transition function.
Then its dual automaton $\widetilde{\A}$ has only $\land$'s in its transition function and therefore is a deterministic alternating
automaton. By the Simulation Theorem we can construct a deterministic automaton $\A'$ on the line which accepts the same language $L'$ as $\widetilde{\A}$.
By the Complementation Theorem, $L'$ is the complement of the language $L$ accepted by $\A$. Since $\A'$ is deterministic we obtain a deterministic
automaton $\lnot \A'$ accepting the complement of $L'$, that is, $L$, by simply complementing the accepting family of $\A$, thus establishing McNaughton's Theorem.
\end{proof}

\section*{Acknowledgments}
We would like to express our deepest gratitude to Tony Martin for a very
interesting and stimulating conversation on infinite games of perfect information and
the Borel determinacy. Also, we wish to warmly thank Maurice
Margenstern for providing fruitful information on the Domino Problem on the hyperbolic plane.
Last but not least, we thank Ilya Kapovich for very helpful suggestions.


\end{document}